\documentclass{amsart}
\usepackage{thmtools}
\usepackage{thm-restate}
\usepackage{float}
\usepackage[usenames,dvipsnames]{xcolor}
\usepackage{mathtools}
\usepackage{amssymb}
\usepackage{hyperref}
\usepackage{graphicx}
\usepackage{pinlabel}
\usepackage[rightcaption]{sidecap}
\usepackage{wrapfig}
\usepackage{caption}
\usepackage{subcaption}
\usepackage[all]{xy}
\usepackage[english]{babel}
\addto\extrasenglish{}
\newtheorem{theorem}{Theorem}[section]
\theoremstyle{definition}
\newtheorem{definition}[theorem]{Definition}
\newtheorem{proposition}[theorem]{Proposition}
\newtheorem{example}[theorem]{Example}

\newtheorem{remark}[theorem]{Remark}

\newtheorem{corollary}[theorem]{Corollary}
\newtheorem{conjecture}[theorem]{Conjecture}
\usepackage{anysize} \marginsize{1in}{1in}{1in}{1in}

\hypersetup{colorlinks=true, linkcolor=NavyBlue, citecolor=NavyBlue}
\graphicspath{ {./Figures/} }
\numberwithin{equation}{section}

\begin{document}

\title{Heegaard diagrams for $5$-manifolds}
\author{Geunyoung Kim}
\address{Department of Mathematics \& Statistics, McMaster University, Hamilton, Ontario, Canada}
\email{kimg68@mcmaster.ca}

\begin{abstract}
We introduce a version of Heegaard diagrams for $5$-dimensional cobordisms with $2$- and $3$-handles, $5$-dimensional $3$-handlebodies, and closed $5$-manifolds. We show that every such smooth $5$-manifold can be represented by a Heegaard diagram, and that two Heegaard diagrams represent diffeomorphic $5$-manifolds if and only if they are related by certain moves. As an application, we construct Heegaard diagrams for $5$-dimensional cobordisms from the standard $4$-sphere to the Gluck twists along knotted $2$-spheres. This provides several statements equivalent to the Gluck twist being diffeomorphic to the standard $4$-sphere. 
\end{abstract}

\maketitle
\addtocontents{toc}{\protect\setcounter{tocdepth}{1}}

\section{Introduction}

We work in the smooth category throughout unless otherwise stated. A ($3$-dimensional) \textit{Heegaard diagram} is a triple $(\Sigma,\alpha,\beta)$, where $\Sigma$ is a closed, orientable, connected surface, and each of $\alpha$ and $\beta$ is the image of an embedding of a disjoint union of circles, as described in \autoref{def:Heegaard diagrams}  for $k=1$. From $\Sigma\times[-1,1]$, we can construct a $3$-manifold by attaching $2$-handles along $\alpha\times\{-1\}$ and $\beta\times\{1\}$, resulting in a $3$-dimensional cobordism between two closed surfaces $\Sigma(\alpha)$ and $\Sigma(\beta)$. Here, $\Sigma(\alpha)$ and $\Sigma(\beta)$ are obtained by performing surgery on $\Sigma$ along $\alpha$ and $\beta$, respectively. 

If $\Sigma(\alpha)$ is diffeomorphic to the $2$-sphere $S^2$, then capping it off yields a $3$-manifold with boundary $\Sigma(\beta)$. If $\Sigma(\beta)$ is diffeomorphic to $S^2$ as well, then capping it off results in a closed $3$-manifold. It is well known that every $3$-manifold (with two boundary components, one boundary component, or no boundary component) can be represented by a Heegaard diagram, and that two Heegaard diagrams represent diffeomorphic $3$-manifolds if and only if they are related by isotopies, handle slides, stabilizations, and diffeomorphisms \cite{reidemeister1933dreidimensionalen,singer1933three}.

In this paper, we extend this approach to dimension $5$ by introducing ($5$-dimensional) Heegaard diagrams for $5$-manifolds (see \autoref{def:Heegaard diagrams} for $k=2$) and proving that two ($5$-dimensional) such diagrams represent diffeomorphic $5$-manifolds if and only if they are related by isotopies, handle slides, stabilizations, and diffeomorphisms (\autoref{thm: Heegaard diagrams representing diffeomorphic 5-manifolds}).

A \textit{$k$-link} $\alpha\subset \Sigma$ in an $n$-manifold $\Sigma$ is the image of an embedding \[f:\coprod^i S^k\hookrightarrow \Sigma,\] where $\coprod^i S^k$ denotes a disjoint union of $i$ copies of the $k$-sphere. We call $\alpha$ a \textit{$k$-knot} if $i=1$. We say that $\alpha$ has  \textit{trivial normal bundle} if a closed regular neighborhood $\nu(\alpha)$ of $\alpha$ in $\Sigma$ is diffeomorphic to $\coprod^i(S^k\times B^{n-k})$.

A \textit{framing} of $\alpha$ is an embedding \[\phi:\coprod ^i (S^k\times B^{n-k})\hookrightarrow \Sigma\] such that
    \begin{enumerate}
        \item $\phi(\coprod^i(S^k\times\{0\}))=\alpha,$
        \item $\phi(\coprod ^i (S^k\times B^{n-k}))=\nu(\alpha)$.
    \end{enumerate} 
The pair $(\alpha,\phi)$ is called a \textit{framed $k$-link}. If the framing is understood from context, we simply refer to $\alpha$.

A \textit{$k$-surgery on $\Sigma$ along $(\alpha,\phi)$} is the $n$-manifold \[\Sigma(\alpha,\phi)=\overline{\Sigma\setminus \nu(\alpha)}\cup_{\phi|_{\coprod^i(S^{k}\times S^{n-k-1})}}\left(\coprod^i(B^{k+1}\times S^{n-k-1})\right)\] obtained by removing a regular neighborhood of $\alpha$ and gluing in $\coprod^i(B^{k+1}\times S^{n-k-1})$ along the boundary using $\phi$. If the framing is understood, we may simply write: \[\Sigma(\alpha)=\overline{\Sigma\setminus\nu(\alpha)}\cup_{\alpha}(\coprod^i(B^{k+1}\times S^{n-k-1})),\] which we refer to as \textit{$k$-surgery on $\Sigma$ along $\alpha$}. 

The $(n+1)$-manifold
\begin{align*}M_{\alpha}&=(\Sigma\times[0,1])\cup_{\alpha\times\{1\}}\left(\coprod^i(B^{k+1}\times B^{n-k})\right)\\&=(\Sigma\times[0,1])\cup_{\phi\times\{1\}}\left(\coprod^i(B^{k+1}\times B^{n-k})\right)\end{align*} is called \textit{the manifold obtained from $\Sigma\times[0,1]$ by attaching $(k+1)$-handles along $\alpha\times\{1\}\subset\Sigma\times\{1\}$}, where the embedding $\phi\times\{1\}:\coprod^i(S^{k}\times B^k)\hookrightarrow\Sigma\times\{1\}$ is defined by $(\phi\times\{1\})(x)=(\phi(x),1)$. Here, $B^{k+1}\times B^{n-k}$ is called an \textit{$(n+1)$-dimensional $(k+1)$-handle}. We call $S^k\times \{0\}$ the \textit{attaching sphere of the handle}, $S^k\times\{0\}$ the \textit{attaching region}, $\{0\}\times S^{n-k-1}$ the \textit{belt sphere}, and $B^{k+1}\times S^{n-k-1}$ the \textit{belt region}.

Note that the boundary of $M_{\alpha}$ is given by $\partial M_\alpha=\partial_{-}M_\alpha\coprod \partial_{+}M_{\alpha}$, where $\partial_{-}M_\alpha$ is identified with $\Sigma=\Sigma\times\{0\}$ and $\partial_{+}M_{\alpha}$ is identified with the $k$-surgery $\Sigma(\alpha)$. Thus, we write $\partial M_{\alpha}=\Sigma\coprod\Sigma(\alpha)$. Moreover, $M_{\alpha}$ can be considered as the manifold obtained from $\Sigma(\alpha)\times[0,1]$ by attaching $(n-k)$-handles along the belt spheres of the $k$-handles.

\begin{definition}\label{def:Heegaard diagrams}
A \textit{$(2k+1)$-dimensional Heegaard diagram} is a triple $(\Sigma,\alpha,\beta)$ such that
    \begin{enumerate}
        \item $\Sigma$ is a closed, orientable, connected $2k$-manifold,
        \item $\alpha$ is a framed $k$-link in $\Sigma$,
        \item $\beta$ is a framed $k$-link in $\Sigma$.
    \end{enumerate}
\end{definition}
    
\begin{remark}\hfill
    \begin{enumerate}
        \item $\alpha$ and $\beta$ may intersect transversely at points. Thus, the union $\alpha\cup\beta$ is immersed, while each of $\alpha$ and $\beta$ is embedded.
        \item When $k=1$, this definition coincides with the classical definition of Heegaard diagrams for $3$-manifolds.
        \item In this paper, we focus on the case $k=2$, i.e., $5$-dimensional Heegaard diagrams.
        \item The set of isotopy classes of framings of $\alpha$ (respectively, $\beta$) is canonically identified with $\pi_k(SO(k))$ after choosing a fixed reference framing. When $k=1$ or $2$, any two framings of $\alpha$ (respectively, $\beta$) are isotopic since $\pi_k(SO(k))$ is trivial. Thus, we can choose an arbitrary framing of $\alpha$ (respectively, $\beta$). In general, $\pi_k(SO(k))$ is non-trivial for $k>2$, but in particular, $\pi_6(SO(6))$ is trivial.
        \item Hughes, Kim, and Miller showed that any surface in a $4$-manifold can be represented by a banded unlink diagram if it is embedded \cite{hughes2020isotopies}, or by a singular banded unlink diagram if it is immersed \cite{hughes2021band}; see \autoref{def:singular banded unlink diagram} and \autoref{thm: existence of singular banded unlink diagram}. Therefore, any $5$-dimensional Heegaard diagram $(\Sigma,\alpha,\beta)$ can be represented by a singular banded unlink diagram $(\mathcal{K},L,B)=(\mathcal{K},J_1\cup J_2, B_1\cup B_2)$, where $(\mathcal{K},J_1,B_1)$ and $(\mathcal{K},J_2,B_2)$ are banded unlink diagrams of $(\Sigma,\alpha)$ and $(\Sigma,\beta)$, respectively. In \autoref{sec: Kirby diagrams for surgeries}, we provide an algorithm for constructing a Kirby diagram of the $2$-surgery $\Sigma(\alpha)$ (respectively, $\Sigma(\beta)$) from the banded unlink diagram $(\mathcal{K},J_1,B_1)$ (respectively, $(\mathcal{K},J_2,B_2)$); see \autoref{pro: kirby diagram for 2-surgeries}. This generalizes Gompf and Stipsicz's algorithm for obtaining a Kirby diagram of the complement of a ribbon surface in a $4$-ball $B^4$ in \cite{gompf20234}. Therefore, we can easily determine the boundary $4$-manifolds of the $5$-dimensional cobordism $M_\alpha\cup_\Sigma M_\beta$ defined below; see \autoref{fig: cobordism from S^4 to non-simply connected homology sphere}.
    \end{enumerate}
\end{remark}

\begin{definition}\label{def: (2k+1)-dimensional cobordism}
Let $(\Sigma,\alpha,\beta)$ be a $(2k+1)$-dimensional Heegaard diagram. We define  \[M_\alpha\cup_{\Sigma}M_\beta=M_\alpha\cup M_\beta=\Sigma\times[-1,1]\cup_{\alpha\times\{-1\}\coprod\beta\times\{1\}} ((k+1)\text{-handles})\] to be the $(2k+1)$-manifold obtained from $\Sigma\times[-1,1]$ by attaching $(2k+1)$-dimensional $(k+1)$-handles along $\alpha\times\{-1\}\coprod\beta\times\{1\}$, where \[M_\alpha=\Sigma\times[-1,0]\cup_{\alpha\times\{-1\}}((k+1)\text{-handles})\quad \text{and}\quad M_\beta=\Sigma\times[0,1]\cup_{\beta\times\{1\}}((k+1)\text{-handles}).\] Here, $M_\alpha\cap M_\beta=\Sigma\times\{0\}$, which is identified with $\Sigma$.
\end{definition}

\begin{remark}\hfill
\begin{enumerate}
    \item The boundary of $M_\alpha\cup_\Sigma M_\beta$ is given by $\partial(M_\alpha\cup_\Sigma M_{\beta})=\Sigma(\alpha)\coprod\Sigma(\beta)$, where $\Sigma(\alpha)$ and $\Sigma(\beta)$ are the results of performing $k$-surgery on $\alpha$ and $\beta$ in $\Sigma$, respectively.
    \item $M_\alpha\cup_\Sigma M_\beta$ can be interpreted as an $(2k+1)$-dimensional cobordism from $\Sigma(\alpha)$ to $\Sigma(\beta)$, obtained from $\Sigma(\alpha)\times[0,1]$ by attaching $k$-handles to $\Sigma(\alpha)\times\{1\}$ and $(k+1)$-handles to $\Sigma$. In other words, $\alpha$ corresponds to the set of the belt spheres of the $k$-handles and $\beta$ corresponds to the set of the attaching spheres of the $(k+1)$-handles.
    \item When $k=1$, if $\Sigma(\alpha)$ is diffeomorphic to $S^2$, then a $3$-ball can be attached to $M_\alpha\cup_\Sigma M_\beta$ along $\Sigma(\alpha)$ to obtain a $3$-manifold $\widehat{M_\alpha}\cup_\Sigma M_\beta$ with one boundary component $\Sigma(\beta)$. If $\Sigma(\beta)$ is also diffeomorphic to $S^2$, then a $3$-handle can be attached to $\widehat{M_\alpha}\cup_\Sigma M_\beta$ along $\Sigma(\beta)$ to obtain a closed $3$-manifold $\widehat{M_\alpha}\cup_\Sigma \widehat{M_\beta}$. There is a unique way to attach, up to diffeomorphism, a $3$-ball along the $2$-sphere boundary because every self-diffeomorphism of $S^2$ extends to a self-diffeomorphism of $B^3$, which is known as Alexander's trick \cite{alexander1923deformation}.
\end{enumerate}
\end{remark}

From now on, we focus on $5$-dimensional Heegaard diagrams and will simply refer to them as Heegaard diagrams.

\begin{definition}\label{def: some 5-manifolds from Heegaard diagrams}
Let $(\Sigma,\alpha,\beta)$ be a Heegaard diagram. We define the following $5$-manifolds:
    \begin{enumerate}
        \item If $\Sigma(\alpha)$ is diffeomorphic to $\#^k(S^1\times S^3)$, let \[\widehat{M_\alpha}\cup_\Sigma M_\beta=(M_\alpha\cup_\Sigma M_\beta)\cup_g(\natural^k(S^1\times B^4))\] for some diffeomorphism $g:\#^k(S^1\times S^3)\rightarrow \Sigma(\alpha)$, where $\widehat{M_\alpha}=M_\alpha\cup_g(\natural^k(S^1\times B^4))$.
        \item If $\Sigma(\beta)$ is diffeomorphic to $\#^r(S^1\times S^3)$, let \[\widehat{M_\beta}\cup_\Sigma\widehat{M_\beta}=(\widehat{M_\alpha}\cup_\Sigma M_\beta)\cup_h(\natural^r(S^1\times B^4))\] for some diffeomorphism $h:\#^r(S^1\times S^3)\rightarrow \Sigma(\beta)$, where $\widehat{M_\beta}=M_\beta\cup_h(\natural^r(S^1\times B^4))$.
    \end{enumerate}
\end{definition}

\begin{remark}\hfill
    \begin{enumerate}
        \item $M_\alpha$ is uniquely determined up to diffeomorphism by the isotopy class of $\alpha$ since the set of framings of the attaching sphere of a $3$-handle is identified with $\pi_2(SO(2))=1$. The same holds for $M_\beta$.
        \item $\widehat{M_\alpha}$ is uniquely determined up to diffeomorphism by Cavicchioli and Hegenbarth \cite{cavicchioli1993determination}, who showed that any self-diffeomorphism of $\#^k(S^1\times S^3)$ extends to a self-diffeomorphism of $\natural^k (S^1\times B^4)$. This result generalizes a theorem of Laudenbach and Poénaru \cite{laudenbach1972note}. Aribi, Courte, Golla, and Moussard used this result in their development of quadrisection diagrams for closed $5$-manifolds \cite{aribi2023multisections}.
    
        We can view $\widehat{M_\alpha}$ as a $5$-manifold obtained from $\Sigma\times [-1,0]$ by attaching $|\alpha|$ $3$-handles, $k$ $4$-handles, and a $5$-handle, where $|\alpha|$ is the number of components of $\alpha$. Alternatively, it can be viewed as a $5$-dimensional $2$-handlebody, which is the union of a $0$-handle, $k$ $1$-handles, and $|\alpha|$ $2$-handles. The Euler characteristic of $\widehat{M_{\alpha}}$ is given by $\chi(\widehat {M_\alpha})=1-k+|\alpha|$. A similar argument applies to $\widehat{M_\beta}$.
        \item $M_\alpha\cup_\Sigma M_\beta$ can be viewed as a $5$-dimensional cobordism from $\Sigma(\alpha)$ to $\Sigma(\beta)$, constructed by attaching $2$- and $3$-handles. Specifically, $\alpha$ and $\beta$ correspond to the set of the belt spheres of the $2$-handles and the set of the attaching spheres of the $3$-handles, respectively.
        \item $\widehat{M_\alpha}\cup_\Sigma M_\beta$ can be viewed as a $5$-dimensional $3$-handlebody, which is the union of a $0$-handle, $k$ $1$-handles, $|\alpha|$ $2$-handles, and $|\beta|$ $3$-handles. The Euler characteristic is given by $\chi(\widehat {M_\alpha}\cup M_\beta)=1-k+|\alpha|-|\beta|$.
        \item $\widehat{M_\alpha}\cup_\Sigma \widehat{M_\beta}$ can be viewed as a closed $5$-manifold, which is the union of a $0$-handle, $k$ $1$-handles, $|\alpha|$ $2$-handles, $|\beta|$ $3$-handles, $r$ $4$-handles, and a $5$-handle. The Euler characteristic is given by $\chi(\widehat {M_\alpha}\cup \widehat{M_\beta})=1-k+|\alpha|-|\beta|+r-1$.
    \end{enumerate}    
\end{remark}

We show that every $5$-dimensional cobordism with $2$- and $3$-handles, every $5$-dimensional $3$-handlebody, and every closed, connected, orientable $5$-manifold can be represented by a Heegaard diagram.

\begin{restatable*}{theorem}{Heegaardexistence}Let $X$ be a $5$-dimensional cobordism with $2$- and $3$-handles, a $5$-dimensional $3$-handlebody, or closed, connected, orientable $5$-manifold.
    \begin{enumerate}
        \item If $X$ is a $5$-dimensional cobordism with $2$- and $3$-handles, then $X$ is diffeomorphic to $M_\alpha\cup_{\Sigma} M_\beta$ for some Heegaard diagram $(\Sigma,\alpha,\beta).$
        \item If $X$ is a $5$-dimensional $3$-handlebody, then $X$ is diffeomorphic to $\widehat{M_\alpha}\cup_{\Sigma} M_\beta$ for some Heegaard diagram $(\Sigma,\alpha,\beta).$
        \item If $X$ is a closed, connected, orientable $5$-manifold, then $X$ is diffeomorphic to $\widehat{M_\alpha}\cup_{\Sigma} \widehat{M_\beta}$ for some Heegaard diagram $(\Sigma,\alpha,\beta).$
    \end{enumerate}
\end{restatable*}

We recall that an \textit{n-dimensional $k$-handlebody} is an $n$-manifold obtained from an $n$-ball $B^n$ by attaching handles of index up to $k$. The following theorem implies that every $5$-dimensional $2$-handlebody is the product of a $4$-dimensional $2$-handlebody and an interval.

\begin{theorem}[\cite{kim2025note}]\label{thm: product structures on handlebodies}
Fix $k\geq0$ and $n\geq2k+1$. Let $X$ be an $n$-dimensional $k$-handlebody. Then there exists a $2k$-dimensional $k$-handlebody $Y\subset X$ such that $X\cong Y\times B^{n-2k}$. 
\end{theorem}

Given a Heegaard diagram $(\Sigma,\alpha,\beta)$, if the $2$-surgery $\Sigma(\alpha)$ is diffeomorphic to $\#^k(S^1\times S^3)$, then we can construct a $5$-dimensional $2$-handlebody $\widehat{M_\alpha}$. The following corollary is then immediate.

\begin{corollary}\label{thm: Heegaard diagrams for 5d 2handlebodies}
Let $(\Sigma,\alpha,\beta)$ be a Heegaard diagram. If $\Sigma(\alpha)\cong\#^k(S^1\times S^3)$, then there exists a $4$-dimensional $2$-handlebody $Y$ such that $\widehat{M_\alpha}\cong Y\times B^1$, and therefore $\Sigma$ is diffeomorphic to the double of $Y$.
\end{corollary}

We show that a Heegaard diagram of a $5$-manifold is unique up to a sequence of isotopies (\autoref{def: isotopy on Heegaard diagrams}), handle slides (\autoref{def: handle slides on Heegaard diagrams}), stabilizations (\autoref{def: stabilizations on Heegaard diagrams}), and diffeomorphisms (\autoref{def: diffeomorphisms on Heggaard diagrams}).

\begin{restatable*}{theorem}{Heegaardmoves}\label{thm: Heegaard diagrams representing diffeomorphic 5-manifolds}Let $(\Sigma,\alpha,\beta)$ and $(\Sigma',\alpha',\beta')$ be Heegaard diagrams.
    \begin{enumerate}
        \item $(\text{$5$-dimensional cobordisms})$ \\Assume $\Sigma(\alpha)\cong\Sigma'(\alpha')$ and $\Sigma(\beta)\cong\Sigma'(\beta')$.    
        Then $M_\alpha\cup_\Sigma M_\beta\cong M_{\alpha'}\cup_{\Sigma'} M_{\beta'}$ if and only if the diagrams are related by isotopies, handle slides, (first, second, and third) stabilizations, and diffeomorphisms.
        \item $(\text{$5$-dimensional $3$-handlebodies})$\\ Assume $\Sigma(\alpha)\cong\#^k(S^1\times S^3)$, $\Sigma'(\alpha')\cong\#^{k'}(S^1\times S^3)$, and $\Sigma(\beta)\cong\Sigma'(\beta')$ for some $k,k'\geq0$. Then $\widehat{M_\alpha}\cup_\Sigma M_\beta\cong \widehat{M_{\alpha'}}\cup_{\Sigma'} M_{\beta'}$ if and only if the diagrams are related by isotopies, handle slides, (first, second, and third) stabilizations, and diffeomorphisms.
        \item $(\text{Closed $5$-manifolds})$\\ Assume $\Sigma(\alpha)\cong\#^k(S^1\times S^3)$, $\Sigma'(\alpha')\cong\#^{k'}(S^1\times S^3)$, $\Sigma(\beta)\cong\#^r(S^1\times S^3)$, and $\Sigma'(\beta')\cong\#^{r'}(S^1\times S^3)$ for some $k,k',r,r'\geq0$. Then $\widehat{M_\alpha}\cup_\Sigma \widehat{M_\beta}\cong \widehat{M_{\alpha'}}\cup_{\Sigma'} \widehat{M_{\beta'}}$ if and only if the diagrams are related by isotopies, handle slides, (first, second, and third) stabilizations, and diffeomorphisms.
    \end{enumerate}
\end{restatable*}

Given a Heegaard diagram $(\Sigma,\alpha,\beta)$, first and third stabilizations do not change $\Sigma$, whereas the second stabilization changes $\Sigma$ to $\Sigma \#(S^2 \times S^2)$. This leads to \autoref{cor: Heegaard 4-manifolds can be used to distinguish 5-manifolds}, which can be used to distinguish two $5$-manifolds that are not diffeomorphic. In other words, if $\Sigma\#(\#^k(S^2\times S^2))\ncong\Sigma'\#(\#^{k'}(S^2\times S^2))$ for all $k,k'\geq0$, then $(\Sigma,\alpha,\beta)$ and $(\Sigma',\alpha',\beta')$ represent non-diffeomorphic $5$-manifolds. For example, since $\pi_1(\Sigma)\cong\pi_1(\Sigma\#(\#^k(S^2\times S^2)))$ for all $k\geq 0$, if $\pi_1(\Sigma)\ncong\pi_1(\Sigma')$, then two Heegaard diagrams $(\Sigma,\alpha,\beta)$ and $(\Sigma',\alpha',\beta')$ represent non-diffeomorphic $5$-manifolds. However, in dimension $3$, any two Heegaard surfaces are diffeomorphic after some stabilizations (connected sum Heegaard surface with $S^1\times S^1$) because every orientable surface is diffeomorphic to $\#^{m}(S^1\times S^1)$ for some $m\geq 0$. Therefore, Heegaard surface cannot be used to distinguish $3$-manifolds in this sense.

\begin{corollary}\label{cor: Heegaard 4-manifolds can be used to distinguish 5-manifolds}
If two Heegaard diagrams $(\Sigma,\alpha,\beta)$ and $(\Sigma',\alpha',\beta')$ represent diffeomorphic $5$-manifolds, then $\Sigma\#(\#^k(S^2\times S^2))$ and $\Sigma'\#(\#^{k'}(S^2\times S^2))$ are diffeomorphic for some $k,k'\geq0$.
\end{corollary}

We recall that the Gluck twist $X_K$ of an $(n+2)$-manifold $X$ along an $n$-knot $K$ with trivial normal bundle is an $(n+2)$-manifold obtained from $X$ by removing a closed regular neighborhood $\nu(K)$ of the $n$-knot and reattaching it in a non-trivial manner (\autoref{def:Gluck twist}). When $X$ is the standard $4$-sphere $S^4$, Gluck showed in \cite{gluck1962embedding} that every Gluck twist $S^4_K$ of $S^4$ along a $2$-knot $K$ is a homotopy $4$-sphere. Thus, by Freedman \cite{freedman1982topology}, it is homeomorphic to $S^4$. However, it remains unknown whether the Gluck twists are diffeomorphic to the standard $S^4$ in general.

We construct a natural $(n+3)$-dimensional cobordism $W_{X,K}$ from an $(n+2)$-manifold $X$ to the Gluck twist $X_K$ along a $n$-knot $K$, using a single $2$-handle and a single $(n+1)$-handle (\autoref{def: 5d cobordism gluck twist} and \autoref{thm: cobordism from X to XK}). Additionally, we present equivalent statements characterizing when the Gluck twists $S^4_K$ is diffeomorphic to the standard $S^4$.

\begin{restatable*}{theorem}{HeegaardGluck}\label{thm: Gluck twist and equivalent statements}
Let $S^4_K$ be the Gluck twist of $S^4$ along a $2$-knot $K\subset S^4$. Let \[(\Sigma,\alpha,\beta)=(S^2\tilde{\times}S^2, F, K\#F)\] be a Heegaard diagram, where $F$ is a fiber of $S^2\Tilde{\times}S^2$.
Then the following statements are equivalent:\hfill
    \begin{enumerate}
        \item $S^4_K$ is diffeomorphic to $S^4$.
        \item $M_\alpha\cup_\Sigma M_\beta\; (\cong W_{S^4,K})$ is diffeomorphic to a twice-punctured $S^2\Tilde{\times}S^3$.
        \item $(S^2\Tilde{\times}S^2, F, K\#F)$ and $(S^2\Tilde{\times}S^2,F,F)$ are related by isotopies, handle slides, stabilizations, and diffeomorphisms.
        \item $(S^2\Tilde{\times}S^2,K\#F)$ is diffeomorphic to $(S^2\Tilde{\times}S^2,F)$.
    \end{enumerate}
\end{restatable*}

Melvin showed that $S^4_K$ is diffeomorphic to $S^4$ if and only if the pair $(\mathbb{C}P^2, K\#\mathbb{C}P^1)$ is diffeomorphic to $(\mathbb{C}P^2,\mathbb{C}P^1)$; see \cite{melvin1977blowing}. We note that $(\mathbb{C}P^2,K\#\mathbb{C}P^1)\cong (\mathbb{C}P^2,\mathbb{C}P^1)$ implies $(4)$ in \autoref{thm: Gluck twist and equivalent statements} because $(S^2\tilde{\times}S^2,F)\cong(\mathbb{C}P^2\#\overline{\mathbb{C}P^2},\mathbb{C}P^1\#\overline{\mathbb{C}P^1})$. However, the converse is not immediately obvious. Although condition $(4)$ is seemingly weaker, it is still sufficient to trivialize the Gluck twist.

\subsection*{Organization}In \autoref{sec: preliminaries}, we review basic handle decomposition theory for manifolds of arbitrary dimension, Kirby diagrams for $4$-manifolds, (singular) banded unlink diagrams for surfaces in $4$-manifolds, and $1$- and $2$-surgery on $4$-manifolds. In \autoref{sec:Heegaard diagrams for 5-manifolds}, we review $5$-dimensional Heegaard diagrams and provide numerous examples. In \autoref{sec:Gluck twists}, we discuss the Gluck twist and construct an interesting 
cobordism from a $4$-manifold to its Gluck twist along a $2$-knot.

\subsection*{Acknowledgements}
The author would like to thank David Gay for valuable discussions and many helpful feedback. The author also thanks Seungwon Kim, Maggie Miller, and Patrick Naylor for helpful discussions. This project was partially supported by National Science Foundation grant DMS-2005554 ``Smooth $4$--Manifolds: $2$--, $3$--, $5$-- and $6$--Dimensional Perspectives''.
 
\section{Preliminaries}\label{sec: preliminaries}
In \autoref{sec: handle decomposition}, we first review some background on handle decomposition theory and certain moves on handle decompositions of a manifold; see \cite{milnor1963morse,milnor2015lectures,kosinski2013differential} for more details. In \autoref{sec: Kirby diagrams}, we discuss handle decompositions of $4$-manifolds via Kirby diagrams; for further details, refer to \cite{kirby2006topology,kirby1978calculus,gompf20234,akbulut20164}. In \autoref{sec: banded unlink diagrams}, we review decompositions of pairs $(X,F)$, where $F$ is an embedded or immersed surface in a $4$-manifold $X$, via (singular) banded unlink diagrams; see \cite{hughes2020isotopies,hughes2021band}. Finally, in \autoref{sec: Kirby diagrams for surgeries}, we describe an algorithm for finding a Kirby diagram of surgery on a $4$-manifold along embedded $1$-spheres or $2$-spheres from a banded unlink diagram.

\subsection{Handle decompositions}\label{sec: handle decomposition}

Let $X$ be an $n$-manifold and $\phi:S^{k-1}\times B^{n-k}\hookrightarrow\partial X$ be an embedding. The \textit{manifold obtained from $X$ by attaching an $n$-dimensional $k$-handle $h^k=B^k\times B^{n-k}$ along $\phi$} is the quotient manifold \[X\cup_\phi h^k=(X\coprod (B^k\times B^{n-k}))/ \sim,\] where $x\sim\phi(x)$ for all $x\in S^{k-1}\times B^{n-k}$. The map $\phi$ is called the \textit{attaching map} of $h^k$.

\begin{definition}\label{def: handle decomposition}
Let $X$ be a compact $n$-manifold with $\partial X=\partial_{-}X\coprod \partial_{+}X.$ A \textit{handle decomposition} of $X$ (relative to $\partial_{-}X)$ is a sequence of manifolds \[X_{-1}\subseteq X_0 \subseteq \dots \subseteq X_{n}=X\] such that
    \begin{enumerate}
        \item $X_{-1}= \partial_{-}X\times [0,1]$,
        \item $\partial_{-} X_{k}=\partial_{-}X$, 
        \item $X_k$ is obtained from $X_{k-1}$ by attaching $k$-handles to $\partial_{+}X_{k-1}$.
    \end{enumerate}
More precisely, \[X_k=X_{k-1}\cup_{\phi_1}(B^k\times B^{n-k})\cup_{\phi_2}\dots\cup_{\phi_t}(B^k\times B^{n-k})\] for some embeddings $\phi_i:S^{k-1}\times B^{n-k}\hookrightarrow \partial_{+}X_{k-1}$ such that $\phi_i(S^{k-1}\times B^{n-k})\cap \phi_j(S^{k-1}\times B^{n-k})=\emptyset$ for all $i\neq j \in \{1,\dots, t\}.$
\end{definition}

\begin{proposition}[\cite{milnor1963morse,milnor2015lectures}]\label{pro: existence of a handle decomposition}
Every compact, smooth $n$-manifold $X$ admits a handle decomposition (relative to $\partial_{-} X$).
\end{proposition}

\begin{proof}
By Morse theory, there exists a self-indexing Morse function $f:X\rightarrow [-1-\frac{1}{2},n+\frac{1}{2}]$ such that $f^{-1}(-1-\frac{1}{2})=\partial_{-}X$, $f^{-1}(n+\frac{1}{2})=\partial_{+}X$, and $f(x)=\operatorname{ind}(x)$ for every critical point $x$, where $\operatorname{ind}(x)$ denotes the index of $x$. Then the sublevel sets $X_k=f^{-1}([-1-\frac{1}{2},k+\frac{1}{2}])$ give a handle decomposition $X_{-1}\subseteq X_0 \subseteq \dots \subseteq X_{n}=X$.
\end{proof}

\begin{remark}\label{rem: basics for handle decompositions} \hfill
    \begin{enumerate}
        \item $X_{0}$ is a disjoint union of $X_{-1}$ and $0$-handles.
        \item If $\partial_{-}X=\emptyset,$ then $X_{-1}=\emptyset.$
        \item If there are no $k$-handles attached to $\partial_{+} X_{k-1}$, then $X_{k-1}=X_k.$
        \item If $X$ is a compact, connected manifold with $\partial_{-}X\neq \emptyset$ and $\partial_{+}X\neq \emptyset$, then $X$ admits a handle decomposition without $0$-handles and $n$-handles. That is, $X_{-1}\subseteq X_0\subseteq X_1 \subseteq \dots \subseteq X_n=X,$ where $X_{-1}=X_0$ and $X_{n-1}=X_n$.
        \item If $X$ is a compact, connected manifold with $\partial_{-}X=\emptyset$, then $X$ admits a handle decomposition $X_0\subseteq X_1 \subseteq \dots \subseteq X_n=X,$ where $X_0$ is a single $0$-handle. Here, $X_k$ is called an \textit{$n$-dimensional $k$-handlebody}.
        \item If $X$ is a closed manifold (i.e., compact with $\partial X=\emptyset$), then $X$ admits a handle decomposition $X_0\subseteq X_1\subseteq \dots \subseteq X_n,$ where $X_0$ is a single $0$-handle and $X_{n}$ is obtained from $X_{n-1}$ by attaching a single $n$-handle.
        \item The boundary $\partial_{+} X_k$ is obtained from $\partial_{+} X_{k-1}$ by $(k-1)$-surgery. Similarly, $\partial_{+} X_{k-1}$ is obtained from $\partial_{+} X_k$ by $(n-k-1)$-surgery.
        \item Let $f:X\rightarrow [-1-\frac{1}{2}, n+\frac{1}{2}]$ be the Morse function in the proof of \autoref{pro: existence of a handle decomposition}. Define a function $g:X\rightarrow [-1-\frac{1}{2}, n+\frac{1}{2}]$ defined by $g(x)=n-1-f(x)$. Then $M_{-1}\subseteq M_0 \subseteq \dots \subseteq M_{n}=X$ is a handle decomposition of $X$ (relative to $\partial_{+}X$), where $M_k=g^{-1}([-1-\frac{1}{2},k+\frac{1}{2}])$, and it is called the \textit{dual handle decomposition} of the handle decomposition $X_{-1}\subseteq X_0 \subseteq \dots \subseteq X_{n}=X$.
        \item The homology of the pair $(X,\partial_{-} X)$ can be computed from a handle decomposition. Let $C_k(X,\partial X)$ be the free abelian group generated by the oriented $k$-handles. Define the boundary map $\partial_k:C_k(X,\partial X)\rightarrow C_{k-1}(X,\partial X)$ by $\partial_k(h^k)=(-1)^{k-1}\sum_i (A^k\cdot B^{k-1}_i)h^{k-1}_i$, where $h^{k-1}_i$ is the indexed $(k-1)$-handle, and $A^k\cdot B^{k-1}_i$ is the algebraic intersection number between the attaching sphere $A^k$ of $h^k$ and the belt sphere $B^{k-1}_i$ of $h^{k-1}_i$. See \cite{durst2019handle} for more details.
    \end{enumerate}
\end{remark}

\begin{definition}\label{def: handle slides}
Let $X\cup_\phi h^k_\alpha \cup_\psi h^k_\beta$ be an $n$-manifold obtained from $X$ by attaching two $k$-handles  along $\phi$ and $\psi$. Let $X\cup_\phi h^k_\alpha \cup_{\psi'} h^k_{\beta'}$ be another $n$-manifold, where the second handle is attached along a different embedding $\psi'$. We say that $X\cup_\phi h^k_\alpha \cup_{\psi'} h^k_{\beta'}$ is obtained from $X\cup_\phi h^k_\alpha \cup_\psi h^k_\beta$ by a \textit{handle slide} of $h^k_\beta$ over $h^k_\alpha$ if there exists an embedding $F:(S^{k-1}\times B^{n-k})\times [0,1]\rightarrow \partial(X\cup_\phi h^k_\alpha)\times [0,1]$ such that 
   \begin{enumerate}
       \item $F(x,t)\subset \partial (X\cup_\phi h^k_\alpha)\times\{t\}$ for every $x\in S^{k-1}\times B^{n-k}$ and $t\in [0,1]$,
       \item $F(x,0)=(\psi(x),0)$ for every $x\in S^{k-1}\times B^{n-k}$,
       \item $F(x,1)=(\psi'(x),1)$ for every $x\in S^{k-1}\times B^{n-k}$,
       \item $F((S^{k-1}\times \{0\})\times [0,1])$ and $B^k_\alpha\times [0,1]$ intersect transversely at one point, where $B^k_\alpha$ is the belt sphere of $h^k_\alpha$. 
   \end{enumerate}
\end{definition}

Since a handle slide is an isotopy of an attaching map, we have the following:

\begin{proposition}
In \autoref{def: handle slides}, the two manifolds $X\cup_\phi h^k_\alpha \cup_\psi h^k_\beta$ and $ X\cup_\phi h^k_\alpha \cup_{\psi'} h^k_{\beta'}$ are diffeomorphic. That is, handle slides do not change the diffeomorphism type.
\end{proposition}

\begin{definition}
Let $N_1\cup N_2\subset M$ be an $n$-submanifold of an $m$-manifold $M$, where $m>n$. An $(n+1)$-submanifold $b\subset M$ is called an \textit{$(n+1)$-dimensional $1$-handle connecting $N_1$ and $N_2$} if there exists an embedding $e:B^1\times B^n\hookrightarrow M$ such that 
    \begin{enumerate}
        \item $b=e(B^1\times B^n)$,
        \item $b\cap N_1=e(\{-1\}\times B^n)$,
        \item $b\cap N_2=e(\{1\}\times B^n)$.
    \end{enumerate} We call \[N_1\#_bN_2=\left((N_1\cup N_2)\setminus e(\partial B^1\times B^1)\right)\cup e(B^1\times \partial B^n)\] the \textit{manifold obtained from $N_1\cup N_2$ by surgery along $b$} or \textit{connected sum of $N_1$ and $N_2$ along $b$}.
\end{definition}

\begin{remark}
In \autoref{def: handle slides}, the attaching sphere $A^k_{\beta'}$ of $h^k_{\beta'}$ is obtained by taking the connected sum of the push-off $\Tilde{A^k_\alpha}$ (with respect to a given framing) of the attaching sphere $A^k_\alpha$ and the attaching sphere $A^k_\beta$ along a $k$-dimensional $1$-handle $b\subset \partial X$ connecting them. That is, $A^k_{\beta'}=\tilde{A^k_\alpha}\#_bA^k_\beta$.
\end{remark}

\begin{definition}\label{def: cancelling pair}
Let $X\cup_{\phi} h^{k-1}\cup_{\psi} h^k$ be an $n$-manifold obtained from $X$ by attaching a $(k-1)$-handle and a $k$-handle. If the attaching sphere of $h^k$ intersects the belt sphere of $h^{k-1}$ intersect transversely at one point in $\partial (X\cup_\phi h^{k-1})$, then the pair $(h^{k-1},h^k)$ is called a \textit{cancelling $(k-1,k)$-pair}. We say that $X\cup_{\phi} h^{k-1}\cup_{\psi} h^k$ is obtained from $X$ by \textit{creation} of a cancelling $(k-1,k)$-pair. Conversely, we say that $X$ is obtained from $X\cup_{\phi} h^{k-1}\cup_{\psi} h^k$ by \textit{annihilation} of a canceling $(k-1,k)$-pair.
\end{definition}

\begin{proposition}[\cite{milnor2015lectures}]
In \autoref{def: cancelling pair}, $X\cup_{\phi} h^{k-1}\cup_{\psi} h^k$ and $X$ are diffeomorphic. That is, a cancelling pair may be added or removed without changing the diffeomorphism type. 
\end{proposition}

\begin{theorem}[\cite{cerf1970stratification}]\label{thm: Cerf's moves} Any two handle decompositions of a compact smooth manifold $(X, \partial_{-}X)$ are related by isotopies, handle slides, and the creation/annihilation of cancelling pairs. 
\end{theorem}

Later, we carefully interpret isotopies, handle slides, and the creation/annihilation of a cancelling pair in \autoref{thm: Cerf's moves} in the setting of Kirby diagrams for $4$-manifolds in \autoref{sec: Kirby diagrams} and Heegaard diagrams for $5$-manifolds in \autoref{sec:Heegaard diagrams for 5-manifolds}.

\subsection{Kirby diagrams for 4-manifolds}\label{sec: Kirby diagrams}

\begin{definition}
Let $K\subset S^3$ be a knot ($1$-knot). An \textit{$m$-framing} of $K$ is an embedding $\phi:S^1\times B^2\hookrightarrow S^3$ such that 
    \begin{enumerate}
        \item $\phi(S^1\times\{(0,0)\})=K$,
        \item $\phi(S^1\times B^2)=\nu(K)$,
        \item $\operatorname{lk}(K,\phi(S^1\times\{(1,0)\}))=m\in\mathbb{Z}\cong\pi_1(SO(2))$.
    \end{enumerate} Here, $\nu(K)$ is a closed
regular neighborhood of $K$, and $\operatorname{lk}(K,\phi(S^1\times\{(1,0)\}))$ is the linking number between $K$ and the \textit{push-off} $\tilde{K}=\phi(S^1\times \{(1,0)\})$. 
We call the pair $(K,\phi)$ an \textit{$m$-framed knot}, and simply draw the knot $K$ with the integer $m.$  A \textit{framed link} is a link $L\subset S^3$ in which each component is a framed knot.
\end{definition}

\begin{definition}
A link $L\subset S^3$ is called the \textit{unlink} or \textit{trivial link} if $L=\phi(\coprod S^1)$ for some embedding $\phi:\coprod B^2\hookrightarrow S^3$. Here, $D_L=\phi(\coprod B^2)$ is called \textit{the collection of the trivial disks} of $L$, so that $\partial D_{L}=L$. The $1$-component unlink is called the \textit{unknot}. The unknot $K\subset S^3$ with a dot is called a \textit{dotted unknot}. A \textit{push-off} of $K$ is a dotted longitude $\tilde{K}\subset\partial \nu(K)$ such that $\operatorname{lk}(K,\tilde{K})=0$. An unlink $L=K_1\cup\dots\cup K_n\subset S^3$ is called a \textit{dotted unlink} if each $K_i$ is a dotted unknot.
\end{definition}

\begin{definition}
    Let $L\subset S^3=\partial B^4$ be a dotted unlink and $D_L\subset S^3$ be its collection of the trivial disks with $\partial D_L=L.$ Let $D_L'\subset B^4$ be the collection of the properly embedded trivial disks obtained by pushing the interior of $D_L$ into the interior of $B^4.$ Define the $4$-manifold\[M_{L}=\overline{B^4\setminus\nu(D_L')}\] to be the closure of the complement of the closed regular neighborhood $\nu(D_L')$ in $B^4$.
\end{definition}

\begin{remark}\hfill
    \begin{enumerate}
        \item $M_{L}\cong \natural^{|L|}(S^1\times B^3),$ where $|L|$ is the number of components of $L$ and $\natural^{|L|}(S^1\times B^3)$ is the boundary connected sum of $|L|$ copies of $(S^1\times B^3)$. Hence, $M_{L}$ can be considered as a manifold obtained from $B^4$ by attaching $|L|$ $1$-handles. The collection of the trivial disks $D_L\setminus\operatorname{int}(\nu(L))\subset S^3$ can be viewed as the hemispheres of the belt spheres of these $1$-handles.
        \item $\partial M_{L}\cong \#^{|L|}(S^1\times S^2),$ where $\#^{|L|}(S^1\times S^2)$ is the connected sum of $|L|$ copies of $S^1\times S^2$.
    \end{enumerate}
\end{remark}

\begin{definition}
A \textit{Kirby diagram} $\mathcal{K}=L_1\cup L_2\subset S^3$ is a link in $S^3$, where $L_1$ is the dotted unlink and $L_2=(K_1,\phi_1)\cup\dots\cup (K_n,\phi_n)$ is a framed link with each $\phi_i$ a framing of $K_i$. We define the $4$-manifold\[M_\mathcal{K}=M_{L_1}\cup_{\phi_1}(B^2\times B^2)\cup_{\phi_2}\dots\cup_{\phi_n}(B^2\times B^2)\] to be the result of attaching $2$-handles along $L_2$ (or $\phi_i's).$
\end{definition}

\begin{remark}\hfill
    \begin{enumerate}
        \item Since $L_2\subset S^3\setminus L_1$, it is embedded in $\partial M_{L_1}$.
        \item $M_\mathcal{K}$ is obtained from $B^4$ by carving out the collection of the properly embedded trivial disks $D_{L_1}'\subset B^4$ with $\partial D_{L_1}'=L_1\subset S^3$ and attaching $2$-handles along $L_2$.
        \item We can easily see how the $2$-handles interact with the $1$-handles by observing the intersections of $L_2$ with the collection of the trivial disks $\tilde{D_{L_1}}=D_{L_1}\setminus \operatorname{int}(\nu(L_1))$, which represent the belt spheres of the $1$-handles, where $D_{L_1}$ is the trivial disks of $L_1$. See the left of \autoref{fig: Mazur manifold}.
        \item Let $U$ be a dotted unknot and $U'$ be a $0$-framed unknot. Then $M_{U}\cong S^1\times B^3$ and $S^2\times B^2\cong M_{U'}$ are not diffeomorphic, though $\partial M_{U}\cong S^1\times S^2\cong \partial M_{U'}.$ The manifold $M_U$ is obtained from $M_{U'}$ by surgery along $S^2\times \{0\}\subset S^2\times B^2$, and conversely $M_{U'}$ is obtained from $M_{U}$ by surgery along $S^1\times \{0\}\subset S^1\times B^3$.
        \item Let $\tilde{\mathcal{K}}$ be the Kirby diagram obtained from $\mathcal{K}$ by replacing the dotted unlink with a $0$-framed unlink. Then $\partial M_{\mathcal{K}}\cong \partial M_{\tilde{\mathcal{K}}}.$
    \end{enumerate}
\end{remark}

\begin{definition}\label{def: Kirby diagram for closed manifold}
Let $\mathcal{K}=L_1\cup L_2\subset S^3$ be a Kirby diagram with $\partial M_{\mathcal{K}}\cong \#^k (S^1\times S^2)$ for some $k\geq 0.$ Let $f:\#^k (S^1\times S^2)\rightarrow \partial M_{\mathcal{K}}$ be a diffeomorphism. We define  \[\widehat{M_{\mathcal{K}}}=M_{\mathcal{K}}\cup_{f}(\natural^k(S^1\times B^3))\] to be the closed $4$-manifold obtained from $M_{\mathcal{K}}$ by gluing $\natural^k(S^1\times B^3)$ along $f$.
\end{definition}

\begin{remark}\hfill
    \begin{enumerate}
        \item $\natural^k(S^1\times B^3)$ can be described either as the union of a $0$-handle and $k$ $1$-handles or as the union of $k$ $3$-handles and a $4$-handle.
        \item $\widehat{M_{\mathcal{K}}}$ is uniquely determined up to diffeomorphism because every self-diffeomorphism of $\#^k(S^1\times S^2)$ extends to a self-diffeomorphism of $\natural^k (S^1\times B^3)$ \cite{laudenbach1972note}. Thus, in the diagram $\mathcal{K}$, it suffices to depict $\widehat{M_\mathcal{K}}$ without explicitly including the $k$ $3$-handles and the $4$-handle since their attachment is uniquely determined.
        \item  A straightforward way to attach the $k$ $3$-handles and the $4$-handle to $M_\mathcal{K}$ is as follows. First, attach $k$ $3$-handles to $M_{\mathcal{K}}$ along the $2$-spheres $\coprod^k(\{x_0\}\times S^2)\subset \#^k(S^1\times S^2)\cong\partial M_{\mathcal{K}}$, where $\{x_0\}\in S^1$. Then attach the $4$-handle along $S^3$ that results from performing surgery on $\partial M_{\mathcal{K}}$ along the attaching spheres of the $3$-handles.
    \end{enumerate}    
\end{remark}

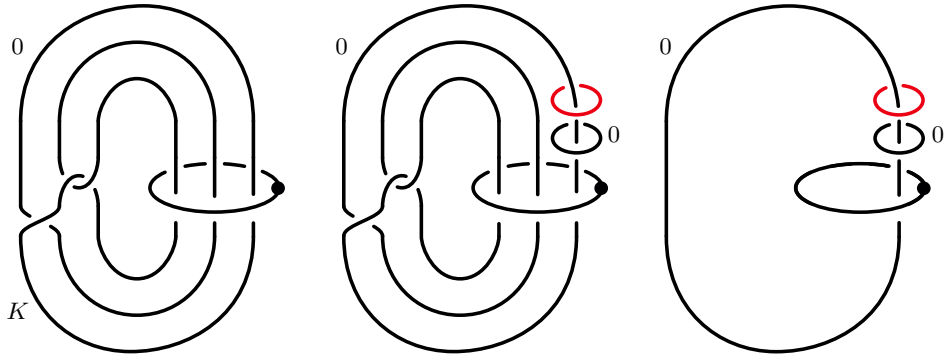
\begin{figure}[ht!]
\labellist
\small\hair 2pt
\pinlabel{$0$} at 15 200
\pinlabel{$0$} at 197 200
\pinlabel{$0$} at 350 150
\pinlabel{$0$} at 379 200
\pinlabel{$0$} at 532 150
\pinlabel{$K$} at 15 50
\endlabellist
\centering
\includegraphics[width=0.8\textwidth]{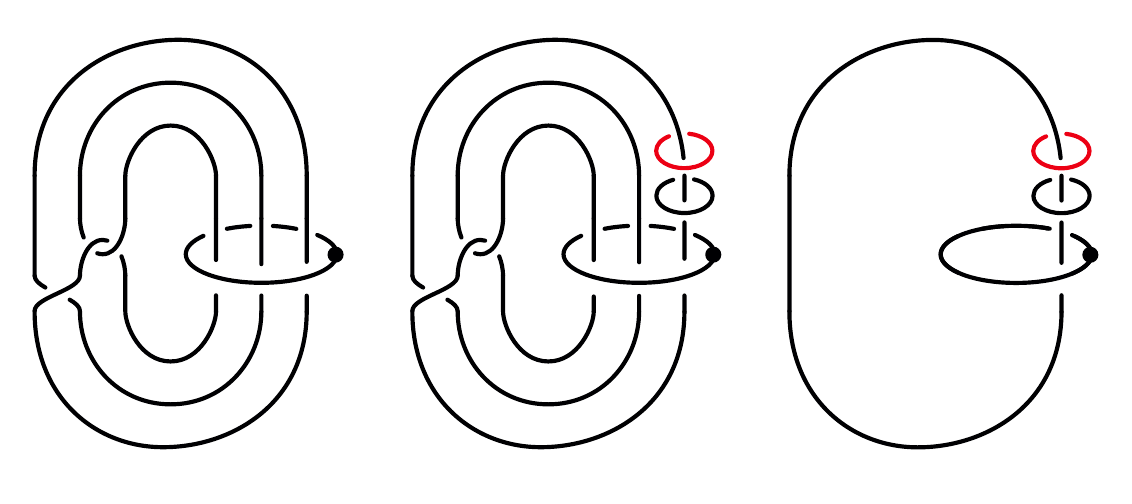}
\caption{\textbf{Left}: A Kirby diagram of Mazur's contractible manifold $M$, which admits a handle decomposition consisting of a $0$-handle, a $1$-handle, and a $2$-handle. The $0$-framed attaching sphere $K$ of the $2$-handle intersects the belt sphere of the $1$-handle geometrically three times and algebraically once. \textbf{Middle}: A Heegaard diagram of $M\times B^1$, obtained from a Kirby diagram of the double of the left by adding a red meridian. \textbf{Right}: Another Heegaard diagram of $M\times B^1$, obtained from the middle diagram after sliding $K$ over the $0$-framed meridian to change the crossings of $K$. This diagram represents $B^5$ after cancelling a $(1,2)$-pair and a $(2,3)$-pair, followed by a first destabilization.}\label{fig: Mazur manifold}
\end{figure}

We recall that if a Heegaard diagram $(\Sigma,\alpha,\beta)$ represents a $5$-dimensional $3$-handlebody or closed $5$-manifold, then $\Sigma$ is diffeomorphic to the double $DY$ of some $4$-dimensional $2$-handlebody $Y$; see \autoref{thm: Heegaard diagrams for 5d 2handlebodies}. We therefore review an algorithm for constructing a Kirby diagram of the double $DY$.

\begin{proposition}[A Kirby diagram of the double $DY$ of a $4$-dimensional $2$-handlebody $Y$]\label{pro: how to draw a Kirby diagram for the double of a 4-dimensional 2-handlebody}
Let $Y$ be a $4$-dimensional $2$-handlebody, and let $\mathcal{K}=L_1\cup L_2\subset S^3$ be a Kirby diagram of $Y$, where $L_1$ is a dotted unlink and $L_2$ is a framed link. The double of $Y$ is $DY=Y\cup_{id} \overline{Y}$, and $\overline{Y}$ has the canonical handle decomposition consisting of $2$-handles, $3$-handles, and a $4$-handle, obtained by turning the handle decomposition of $Y$ upside down. The attaching spheres of the $2$-handles of $\overline{Y}$ are glued to the belt sphere of the $2$-handles of $Y$. Therefore, we can obtain a Kirby diagram $\mathcal{K}'=\mathcal{K}\cup J$ of $DY$, where $J$ is a union of the $0$-framed meridians of $L_2$, i.e., $\widehat{M_\mathcal{K'}}\cong DY$.
\end{proposition}

\begin{example}
The left of \autoref{fig: Mazur manifold} shows a Kirby diagram of the Mazur manifold, a contractible $4$-manifold not homeomorphic to $B^4$ \cite{mazur1961note}. The middle diagram in \autoref{fig: Mazur manifold} shows a Kirby diagram of the double of the Mazur manifold (ignoring the red circle). Here, the $3$-handle and $4$-handle are omitted.
\end{example}

\begin{theorem}[\cite{kirby2006topology}]
Every closed, connected, orientable, smooth $4$-manifold is diffeomorphic to $\widehat{M_{\mathcal{K}}}$ for some Kirby diagram $\mathcal{K}$.
\end{theorem}

We now define a collection of moves on Kirby diagrams, interpreting isotopies, handle slides, and cancelling pairs of $4$-manifolds in the context of Kirby diagrams.

\begin{definition}
Let $\mathcal{K}\subset S^3$ be a Kirby diagram. Let $K_i, K_j\subset \mathcal{K}$ be two knots and, let $\tilde{K_j}\subset \partial\nu(K_j)$ be a push-off of $K_j$. The knot $K_j$ may be either a dotted unknot or a framed knot. A $2$-dimensional submanifold $b\subset S^3$ is called a \textit{sliding band connecting $K_i$ and $\tilde{K_j}$} if there exists an embedding $e:B^1\times B^1\hookrightarrow S^3$ such that
    \begin{enumerate}
        \item  $b=e(B^1\times B^1)$,
        \item  $b\cap K_i=e(\{-1\}\times B^1)$,
        \item  $b\cap \tilde{K_j}=e(\{1\}\times B^1)$,
        \item  $e((-1,1)\times B^1)\cap (\mathcal{K}\cup \nu(K_j))=\emptyset$.
    \end{enumerate}
 We call \[K_i\#_b \tilde{K_j}=((K_1\cup K_2)\setminus e(\partial B^1\times B^1))\cup e(B^1\times\partial B^1)\] the \textit{manifold obtained from $K_i\cup \tilde{K_j}$ by surgery along $b$} or \textit{connected sum of $K_i$ and $\tilde{K_j}$ along $b$}.
\end{definition}

\begin{definition}\label{def:handle slides on a Kirby diagram} Let $\mathcal{K}=L_1\cup L_2\subset S^3$ be a Kirby diagram, where $L_1$ is the dotted unlink and $L_2$ is a framed link.
    \begin{enumerate}
        \item Let $K_i,K_j\subset L_1$ be two dotted unknots. Let $D_{L_1}$ be the set of the trivial disks of $L_1$ with $\partial D_{L_1}=L_1$. Let $\tilde{K_j}$ be a push-off of $K_j$. Let $b\subset S^3$ be a sliding band connecting $K_i$ and $\tilde{K_j}$ such that $b\cap \operatorname{int}(D_{L_1})=\emptyset$. We call $K_i\#_b \tilde{K_j}$ \textit{the result of sliding $K_i$ over $K_j$} or a \textit{$1$-handle slide over a $1$-handle}. We say that two Kirby diagrams $\mathcal{K}$ and $\mathcal{K'}$ are related by a \textit {$1$-handle slide over a $1$-handle} if $\mathcal{K'}=(\mathcal{K}\setminus K_i)\cup (K_i\#_b \tilde{K_j})$. See the first row of \autoref{fig: Handle slides}.
        \item Let $K_i\subset L_2$ be an $m_i$-framed knot and $K_j\subset L_1$ be a dotted unknot. Let $\tilde{K_j}$ be a push-off of $K_j$. Let $b\subset S^3$ be a sliding band connecting $K_i$ and $\tilde{K_j}.$ We define $K_i\#_b \tilde{K_j}$ as an $(m_i+2\operatorname{lk}(K_i,\tilde{K_j}))$-framed knot and call it \textit{the result of sliding $K_i$ over $K_j$} or a \textit{$2$-handle slide over a $1$-handle}. Here, the linking number $\operatorname{lk}(K_i,\tilde{K_j})$ is calculated by orienting $K_i$ and $\tilde{K_j}$ so that the orientation of $(K_i\cup \tilde{K_j})\setminus b$ extends to the orientation of $K_i\#_b \tilde{K_j}.$ We say that two Kirby diagrams $\mathcal{K}$ and $\mathcal{K'}$ are related by a \textit{$2$-handle slide over a $1$-handle} if $\mathcal{K'}=(\mathcal{K}\setminus K_i)\cup (K_i\#_b \tilde{K_j})$. See the second row of \autoref{fig: Handle slides}.
        \item Let $K_i,K_j\subset L_2$ be $m_i$-framed knot and $m_j$-framed knot, respectively. Let $\tilde{K_j}$ be a push-off of $K_j$. Let $b\subset S^3$ be a sliding band connecting $K_i$ and $\tilde{K_j}.$ We define $K_i\#_b \tilde{K_j}$ as an $(m_i+m_j+2\operatorname{lk}(K_i,\tilde{K_j}))$-framed knot and call it \textit{the result of sliding $K_i$ over $K_j$} or a \textit{$2$-handle slide over a $2$-handle}. Here, the linking number $\operatorname{lk}(K_i,\tilde{K_j})$ is calculated by orienting $K_i$ and $\tilde{K_j}$ so that the orientation of $(K_i\cup \tilde{K_j})\setminus b$ extends to the orientation of $K_i\#_b \tilde{K_j}.$ We say that two Kirby diagrams $\mathcal{K}$ and $\mathcal{K'}$ are related by a \textit{$2$-handle slide over a $2$-handle} if $\mathcal{K'}=(\mathcal{K}\setminus K_i)\cup (K_i\#_b \tilde{K_j})$. See the third row of \autoref{fig: Handle slides}.
    \end{enumerate}   
\end{definition}

We note that a handle slide is originally defined between handles of the same index. However, the notion of a $2$-handle slide over a $1$-handle in \autoref{def:handle slides on a Kirby diagram} arises from the dotted notation for $1$-handles. In fact, this handle slide corresponds to an isotopy of a $2$-handle that does not interact with a $1$-handle; see \cite{gompf20234,akbulut20164} for more details.

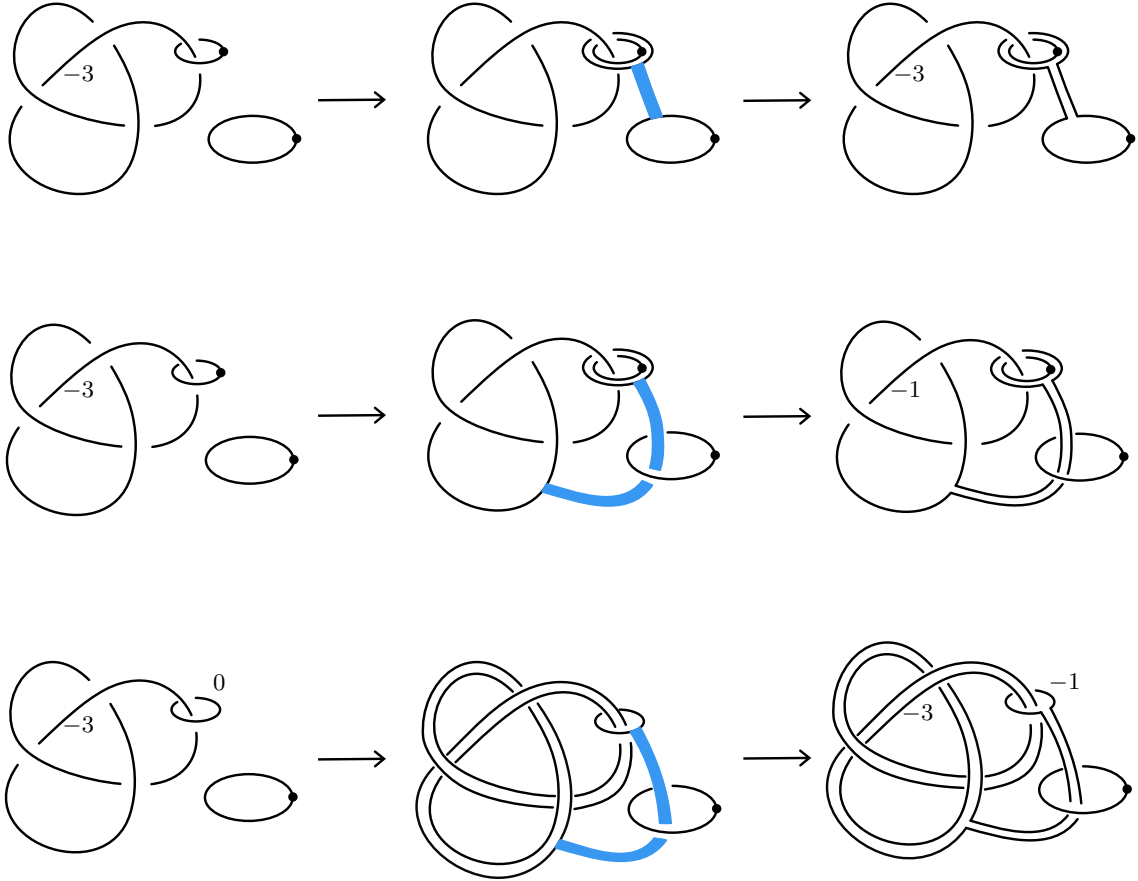
\begin{figure}[ht!]
\labellist
\small\hair 2pt
\pinlabel{$-3$} at 95 700
\pinlabel{$-3$} at 773 700
\pinlabel{$-3$} at 95 440
\pinlabel{$-1$} at 770 440
\pinlabel{$-3$} at 95 165
\pinlabel{$0$} at 210 200
\pinlabel{$-1$} at 900 200
\pinlabel{$-3$} at 780 175
\endlabellist
\centering
\includegraphics[width=1\textwidth]{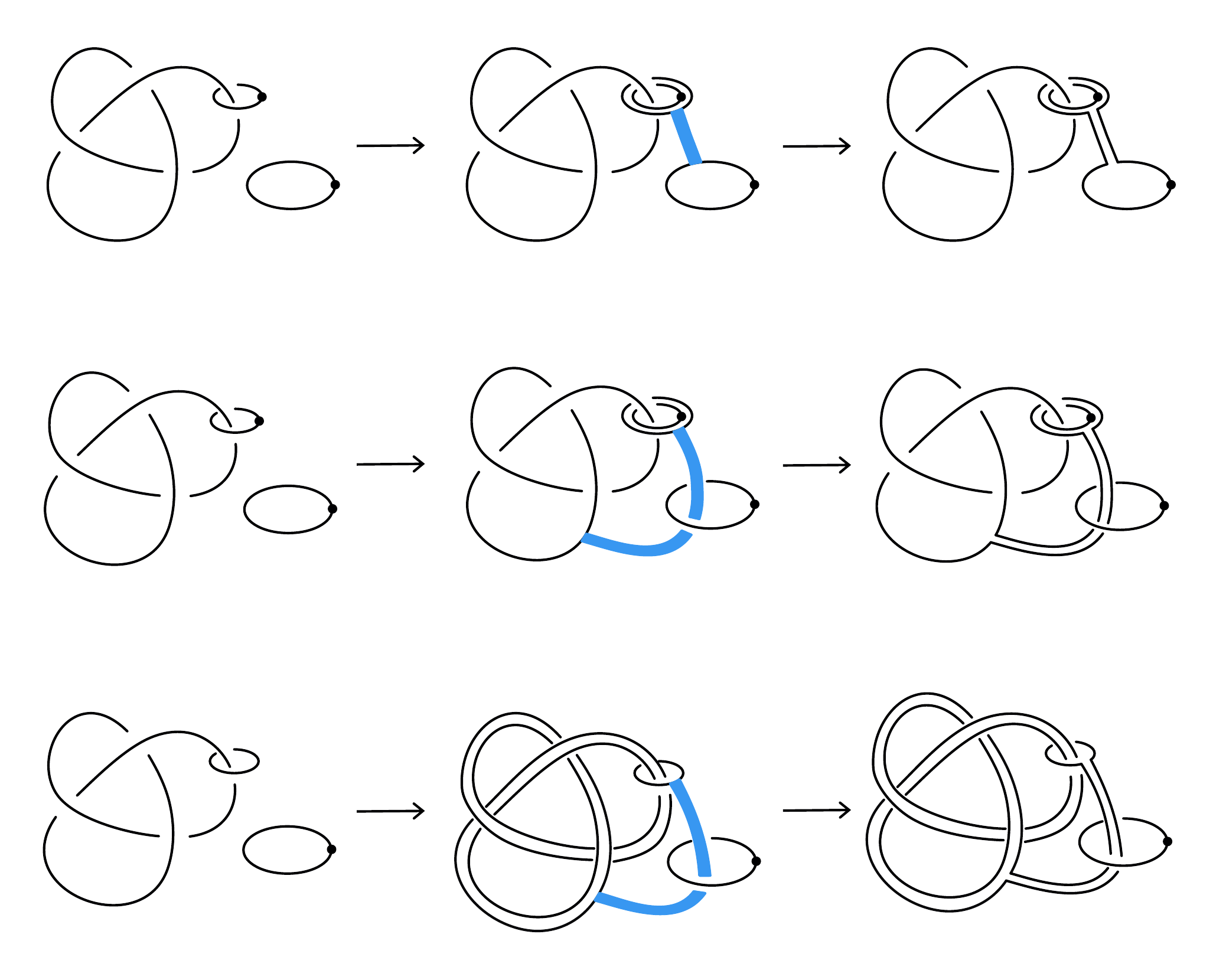}
\caption{Three types of handle slides. \textbf{First row}: A $1$-handle slide over a $1$-handle. \textbf{Second row}: A $2$-handle slide over a $1$-handle. \textbf{Third row}: A $2$-handle slide over a $2$-handle.}
\label{fig: Handle slides}
\end{figure}

\begin{definition}\label{def:cancelling pairs on a Kirby diagram}Let $\mathcal{K}=L_1\cup L_2\subset S^3$ be a Kirby diagram, where $L_1$ is a dotted unlink and $L_2$ is a framed link.
    \begin{enumerate}
        \item Let $L=K_1\cup K_2\subset S^3\setminus\mathcal{K}$ be a two-component link, where $K_1$ is a framed knot and $K_2$ is a dotted meridian of $K_1.$ We call $L$ a \textit{cancelling $(1,2)$-pair}. Let $\mathcal{K'}=\mathcal{K}\cup L$ be a new Kirby diagram. We say that $\mathcal{K'}$ is obtained from $\mathcal{K}$ by \textit{creating} a cancelling $(1,2)$-pair and that $\mathcal{K}$ is obtained from $\mathcal{K'}$ by \textit{annihilating} a cancelling $(1,2)$-pair. See the left of \autoref{fig: Cancelling pairs}.
        \item Let $B\subset S^3$ be a $3$-ball such that $\mathcal{K}\cap B=\emptyset.$ Let $U\subset B$ be a $0$-framed unknot.  We call such a knot $U$ a \textit{cancelling $(2,3)$-pair}. Let $\mathcal{K'}=\mathcal{K}\cup U$ be a new Kirby diagram. We say that $\mathcal{K'}$ is obtained from $\mathcal{K}$ by \textit{creating} a cancelling $(2,3)$-pair and that $\mathcal{K}$ is obtained from $\mathcal{K'}$ by \textit{annihilating} a cancelling $(2,3)$-pair. See the right of \autoref{fig: Cancelling pairs}.
    \end{enumerate}
\end{definition}

\begin{figure}[ht!]
\labellist
\small\hair 2pt
\pinlabel{$n$} at 5 110
\pinlabel{$0$} at 460 100
\endlabellist
\centering
\includegraphics[width=0.6\textwidth]{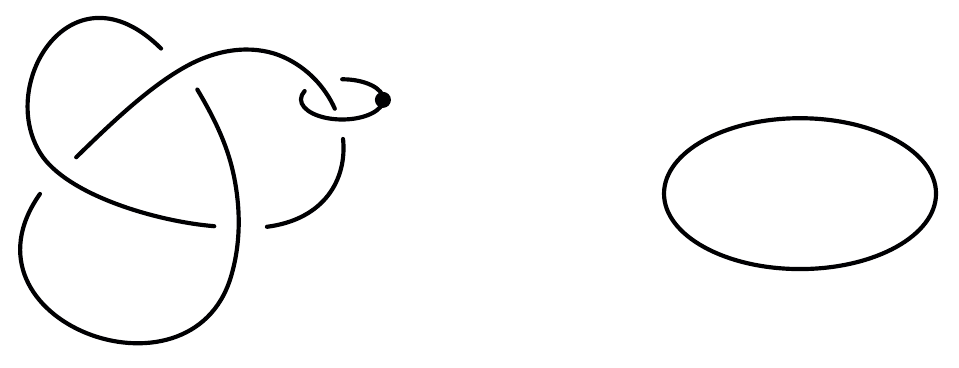}
\caption{\textbf{Left}: A cancelling $(1,2)$-pair. \textbf{Right}: A cancelling $(2,3)$-pair.}
\label{fig: Cancelling pairs}
\end{figure}

We note that for a cancelling $(2,3)$-pair, we do not draw the $3$-handle, which is cancelled with the $2$-handle attached along the $0$-framed unknot $U$. That is, $M_{\mathcal{K}}\cong M_{\mathcal{K}'} \cup $  $3$-handle. More precisely, the $3$-handle is attached along the standard $2$-sphere $\{x_0\}\times S^2\subset \partial M_{\mathcal{K}}\#(S^1\times S^2)\cong \partial M_{\mathcal{K}'}$, where $x_0\in S^1$.

\begin{theorem}[\cite{kirby2006topology}]
Let $\mathcal{K,K'}$ be Kirby diagrams of closed $4$-manifolds. Then $\widehat{M_\mathcal{K}}\cong \widehat{M_{\mathcal{K'}}}$ if and only if they are related by isotopies, handle slides $(1$-handles over $1$-handles, $2$-handles over $1$-handles, and $2$-handles over $2$-handles$)$, and creation/annihilation of cancelling pairs $($$(1,2)$-cancelling pairs and $(2,3)$-cancelling pairs$)$. 
\end{theorem}

We refer to the moves defined in \autoref{def:handle slides on a Kirby diagram} and \autoref{def:cancelling pairs on a Kirby diagram} as \textit{Kirby moves}.

\begin{definition}
Let $\mathcal{K}\subset S^3$ be a Kirby diagram. Let $B\subset S^3$ be a $3$-ball such that $\mathcal{K}\cap B=\emptyset.$ Let $K\subset B$ be a $(\pm1)$-framed unknot. Let $\mathcal{K'}=\mathcal{K}\cup K$ be a new Kirby diagram. We say that $\mathcal{K'}$ is obtained from $\mathcal{K}$ by \textit{blowing up} and that $\mathcal{K}$ is obtained from $\mathcal{K'}$ by \textit{blowing down}. 
\end{definition}

We note that $M_{K}$ is diffeomorphic to $\mathbb{C}P^2\setminus \operatorname{int}(B^4)$ when $K$ is a $1$-framed unknot and to $\overline{\mathbb{C}P^2}\setminus \operatorname{int}(B^4)$ when $K$ is a $(-1)$-framed unknot. In either case, $\partial M_K\cong S^3$.

\begin{theorem}[\cite{lickorish1962representation}]
Every closed, orientable, connected $3$-manifold is diffeomorphic to $\partial M_{\mathcal{K}}$ for some Kirby diagram $\mathcal{K}.$ In particular, we can assume that $\mathcal{K}$ has no dotted unlink.
\end{theorem}

\begin{theorem}[\cite{kirby1978calculus}]\label{thm: Kirby moves for diffemorphic 3-manifolds}
Let $\mathcal{K}$ and $\mathcal{K}'$ be Kirby diagrams. Let $\tilde{\mathcal{K}}$ and $\tilde{\mathcal{K'}}$ be Kirby diagrams obtained from $\mathcal{K}$ and $\mathcal{K}'$, respectively, by replacing each dotted unlink with a $0$-framed unlink. Then $\partial M_{\mathcal{K}}\cong\partial M_{\mathcal{K}'}$ if and only if $\tilde{\mathcal{K}}$ and $\tilde{\mathcal{K'}}$ are related by isotopies, $2$-handle slides over $2$-handles, and blow-ups or blow-downs. 
\end{theorem}

\autoref{thm: Kirby moves for diffemorphic 3-manifolds} can be used to determine whether a given Kirby diagram $\mathcal{K}$ represents a closed $4$-manifold, that is, whether $\partial M_{\mathcal{K}}\cong\#^k(S^1\times S^2)$ for some $k\geq0$. If this is the case, then the Kirby diagram $\tilde{\mathcal{K}}$ (obtained from $\mathcal{K}$ by replacing the dotted unlink with a $0$-framed unlink) and a $k$-component $0$-framed unlink are related by isotopies, $2$-handle slides over $2$-handles, and blow-ups or blow-downs.

\subsection{Banded unlink diagrams for surfaces in 4-manifolds}\label{sec: banded unlink diagrams} 

\begin{definition}
A \textit{singular link} $L$ in a $3$-manifold $Z$ is the image of an immersion $i:\coprod^m S^1\rightarrow Z$ that is injective except at isolated transverse double points. At each double point $p$, we include a small disk $v\cong B^2$ embedded in $Z$ such that $(v,v\cap L)\cong(B^2,\{(x,y)\in B^2|xy=0\}).$ We refer to these disks as the \textit{vertices} of $L$.
\end{definition}

\begin{definition}
A \textit{marked singular link} $(L,\sigma)$ in a $3$-manifold $Z$ is a singular link $L$ together with decorations $\sigma$ on the vertices of $L$, as follows. Let $v$ be a vertex of $L$ with $\partial v\cap \overline{(L\setminus v)}$ consisting of the four points $p_1,p_2,p_3,p_4$ in cyclic order. Choose a co-orientation of the disk $v$. On the positive side of $v,$ add an arc connecting $p_1$ and $p_3.$ On the negative side of $v,$ add an arc connecting $p_2$ and $p_4.$ See the left of \autoref{fig: Singular points}. 

Let $L^{+}$ denote the link in $Z$ obtained from $(L,\sigma)$ by pushing the arc of $L$ between $p_1$ and $p_3$ off $v$ in the positive direction, and repeating at every vertex in $L.$ This is called the \textit{positive resolution} of $(L,\sigma)$; see the top right of \autoref{fig: Singular points}.

Similarly, let $L^{-}$ denote the link in $Z$ obtained from $(L,\sigma)$ by pushing the arc of $L$ between $p_1$ and $p_3$ off $v$ in the negative direction, and repeating at each vertex in $L.$ This is called the \textit{negative resolution} of $(L,\sigma)$; see the bottom right of \autoref{fig: Singular points}.

For each marked vertex $v$ of $L,$ these opposite push-offs form a bigon in a neighborhood of $v$, which bounds an embedded disk $c_v.$ This disk may be chosen so that its interior intersects $L$ transversely in a single point near $v.$ We call such a disk a \textit{companion disk} of $v$; see the middle right of \autoref{fig: Singular points}. We denote $C_L$ by the union of all of the companion disks. 
\end{definition}

\begin{figure}[ht!]
\labellist
\small\hair 2pt
\pinlabel{$L$} at 45 100
\pinlabel{$v$} at 70 155
\pinlabel{$L^{-}$} at 345 60
\pinlabel{$L\cup c_v$} at 345 155
\pinlabel{$L^{+}$} at 345 255
\endlabellist
\centering
\includegraphics[width=0.59\textwidth]{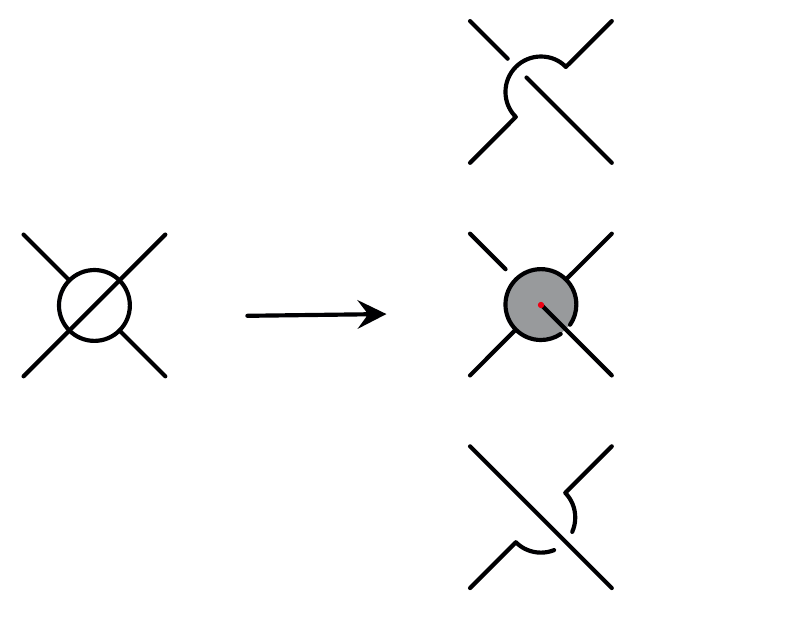}
\caption{\textbf{Left}: A vertex $v$ of $(L,\sigma)$. \textbf{Top right}: The positive resolution of $(L,\sigma)$. \textbf{Middle right}: A union of $L$ and a companion disk $c_v$. \textbf{Bottom right}: The negative resolution of $(L,\sigma)$.}
\label{fig: Singular points}
\end{figure}
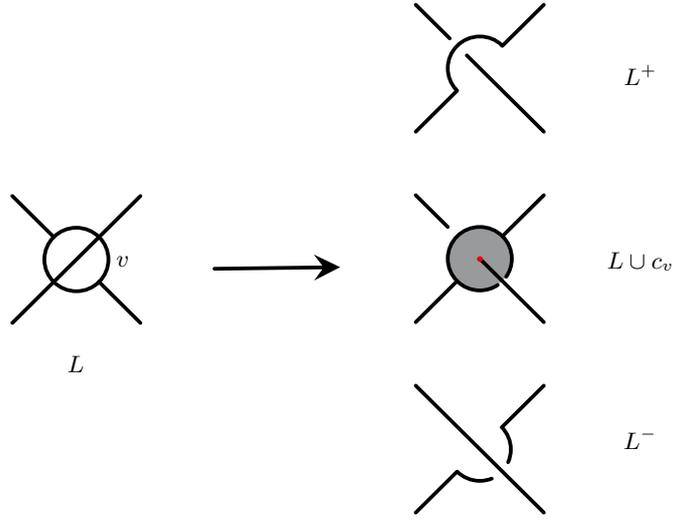

\begin{definition}
Let $L$ be a marked singular link in $Z$, and let $V_L$ denote the union of the vertices of $L$. A \textit{band} $b$ attached to $L$ is the image of an embedding $\phi:B^1\times B^1\rightarrow Z\setminus V_L$ such that $b\cap L =\phi(B^1\times \{-1,1\})$. We call $\phi(B^1\times \{0\})$ the \textit{core} of the band $b$. 

Let $L_b$ be the singular link defined by $L_b=(L- \phi(\{-1,1\}\times B^1))\cup \phi(B^1\times\{-1,1\}).$ We say that $L_b$ is the result of performing \textit{band surgery} to $L$ along $b$. 

If $B$ is a finite collection of pairwise disjoint bands for $L$, then we denote by $L_B$ the singular link obtained by performing band surgery along each band in $B$. We say that $L_B$ is the result of \textit{resolving} the bands in $B$. Note that the self-intersections of $L_B$ naturally correspond to those of $L$, so a choice of markings for $L$ induces markings for $L_B$. 

A triple $(L,\sigma,B)$, where $(L,\sigma)$ is a marked singular link and $B$ is a collection of disjoint bands for $L$, is called a \textit{marked singular banded link}. To ease notation, we may refer to the pair $(L,B)$ as a \textit{singular banded link} and implicitly remember that L is a marked singular link. 

We call $(L,B)$ a \textit{banded link} if $L$ has no vertices. In this case, the negative and positive resolutions $L^-$ and $L^+$ are both identified with $L$, and the collection of companion disks is $C_L=\emptyset$. 
\end{definition}

\begin{definition}\label{def:singular banded unlink diagram}
Let $\mathcal{K}=L_1\cup L_2\subset S^3$ be a Kirby diagram with $\partial M_{\mathcal{K}}\cong\#^k(S^1\times S^2)$. Let $(L,B)$ be a singular banded link in $S^3\setminus \mathcal{K}$. The triple $(\mathcal{K},L,B)$ is called a \textit{singular banded unlink diagram} if $L^-$ is the unlink in $S^3\setminus L_1$ and $L^{+}_{B}$ is the unlink in $\partial M_{\mathcal{K}}$. We refer to a triple $(\mathcal{K},L,B)$ as \textit{banded unlink diagram} if $(L,B)$ is a banded link (without singular points), $L$ is the unlink in $S^3\setminus L_1$, and $L_B$ is the unlink in $M_{\mathcal{K}}$.
\end{definition}

\begin{remark}
The singular banded link $(L,B)$ is in $S^3\setminus \mathcal{K}$, so it is also in $\partial M_{\mathcal{K}}$ because $S^3\setminus\mathcal{K}\subset\partial M_{\mathcal{K}}$ by the construction of $M_\mathcal{K}$. Therefore, we may view $L^{+}_B$ as a link in $\partial M_{\mathcal{K}}\cong\#^k(S^1\times S^2)$.
\end{remark}

\begin{example}\label{exercise: singular banded unlink diagram}
Let $(\mathcal{K},L,B)=(\mathcal{K},J_1\cup J_2, B_1\cup B_2)$ be the diagram in the left of \autoref{fig: cobordism from S^4 to non-simply connected homology sphere}, where $\mathcal{K}$ is the black $0$-framed Hopf link, $J_1$ is the red unknot, $B_1=\emptyset$, $J_2$ is the blue $3$-component unlink, and $B_2$ is the set of blue bands attached to $J_2$. We can verify that
    \begin{enumerate}
        \item $(\mathcal{K},J_1,B_1)$ is a banded unlink diagram,
        \item $(\mathcal{K},J_2,B_2)$ is a banded unlink diagram,
        \item $(\mathcal{K},L,B)$ is a singular banded unlink diagram.
    \end{enumerate}
\end{example}

\begin{definition}\label{def: realizing surface}
Let $(\mathcal{K},L,B)$ be a singular banded unlink diagram, where $\mathcal{K}=L_1\cup L_2\subset S^3$ and $\partial M_{\mathcal{K}}\cong\#^k(S^1\times S^2)$. Let $C_{L}$ be the collection of the companion disks of $L$. Let $D_{L^{-}}\subset S^3$ be the collection of the trivial disks bounded by the unlink $L^{-}\subset S^3\setminus L_1$, and let $D_{L^+_B}\subset \partial M_{\mathcal{K}}$ be the collection of the trivial disks bounded by $L^+_B\subset S^3\setminus \mathcal{K}\subset\partial M_\mathcal{K}$.  

Consider a diffeomorphism $\phi:S^3\times[-1,0]\rightarrow N\subset B^4$ from $S^3\times[-1,0]$ to a collar $N$ of $B^4$ such that $\phi(x,0)=x$ for every $x\in S^3=\partial B^4$, i.e., $\partial B^4$ is identified with $S^3\times\{0\}$. We define a properly immersed surface $F$ in $S^3\times[-1,0]$ as follows:
    \begin{equation*}
    F \cap (S^3\times\{t\})=
        \begin{cases}
        L^+_B\times \{t\} & \hspace{5mm} t\in (-\frac{1}{3},0] \\
        (L^+\cup B)\times \{t\} & \hspace{5mm} t=-\frac{1}{3} \\
        L^+\times\{t\} & \hspace{5mm} t \in (-\frac{2}{3},-\frac{1}{3}) \\
        (L^-\cup C_L)\times \{t\} & \hspace{5mm} t=-\frac{2}{3} \\
        L^- \times \{t\} & \hspace{5mm} t\in [-1,-\frac{2}{3}).
        \end{cases}
    \end{equation*}
    
Then $F$ is a properly immersed surface in $S^3\times[-1,0]$ with two boundary components \[(L^{-}\times\{-1\})\coprod (L^{+}_B\times\{0\})\subset S^3\times\{-1\}\coprod S^3\times\{0\},\] and with isolated transverse self-intersections contained in $S^3\times\{-\frac{2}{3}\}$. Define a properly immersed surface in $B^4$ by \[S(\mathcal{K},L,B)=\phi(F\cup (D_{L^{-}}\times\{-1\}))\subset B^4.\] 
Here, $F\cup (D_{L^{-}}\times\{-1\})$ is obtained from $F$ by capping off $L^{-}\times\{-1\}$ with trivial disks $D_{L^-}\times\{-1\}$.

Note that $\partial S(\mathcal{K},L,B)=L^+_B$ in $S^3\setminus\mathcal{K}$, so $S(\mathcal{K},L,B)$ is also properly immersed in $M_{\mathcal{K}}$, where $M_{\mathcal{K}}$ is obtained from $B^4$ by carving out the properly embedded trivial disks in $B^4$ bounded by $L_1$ and attaching $2$-handles along $L_2$. Since $L^+_B$ is the trivial link in $\partial M_{\mathcal{K}}$, we may regard $S(\mathcal{K},L,B)$ as properly immersed in $\widehat{M_{\mathcal{K}}}^\circ=\widehat{M}_\mathcal{K}\setminus \text{4-handle}$, where $\widehat{M_\mathcal{K}}$ is obtained from $M_{\mathcal{K}}$ by attaching $k$ $3$-handles along the $2$-spheres $\coprod^k(\{x_0\}\times S^2)\subset \#^k(S^1\times S^2)\cong\partial M_{\mathcal{K}}$ and a $4$-handle. The $3$-handles can be attached so that $L^+_B$ is still the trivial link in $\partial \widehat{M_\mathcal{K}}^\circ\cong S^3$.

Finally, define an immersed surface in $\widehat{M_{\mathcal{K}}}$ by \[\widehat{S(\mathcal{K},L,B)}=S(\mathcal{K},L,B)\cup D_{L^{+}_B}\subset \widehat{M_{\mathcal{K}}}.\]
\end{definition}

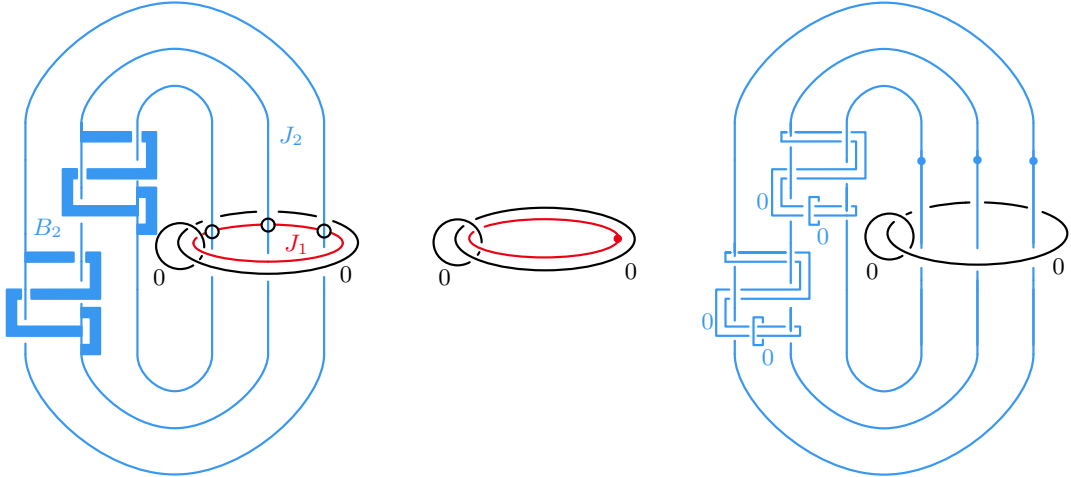
\begin{figure}[ht!]
\labellist
\small\hair 2pt
\pinlabel{$0$} at 155 190
\pinlabel{$0$} at 325 192
\pinlabel {\textcolor[RGB]{56,151,241}{$B_2$}}  at 53 235
\pinlabel {\textcolor[RGB]{56,151,241}{$J_2$}}  at 275 320
\pinlabel {\textcolor[RGB]{236,0,20}{$J_1$}}  at 280 220
\pinlabel{$0$} at 585 195
\pinlabel{$0$} at 415 190
\pinlabel{$0$} at 805 195
\pinlabel{$0$} at 975 200
\pinlabel {\textcolor[RGB]{56,151,241}{$0$}}  at 760 227
\pinlabel {\textcolor[RGB]{56,151,241}{$0$}}  at 705 260
\pinlabel {\textcolor[RGB]{56,151,241}{$0$}}  at 710 116
\pinlabel {\textcolor[RGB]{56,151,241}{$0$}}  at 655 149
\endlabellist
\centering
\includegraphics[width=0.9\textwidth]{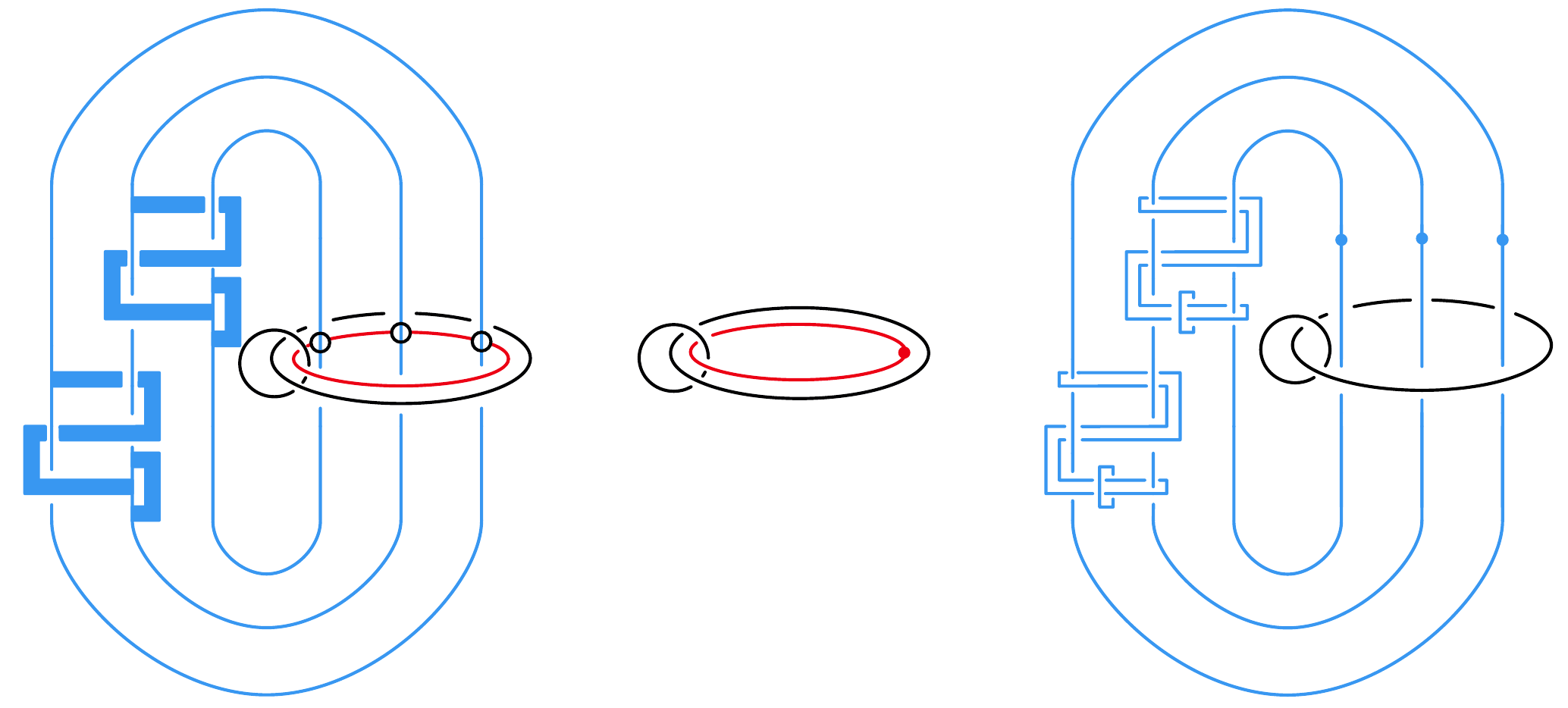}
\caption{\textbf{Left}: A Heegaard diagram $(\Sigma,\alpha,\beta)$ of a $5$-dimensional cobordism from $S^4$ to a non-simply connected homology $4$-sphere, consisting of a $2$-handle and a $3$-handle that are algebraically but not geometrically cancelled. Alternatively, it can be interpreted as a Heegaard diagram of a contractible $5$-manifold with a $0$-handle, a $2$-handle, and a $3$-handle, which is not homeomorphic to $B^5$. Here, $\Sigma$ (in black) represents $S^2\times S^2$, $\alpha$ (in red) is the belt sphere of the $2$-handle representing $\{x_0\}\times S^2\subset S^2\times S^2$, and $\beta$ (in blue) is the attaching sphere of the $3$-handle representing a $2$-knot homotopic but not isotopic to $S^2\times \{y_0\}\subset S^2\times S^2$. \textbf{Middle}: A Kirby diagram of $\Sigma(\alpha)$, which is diffeomorphic to $S^4$. \textbf{Right}: A Kirby diagram of $\Sigma(\beta)$, which is diffeomorphic to the non-simply connected homology $4$-sphere. }
\label{fig: cobordism from S^4 to non-simply connected homology sphere}
\end{figure}

\begin{remark}\label{rem: realizing surface properties}\hfill
    \begin{enumerate}
        \item $S(\mathcal{K},L,B)$ is properly immersed in $B^4$, $M_{\mathcal{K}}$, and $\widehat{M_{\mathcal{K}}}^\circ$.
        \item $D_{L^+_B}$ is embedded in $\partial M_{\mathcal{K}}$ and $\partial \widehat{M_{\mathcal{K}}}^\circ$.
        \item If $(\mathcal{K},L,B)$ is a banded unlink diagram, then $\widehat{S(\mathcal{K},L,B)}$ is an embedded surface in $\widehat{M_{\mathcal{K}}}$.
        \item The Euler characteristic is $\chi(S(\mathcal{K},L,B))=|L^{-}|-|B|+|L^{+}_B|$. 
    \end{enumerate}
\end{remark}

\begin{definition}
Let $(\mathcal{K},L,B)$ and $(\mathcal{K},L',B')$ be singular banded unlink diagrams, where $\mathcal{K}=L_1\cup L_2\subset S^3$. We say that $(\mathcal{K},L,B)$ and $(\mathcal{K},L',B')$ are related by \textit{singular band moves} if $(\mathcal{K},L',B')$ is obtained from  $(\mathcal{K},L,B)$ by a sequence of moves in \autoref{fig: Band moves} and \autoref{fig: Singular moves}. These moves include:
    \begin{enumerate}
        \item Isotopy in $S^3\setminus\mathcal{K}$,
        \item Cup/cap moves,
        \item Band slides,
        \item Band swims,
        \item Slides of bands over components of $L_2$,
        \item Swims of bands about $L_2$,
        \item Slides of unlinks and bands over $L_1$,
        \item Sliding a vertex over a band,
        \item Passing a vertex past the edge of a band,
        \item Swimming a band through a vertex.
    \end{enumerate} We refer to moves $(1)-(7)$, which do not involve the self-intersections of $L$,
as \textit{band moves} (omitting the word ``singular"). The remaining moves $(8)-(10)$ are specific to interactions between singular points and bands.
\end{definition}

\begin{figure}[ht!]
\labellist
\small\hair 2pt
\pinlabel{$(1)$ isotopy in $S^3\setminus\mathcal{K}$} at 205 550
\pinlabel{cap} at 153 410
\pinlabel{cup} at 253 410
\pinlabel{$(2)$} at 153 370
\pinlabel{$(2)$} at 253 370

\pinlabel{band} at 205 285
\pinlabel{slide} at 205 260
\pinlabel{$(3)$} at 205 220

\pinlabel{band} at 205 130
\pinlabel{swim} at 205 105
\pinlabel{$(4)$} at 205 65

\pinlabel{band/} at 648 620
\pinlabel{$2$-handle} at 648 595
\pinlabel{slide} at 648 570
\pinlabel{$(5)$} at 648 529
\pinlabel{$n$} at 584 565
\pinlabel{$n$} at 780 565
\pinlabel\rotatebox{78}{$n$} at 763 525

\pinlabel{band/} at 648 456
\pinlabel{$2$-handle} at 648 431
\pinlabel{swim} at 648 406
\pinlabel{$(6)$} at 648 365
\pinlabel{$n$} at 535 345
\pinlabel{$n$} at 790 345

\pinlabel{$(7)$} at 648 210
\pinlabel{slides over dotted circles} at 648 170

\pinlabel{$(7)$} at 648 60
\endlabellist
\centering
\includegraphics[width=1\textwidth]{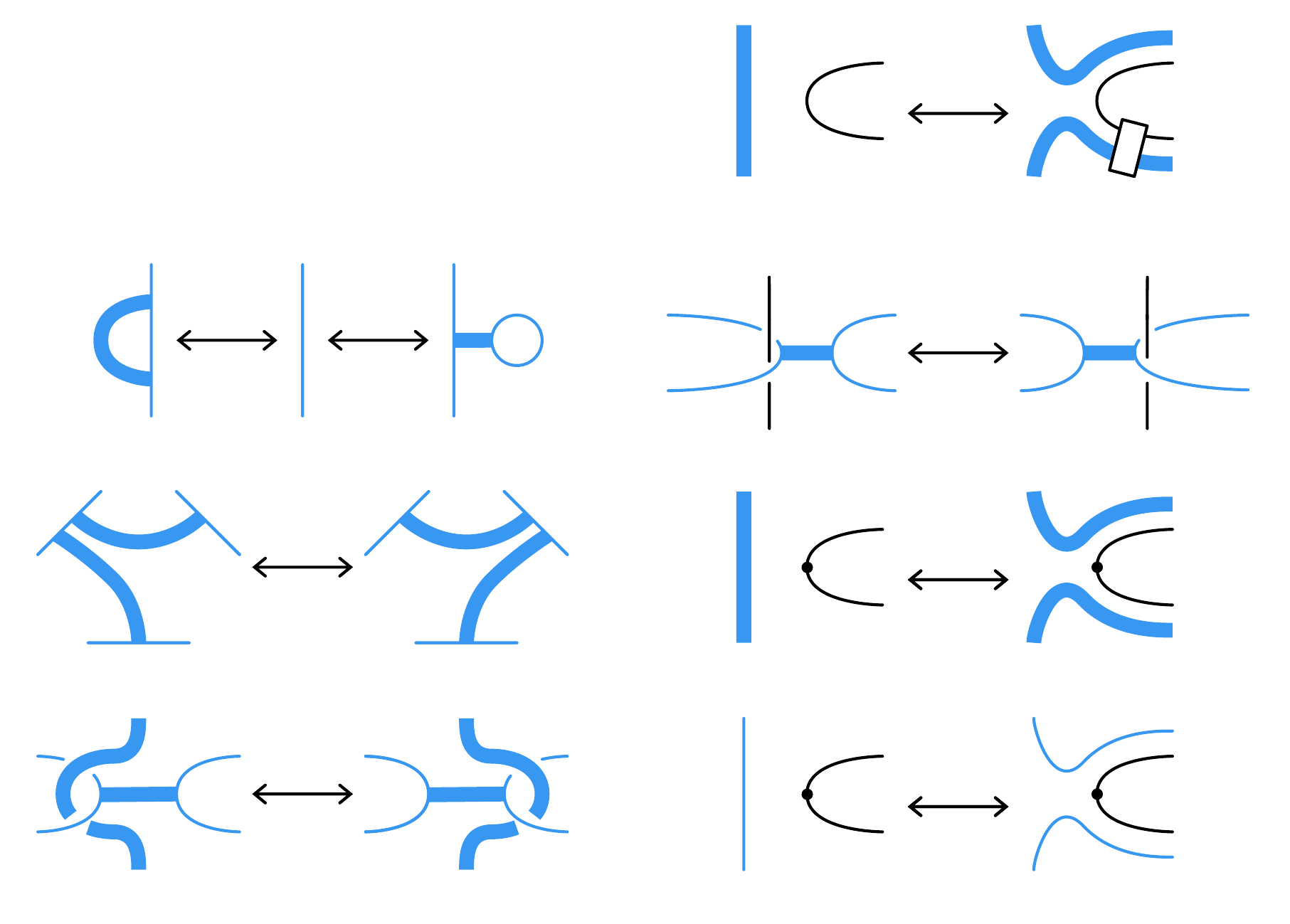}
\caption{Singular band moves without singular points.}
\label{fig: Band moves}
\end{figure}
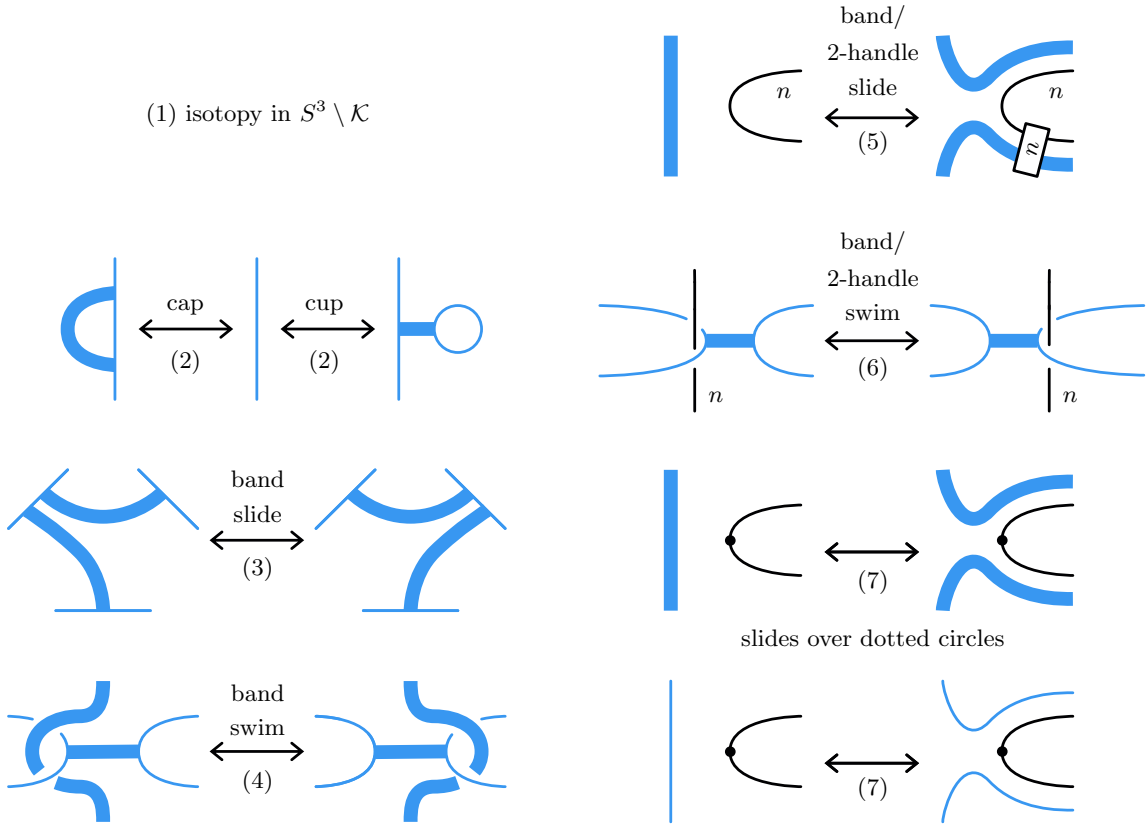

\begin{figure}[ht!]
\labellist
\small\hair 2pt
\pinlabel{intersection/band} at 205 390
\pinlabel{slide} at 205 370
\pinlabel{$(8)$} at 205 340

\pinlabel{intersection/band} at 205 237
\pinlabel{swim} at 205 217
\pinlabel{$(9)$} at 205 187

\pinlabel{intersection/band} at 205 90
\pinlabel{pass} at 205 70
\pinlabel{$(10)$} at 205 40
\endlabellist
\centering
\includegraphics[width=0.55\textwidth]{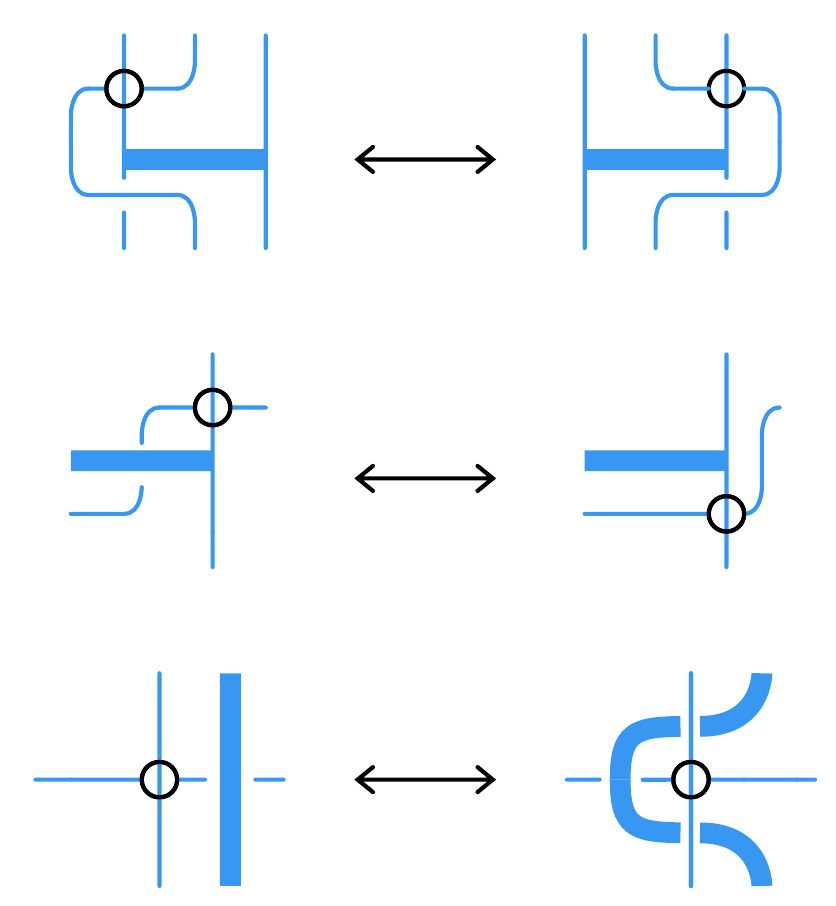}
\caption{Singular band moves with singular points.}
\label{fig: Singular moves}
\end{figure}

\begin{theorem}[\cite{hughes2020isotopies},\cite{hughes2021band}]\label{thm: existence of singular banded unlink diagram}
Let $(Y,F)$ be a pair, where $Y$ is a closed, connected, orientable $4$-manifold and $F\subset Y$ is an immersed surface. Then there exists a singular banded unlink diagram $(\mathcal{K},L,B)$ such that $(Y,F)$ is diffeomorphic to $(\widehat{M_{\mathcal{K}}},
\widehat{S(\mathcal{K},L,B)})$.
\end{theorem}

\begin{theorem}[\cite{hughes2020isotopies},\cite{hughes2021band}]
Let $(\mathcal{K},L,B)$ and $(\mathcal{K},L',B')$ be singular banded unlink diagrams. Then $\widehat{S(\mathcal{K},L,B)}$ and $\widehat{S(\mathcal{K},L',B')}$ are isotopic in $\widehat{M_{\mathcal{K}}}$ if and only if they are related by singular band moves.
\end{theorem}

\begin{theorem}[\cite{hughes2020isotopies},\cite{hughes2021band}]
Let $(\mathcal{K},L,B)$ and $(\mathcal{K}',L',B')$ be singular banded unlink diagrams. Then $(\widehat{M_{\mathcal{K}}},\widehat{S(\mathcal{K},L,B)})$ and $(\widehat{M_{\mathcal{K}'}},\widehat{S(\mathcal{K}',L',B')})$ are diffeomorphic if and only if they are related by Kirby moves and singular band moves.
\end{theorem}

\begin{example}
The pairs $(S^2\tilde{\times}S^2, F)$ and $ (\mathbb{C}P^2\#\overline{\mathbb{C}P^2},\mathbb{C}P^1\#\overline{\mathbb{C}P^1})$ are diffeomorphic, where $F$ denotes a fiber of the non-trivial bundle over $S^2$.
\end{example}

\begin{proof}
The diagram in the left of \autoref{fig: nontrivial s2bundle over s2 with a fiber} is obtained from the one in the right of \autoref{fig: nontrivial s2bundle over s2 with a fiber} by sliding the $(-1)$-framed unknot over the $1$-framed unknot along the obvious band.
\end{proof}

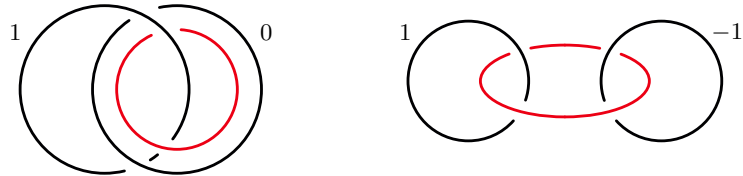
\begin{figure}[ht!]
\labellist
\small\hair 2pt
\pinlabel{$1$} at 30 150
\pinlabel{$0$} at 207 150
\pinlabel{$1$} at 305 150
\pinlabel{$-1$} at 532 150
\endlabellist
\centering
\includegraphics[width=0.65\textwidth]{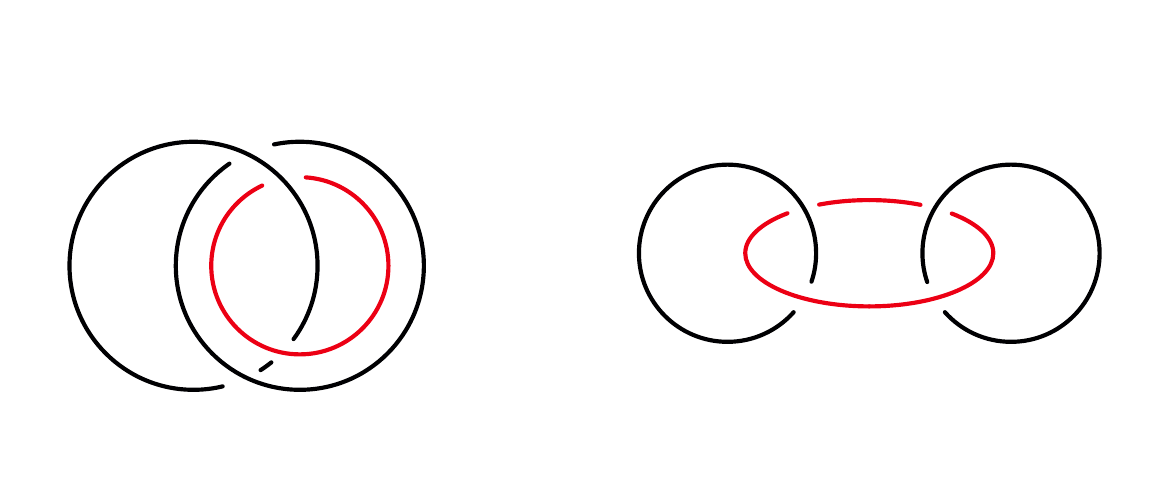}
\caption{$(S^2\tilde{\times}S^2, F)\cong (\mathbb{C}P^2\#\overline{\mathbb{C}P^2},\mathbb{C}P^1\#\overline{\mathbb{C}P^1})$, where $F$ is a fiber of $S^2\tilde{\times}S^2$.}
\label{fig: nontrivial s2bundle over s2 with a fiber}
\end{figure}

\subsection{Kirby diagrams for $1$- and $2$- surgery}\label{sec: Kirby diagrams for surgeries}

We illustrate how to obtain a Kirby diagram of $1$- and $2$-surgery on a $4$-manifold. Note that $1$-surgery corresponds to attaching a $5$-dimensional $2$-handle, and $2$-surgery corresponds to attaching a $5$-dimensional $3$-handle.

We begin by describing how to obtain a Kirby diagram of $1$-surgery on an arbitrary $4$-manifold using a pair of a Kirby diagram and an embedded circle in the Kirby diagram.

\begin{proposition}[A Kirby diagram of $1$-surgery]\label{pro: Kirby diagram for 1-surgeries}
Let $(Y,\gamma)$ be a pair, where $Y$ is a closed $4$-manifold and $\gamma\subset Y$ is an embedded circle. Let $(\mathcal{K},c)$ be a pair, where $\mathcal{K}$ is a Kirby diagram of $Y$ and $c\subset S^3\setminus\mathcal{K}$ is an embedded circle representing $\gamma$. Then we can obtain a Kirby diagram $\mathcal{K}'$ of the $1$-surgery \[Y(\gamma)=\overline{Y\setminus \nu(\gamma)}\cup (B^2\times S^2)\] on $Y$ along $\gamma$ by following these steps:
    \begin{enumerate}
        \item Start with the pair $(\mathcal{K},c)$; see the top left of \autoref{fig: Kirby diagram for 1-surgery }.
        \item Add a cancelling $(1,2)$-pair to $\mathcal{K}$, where the $2$-handle $c$ with one of two possible framings, and the $1$-handle is a dotted meridian $m$ of $c$; see the top right of \autoref{fig: Kirby diagram for 1-surgery }.
        \item Replace the dotted meridian with a $0$-framed $2$-handle; see the bottom left of \autoref{fig: Kirby diagram for 1-surgery }.
    \end{enumerate}
\end{proposition}

\begin{proof}
Introducing a cancelling $(1,2)$-pair in step (2) still represents $Y$. The dot-zero exchange in (3) corresponds to performing $1$-surgery on $Y$ along $\gamma$, removing $S^1\times B^3$ and gluing in $B^2\times S^2$. Since $\pi_1 (SO(3))\cong \mathbb{Z}_2$, the circle $\gamma$ admits two possible framings. Alternatively, any integer framing of $c$ (representing $\gamma$) can be adjusted to $0$ or $1$ by handle slides of $c$ over its meridian $m$.
\end{proof}

\begin{figure}[ht!]
\labellist
\small\hair 2pt
\pinlabel{$0$} at 425 430
 
\pinlabel{$0$} at 187 190
\pinlabel{$0$} at 207 145

\pinlabel{$0$} at 425 190
\pinlabel{$0$} at 445 145
 
\pinlabel{\textcolor[RGB]{253,141,51}{$c$}} at 190 280
 
\pinlabel{\textcolor[RGB]{56,151,241}{$b$}} at 260 40
\endlabellist
\centering
\includegraphics[width=0.55\textwidth]{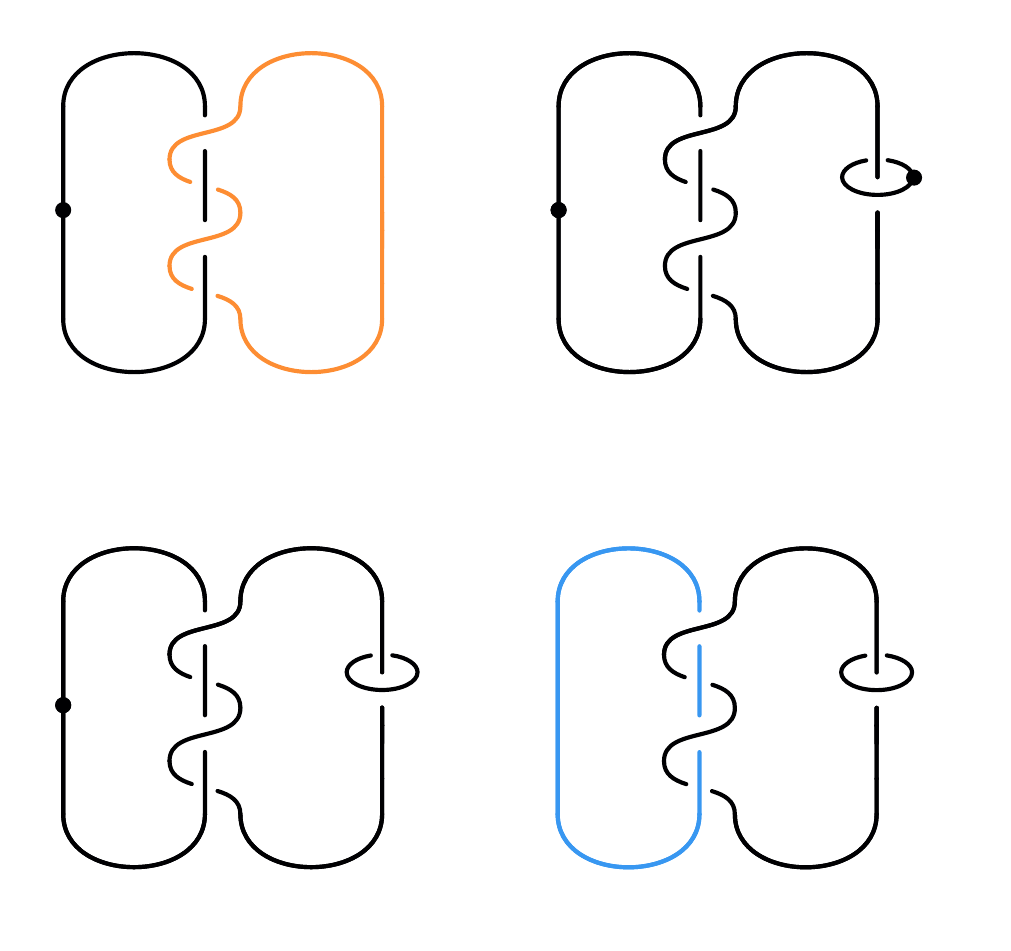}
\caption{\textbf{Top left}: A Kirby diagram $\mathcal{K}$ of $S^1\times S^3$ and an orange circle $c$ representing $2\in\mathbb{Z}\cong \pi_1(S^1\times S^3)$. \textbf{Top right}: A new Kirby diagram of $S^1\times S^3$ obtained by introducing a cancelling $(1,2)$-pair. \textbf{Bottom left}: A Kirby diagram $\mathcal{K}'$ of the $1$-surgery on $S^1\times S^3$ along $\gamma$ representing $2\in\mathbb{Z}\cong \pi_1(S^1\times S^3)$ with the trivial framing. \textbf{Bottom right}: A banded unlink diagram obtained from the bottom left by replacing the dotted circle with the blue circle $b$. This diagram represents a pair $(S^2\times S^2, J)$, where $J$ is a $2$-knot representing $(2,0)\in \mathbb{Z}\oplus\mathbb{Z}\cong H_2(S^2\times S^2)$. The Kirby diagram of the surgery on $S^2\times S^2$ along $J$ is the bottom left by \autoref{pro: kirby diagram for 2-surgeries}.}
\label{fig: Kirby diagram for 1-surgery }
\end{figure}
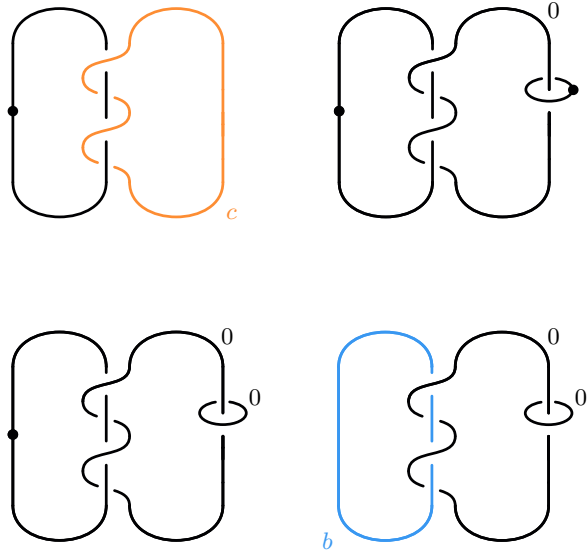

We now explain how to obtain a Kirby diagram of $2$-surgery on an arbitrary $4$-manifold from a banded unlink diagram.

\begin{proposition}[A Kirby diagram of $2$-surgery]\label{pro: kirby diagram for 2-surgeries}
Let $(\mathcal{K},L,B)$ be a banded unlink diagram of a pair $(Y,J)$, where $Y$ is a $4$-manifold and $J\subset Y$ is a $2$-knot with trivial normal bundle. Then we can obtain a Kirby diagram $\mathcal{K}'$ of the $2$-surgery \[Y(J)=\overline{Y\setminus \nu(J)}\cup (B^3\times S^1)\] on $Y$ along $J$ by following these steps:
    \begin{enumerate}
        \item Start with the banded unlink diagram $(\mathcal{K},L,B)$.
        \item Replace the unlink $L$ with a dotted unlink; see the top of \autoref{fig: Kirby diagram for surface complement}.
        \item Replace the bands $B$ with $0$-framed $2$-handles; see the bottom of \autoref{fig: Kirby diagram for surface complement}
    \end{enumerate} Furthermore, if
$\widehat{M_{\mathcal{K}}}=M_{\mathcal{K}}\cup (\natural^{k}(S^1\times B^3))$, then $\widehat{M_{\mathcal{K}'}}=M_{\mathcal{K}'}\cup (\natural^{k+|L_B|}(S^1\times B^3))$, where $|L_B|$ is the number of components of the result $L_B$ of performing surgery on $L$ along $B$. 
\end{proposition}

\begin{proof}
Let $(\mathcal{K},L,B)$ be a banded unlink diagram such that \[(Y,J)\cong(\widehat{M_{\mathcal{K}}},\widehat{S(\mathcal{K},L,B)}),\] where $\mathcal{K}=L_1\cup L_2\subset S^3$ and $\partial M_{\mathcal{K}}\cong \#^k(S^1\times S^2)$. We may assume that \[\widehat{M_{\mathcal{K}}}=M_{\mathcal{K}}\cup_N(\text{$k$ $3$-handles})\cup\text{$4$-handle},\] where the $3$-handles are attached along \[N=\coprod^k(\{x_0\}\times S^2)\subset\#^k(S^1\times S^2)\cong\partial M_{\mathcal{K}},\] and  the $4$-handle is then attached to the resulting $S^3$ boundary. Let $\widehat{M_{\mathcal{K}}}^\circ$ be $\widehat{M_{\mathcal{K}}}$ with the $4$-handle removed, i.e., \[\widehat{M_{\mathcal{K}}}^\circ=\widehat{M_{\mathcal{K}}}\setminus\text{$4$-handle}.\]

We have \[\widehat{S(\mathcal{K},L,B)}=S(\mathcal{K},L,B)\cup D_{L_B},\] where $S(\mathcal{K},L,B)$ is properly embedded in $ \widehat{M_{\mathcal{K}}}^\circ$, and $D_{L_B}$ is the collection of the trivial disk bounded by the link $L_B$ in the boundary $3$-sphere; see \autoref{def: realizing surface} and \autoref{rem: realizing surface properties}.

The complement \[\overline{\widehat{M_{\mathcal{K}}}^\circ\setminus\nu(S(\mathcal{K},L,B))}\] is obtained from \[\overline{B^4\setminus\nu(S(\mathcal{K},L,B))}\] by carving out a regular neighborhood of the collection of the properly embedded trivial disks $D_{L_1}'\subset B^4$ bounded by the dotted unlink $L_1$, then attaching $2$-handles along the framed link $L_2$ and $k$ $3$-handles along $N$. Note that $S(\mathcal{K},L,B)$ can be embedded in each of $B^4,M_{\mathcal{K}}$, and $\widehat{M_{\mathcal{K}}}^\circ$ by the construction of $S(\mathcal{K},L,B)$. 

By Chapter $6.2$ in \cite{gompf20234}, we can construct a Kirby diagram of \[\overline{B^4\setminus \nu(S(\mathcal{K},L,B))}.\] The key idea is that an $i$-handle of $S(\mathcal{K},L,B)$ induces an $(i+1)$-handle in the handle decomposition of the complement $\overline{B^4\setminus \nu(S(\mathcal{K},L,B))}$. The attaching sphere of the $(i+1)$-handle is $\partial(C\times B^1)$, where $C$ is the core of the $i$-handle of $S(\mathcal{K},L,B)$. Thus, the unlink $L$ and bands $B$ induce the dotted unlink and the $0$-framed $2$-handles, respectively; see \autoref{fig: Kirby diagram for surface complement}. 

Similarly, the complement  \[\overline{\widehat{M_{\mathcal{K}}}\setminus\nu(\widehat{S(\mathcal{K},L,B)})}\] is obtained from \[\overline{\widehat{M_{\mathcal{K}}}^\circ\setminus\nu(S(\mathcal{K},L,B))}\] by attaching $3$-handles along $(|L_B|-1)$ $2$-spheres among $|L_B|$ $2$-spheres $\partial(\nu(D_{L_B}))$, the boundary of the thickenings of the $2$-handles of $\widehat{S(\mathcal{K},L,B)}$. 
That is, $3$-handles are first attached along $\partial(\nu(D_{L_B}))$, followed by a $4$-handle, with one of the $3$-handles being cancelled by the $4$-handle.

The $2$-surgery \[Y(J)=\overline{Y\setminus\nu(J)}\cup (B^3\times S^1)\cong\overline{\widehat{M_{\mathcal{K}}}\setminus\nu(\widehat{S(\mathcal{K},L,B)})}\cup(B^3\times S^1)\] is thus obtained from \[\overline{\widehat{M_{\mathcal{K}}}\setminus\nu(\widehat{S(\mathcal{K},L,B)})}\] by attaching a $3$-handle and a $4$-handle.

Let $\mathcal{K}'$ be the Kirby diagram obtained from $(\mathcal{K},L,B)$ by replacing the unlink $L$ with a dotted unlink and replacing bands $B$ with $0$-framed $2$-handles. Then clearly \[Y(J)\cong M_{\mathcal{K}'}\cup(\natural^{k+|L_B|}(S^1\times B^3)).\] We note that by \cite{laudenbach1972note}, there exists a unique way, up to diffeomorphism, to attach $(k+|L_B|)$ $3$-handles and a $4$-handle to $M_{\mathcal{K}'}$. Therefore, $\mathcal{K}'$ is a Kirby diagram of $Y(J)$.
\end{proof}

\begin{figure}[ht!]
\labellist
\small\hair 2pt
\pinlabel{$0$} at 340 90
\endlabellist
\centering
\includegraphics[width=0.5\textwidth]{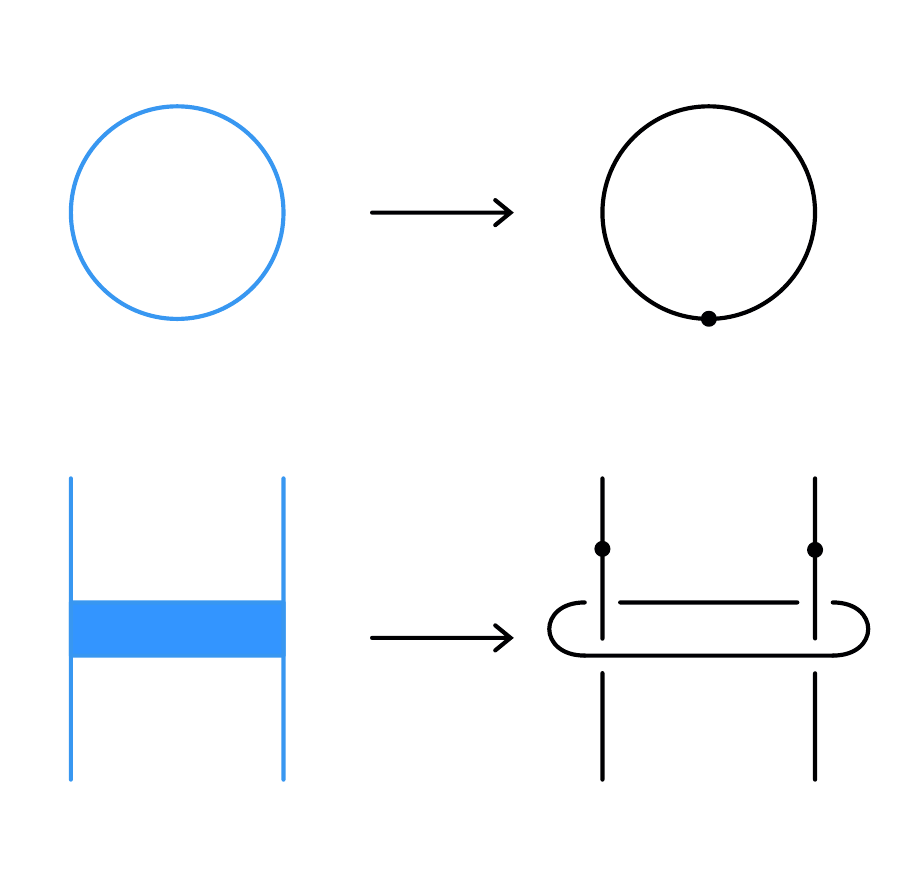}
\caption{An algorithm for constructing a Kirby diagram of $2$-surgery from a banded unlink diagram.}
\label{fig: Kirby diagram for surface complement}
\end{figure}

\begin{example}\hfill
    \begin{enumerate}
        \item The middle of \autoref{fig: cobordism from S^4 to non-simply connected homology sphere} is a Kirby diagram of the surgery $\Sigma(\alpha)$ on $\Sigma$ along $\alpha$. In this diagram, the $(k+|L_B|)$ $3$-handles are not shown, where  $k=0$ and $|L_B|=1$. After removing a cancelling $(1,2)$-pair and a cancelling $(2,3)$-pair, we see that the Kirby diagram represents $S^4$.
        \item The right of \autoref{fig: cobordism from S^4 to non-simply connected homology sphere} is a Kirby diagram of the surgery $\Sigma(\beta)$ on $\Sigma$ along $\beta$. Again, $(k+|L_B|)$ $3$-handles are omitted, where $k=0$ and $|L_B|=3$. This Kirby diagram represents a non-simply connected homology $4$-sphere \cite{kim2025nknotssntimess2contractible}.
        \item If we interpret the top left of \autoref{fig: Kirby diagram for surface complement} as a banded unlink diagram of the unknotted $2$-sphere $U$ in $S^4$, then the top right of the same figure gives a Kirby diagram of the surgery on $S^4$ along $U$, which is diffeomorphic to $S^1\times S^3$.
        \item In the left of \autoref{fig: nontrivial s2bundle over s2 with a fiber}, we obtain a Kirby diagram of the surgery on $S^2\tilde{\times}S^2$ along a fiber by replacing the red circle with a dotted circle. The resulting manifold is diffeomorphic to $S^4$, as shown by removing a cancelling $(1,2)$-pair and a cancelling $(2,3)$-pair. A similar argument applies to the right of \autoref{fig: nontrivial s2bundle over s2 with a fiber}, where surgery on $\mathbb{C}P^2\#\overline{\mathbb{C}P^2}$ along $\mathbb{C}P^1\#\overline{\mathbb{C}P^1}$ also yields $S^4$.
    \end{enumerate}
\end{example}

\section{Heegaard diagrams for 5-manifolds}\label{sec:Heegaard diagrams for 5-manifolds}

We recall that a Heegaard diagram $(\Sigma,\alpha,\beta)$ is a triple, where $\Sigma$ is a closed $4$-manifold and each of $\alpha$ and $\beta$ is a $2$-link with trivial normal bundle (see \autoref{def:Heegaard diagrams}). We can construct a $5$-dimensional cobordism $M_\alpha \cup_\Sigma M_\beta$ from $\Sigma(\alpha)$ to $\Sigma(\beta)$ using only $2$- and $3$-handles, where $\Sigma(\alpha)$ and $\Sigma(\beta)$ are the results of $2$-surgery on $\Sigma$ along $\alpha$ and $\beta$, respectively (see \autoref{def: (2k+1)-dimensional cobordism}). In this construction, $\alpha$ can be regarded as the belt spheres of the $2$-handles, and $\beta$ as the attaching spheres of the $3$-handles.

If $\Sigma(\alpha) \cong \#^k(S^1 \times S^3)$, then we can construct a $5$-dimensional $3$-handlebody $\widehat{M_{\alpha}} \cup_\Sigma M_\beta$ by capping off $\Sigma(\alpha)$ with $\natural^k(S^1 \times B^4)$, where $\natural^k(S^1 \times B^4)$ is considered as the union of a single $0$-handle and $k$ $1$-handles (see \autoref{def: some 5-manifolds from Heegaard diagrams}). The manifold $\widehat{M_\alpha}$ is a $5$-dimensional $2$-handlebody, and there exists a $4$-dimensional $2$-handlebody $Y$ such that $\widehat{M_\alpha} \cong Y \times B^1$ (see \autoref{thm: Heegaard diagrams for 5d 2handlebodies}). Clearly, $\Sigma$ is the double of $Y$. If $\Sigma(\beta) \cong \#^r(S^1 \times S^3)$, then we can also cap off along $\Sigma(\beta)$ with $\natural^r(S^1 \times B^4)$ to obtain a closed $5$-manifold $\widehat{M_\alpha} \cup_\Sigma \widehat{M_\beta}$, where $\natural^r(S^1 \times B^4)$ is considered as the union of $r$ $4$-handles and a single $5$-handle (see \autoref{def: some 5-manifolds from Heegaard diagrams}).

We now show that every $5$-dimensional cobordism with $2$- and $3$-handles, every $5$-dimensional $3$-handlebody, and every closed, connected, orientable $5$-manifold admits a Heegaard diagram.

\Heegaardexistence

\begin{proof}
Let $X$ be a $5$-dimensional cobordism with $2$- and $3$-handles, a $5$-dimensional $3$-handlebody, or closed $5$-manifold. Let $f:X\rightarrow \mathbb{R}$ be a self-indexing Morse function, where $f(\partial_{-}X)=-\frac{1}{2}$ and $f(\partial_{+}X)=\frac{11}{2}$. Define $A$ as the union of the ascending manifolds of the critical points of index $2$, and $B$ as the union of the descending manifolds of the critical points of index $3$. Define the triple: \[(\Sigma,\alpha,\beta)=(f^{-1}(\frac{5}{2}), f^{-1}(\frac{5}{2})\cap A,f^{-1}(\frac{5}{2})\cap B).\]
We consider the following three cases:
\begin{enumerate}
\item Let $X$ be a $5$-dimensional cobordism with only $2$- and $3$-handles. Clearly, $f^{-1}((-\infty,\frac{5}{2}])\cong M_\alpha$ and $f^{-1}([\frac{5}{2},\infty))\cong M_\beta$. Therefore, $(\Sigma,\alpha,\beta)$ is a Heegaard diagram of \[X=f^{-1}((-\infty,\infty))=f^{-1}((-\infty,\frac{5}{2}])\cup f^{-1}([\frac{5}{2},\infty))\cong M_\alpha\cup_\Sigma M_\beta.\]
\item Let $X$ be a $5$-dimensional $3$-handlebody. The sublevel set $f^{-1}((-\infty,\frac{5}{2}])$ decomposes as \[f^{-1}((-\infty,\frac{5}{2}])=f^{-1}((-\infty,\frac{3}{2}])\cup f^{-1}([\frac{3}{2},\frac{5}{2}]),\] where $f^{-1}([\frac{3}{2},\frac{5}{2}])\cong M_\alpha$ and $f^{-1}((-\infty,\frac{3}{2}])\cong\natural^k(S^1\times B^4)$ for some $k\geq0$. By \cite{cavicchioli1993determination}, $\widehat{M_\alpha}$ is diffeomorphic to $f^{-1}((-\infty,\frac{5}{2}])$. Therefore, $(\Sigma,\alpha,\beta)$ is a Heegaard diagram of \[X=f^{-1}((-\infty,\infty))=f^{-1}((-\infty,\frac{5}{2}])\cup f^{-1}([\frac{5}{2},\infty))\cong \widehat{M_\alpha}\cup_\Sigma M_\beta.\] 
\item Let $X$ be a closed $5$-manifold. Similar to the argument in (2), we have $f^{-1}((-\infty,\frac{5}{2}])\cong\widehat{M_\alpha}$ and $f^{-1}([\frac{5}{2},\infty))\cong\widehat{M_\beta}$. Therefore, $(\Sigma,\alpha,\beta)$ is a Heegaard diagram of \[X=f^{-1}((-\infty,\infty))=f^{-1}((-\infty,\frac{5}{2}])\cup f^{-1}([\frac{5}{2},\infty))\cong  \widehat{M_\alpha}\cup_\Sigma \widehat{M_\beta}.\] 
\end{enumerate}
\end{proof}

We introduce several moves such as isotopies, handle slides, stabilizations, and diffeomorphisms defined on Heegaard diagrams. We begin with isotopies.

\begin{definition}\label{def: isotopy on Heegaard diagrams}
Let $(\Sigma,\alpha,\beta)$ and $(\Sigma,\alpha',\beta')$ be Heegaard diagrams. We say that they are related by an \textit{isotopy} if $\alpha$ is isotopic to $\alpha'$ and $\beta$ is isotopic to $\beta'.$
\end{definition}

Next, we introduce handle slides.

\begin{definition}\label{def: handle slides on Heegaard diagrams}
Let $(\Sigma,\alpha,\beta)$ be a Heegaard diagram. Let $\alpha_i, \alpha_j\subset \alpha=\alpha_1\cup\cdots\cup\alpha_m$ be two $2$-knots and $\tilde{\alpha_j}\subset \partial\nu(\alpha_j)$ be a push-off of $\alpha_j.$ A $3$-dimensional submanifold $c\subset \Sigma$ is called a \textit{sliding cylinder connecting} $\alpha_i$ and $\tilde{\alpha_j}$ if there exists an embedding $e:B^1\times B^2\hookrightarrow \Sigma$ such that 
    \begin{enumerate}
        \item $c=e(B^1\times B^2)$,
        \item $c\cap \alpha_i=e(\{-1\}\times B^2)$,
        \item $c\cap \tilde{\alpha_j}=e(\{1\}\times B^2)$,
        \item $e((-1,1)\times B^2)\cap (\alpha\cup \nu(\alpha_j))=\emptyset$.
    \end{enumerate}
We define the cylinder sum as\[\alpha_i\#_c \tilde{\alpha_j}=(\alpha_1\cup \alpha_2)\setminus e(\partial B^1\times B^2)\cup e(B^1\times\partial B^2).\] We call $\alpha_i\#_c\tilde{\alpha_j}$ a \textit{handle slide of $\alpha_i$ over $\alpha_j$} (along $c$).
We say that $(\Sigma,\alpha)$ and $(\Sigma,\alpha')$ are related by a \textit{handle slide} if $\alpha'=(\alpha\setminus \alpha_i) \cup (\alpha_i\#_c \tilde{\alpha_j})$. We say that two Heegaard diagrams $(\Sigma,\alpha,\beta)$ and $(\Sigma,\alpha',\beta')$ are related by a \textit{handle slide} if $(\Sigma,\alpha)$ and $(\Sigma,\alpha')$ are related by a handle slide and $\beta=\beta'$ or $(\Sigma,\beta)$ and $(\Sigma,\beta')$ are related by a handle slide and $\alpha=\alpha'.$ See \autoref{fig: Handle slide on Heeggar diagram}.
\end{definition}

\begin{figure}[ht!]
\labellist
\small\hair 2pt
\pinlabel{$0$} at 425 1530
\pinlabel{$0$} at 310 1630

\pinlabel{$0$} at 425 940
\pinlabel{$0$} at 310 1040

\pinlabel{$0$} at 425 350
\pinlabel{$0$} at 310 450 
\endlabellist
\centering
\includegraphics[width=0.73\textwidth]{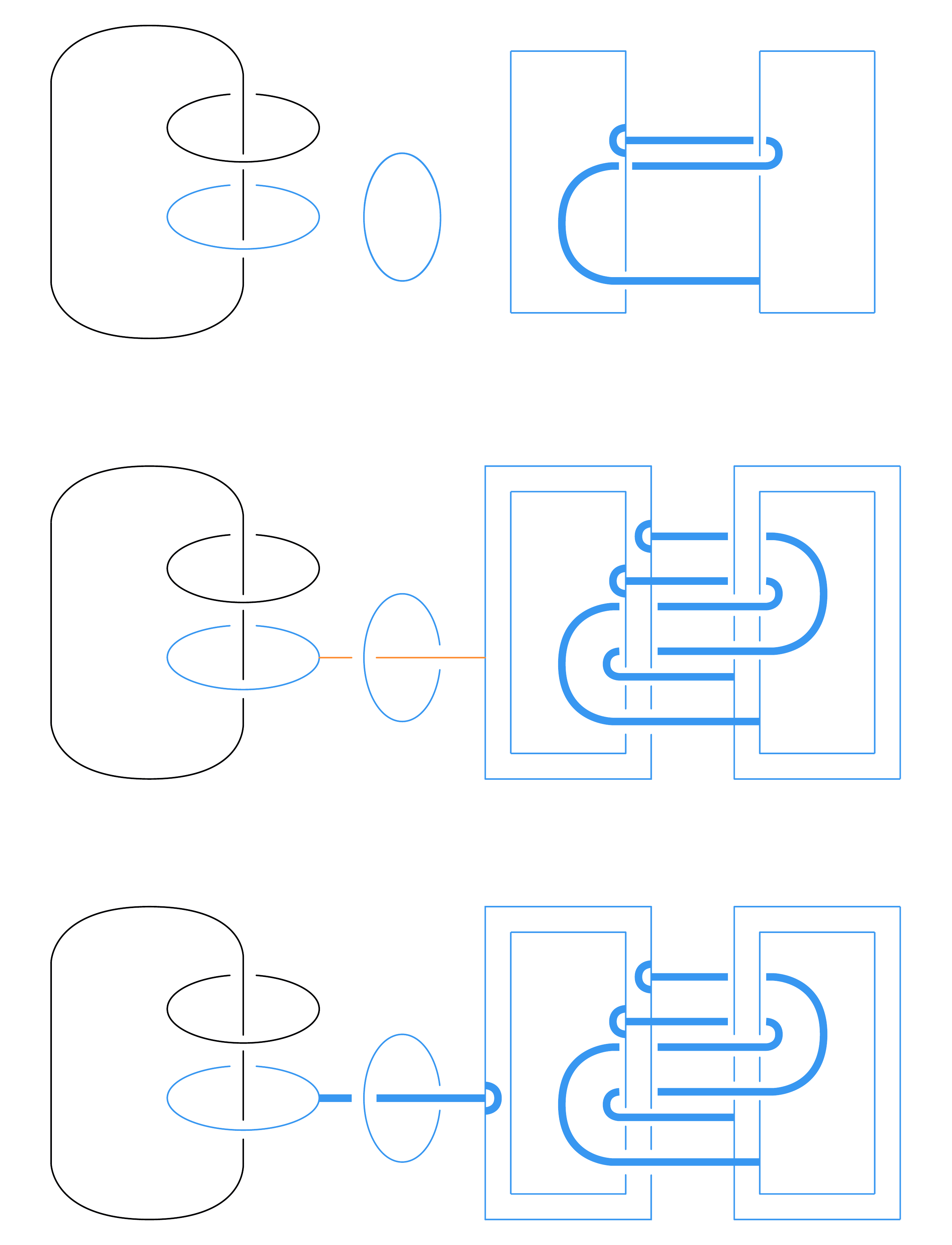}
\caption{From \textbf{top} to \textbf{bottom}: A handle slide of $S^2\times\{y_0\}\subset S^2\times S^2$ over the spun trefoil along a sliding cylinder whose core is the orange arc. The orange arc is the core of a sliding cylinder (a $3$-dimensional $1$-handle) connecting $S^2\times\{y_0\}$ and a parallel push-off of the spun trefoil. For the given orange arc, there are two possible sliding cylinders whose core is the orange arc; see \autoref{fig: Banded unlink diagram for surgery along 3d1h} for the banded unlink diagram of the surgery along the cylinder.} 
\label{fig: Handle slide on Heeggar diagram}
\end{figure}

\begin{figure}[ht!]
\centering
\includegraphics[width=0.55\textwidth]{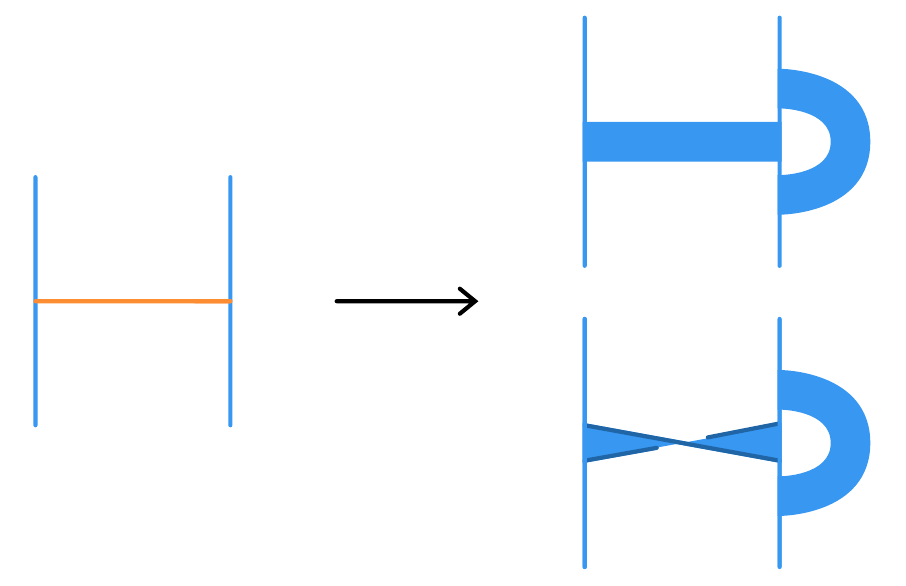}
\caption{\textbf{Left}: An orange arc connecting two different components. \textbf{Right}: There are two possible banded unlink diagrams of the surgery along a sliding cylinder whose core is the orange arc. Surgery along $B^1\times B^2$ is to remove $\{-1,1\}\times B^2$ and glue $B^1\times S^1$, where $B^1\times\{0\}$ is identified with the orange arc. Here, $B^1\times S^1=B^1\times (B^1_{-}\cup B^1_{+})=(B^1\times B^1_{-})\cup (B^1\times B^1_{+})$, where $B^1\times B^1_{-}$ and $B^1\times B^1_{+}$ correspond to a long band and a rainbow band, respectively. If the long band is twisted, we can untwist the long band by sliding it over the rainbow band.}
\label{fig: Banded unlink diagram for surgery along 3d1h}
\end{figure}

We now introduce three types of stabilizations.

\begin{definition}\label{def: stabilizations on Heegaard diagrams}Let $(\Sigma,\alpha,\beta)$ be a Heegaard diagram.
    \begin{enumerate}
        \item  A \textit{first stabilization} of $(\Sigma,\alpha,\beta)$ is the Heegaard diagram \[(\Sigma,\alpha',\beta)=(\Sigma,\alpha\cup U,\beta),\] where $U$ is the trivial $2$-knot in a $4$-ball $B\subset \Sigma$ such that $B\cap \alpha=\emptyset$. We say that $(\Sigma,\alpha,\beta)$ and $(\Sigma,\alpha',\beta)$ are related by a \textit{first stabilization}. See the left of \autoref{fig: stabilizations on a Heegaard diagram}.
        \item A \textit{second stabilization} of $(\Sigma,\alpha,\beta)$ is the Heegaard diagram \[(\Sigma',\alpha',\beta')=(\Sigma\# (S^2\times S^2),\alpha \cup (\{x_0\}\times S^2),\beta\cup (S^2\times\{y_0\}))\] obtained by performing the connected sum of $(\Sigma,\alpha,\beta)$ with $(S^2\times S^2,\{x_0\}\times S^2,S^2\times\{y_0\})$, where $x_0,y_0 \in S^2.$ We say that $(\Sigma,\alpha,\beta)$ and $(\Sigma',\alpha',\beta')$ are related by a \textit{second stabilization}. See the middle of \autoref{fig: stabilizations on a Heegaard diagram}.
        \item A \textit{third stabilization} of $(\Sigma,\alpha,\beta)$ is the Heegaard diagram \[(\Sigma,\alpha,\beta')=(\Sigma,\alpha,\beta \cup U),\] where $U$ is the trivial $2$-knot in a $4$-ball $B\subset \Sigma$ such that $B\cap \beta=\emptyset$. We say that $(\Sigma,\alpha,\beta)$ and $(\Sigma,\alpha,\beta')$ are related by a \textit{third stabilization}. See the right of \autoref{fig: stabilizations on a Heegaard diagram}.
    \end{enumerate}     
\end{definition}

\begin{figure}[ht!]
\labellist
\small\hair 2pt
\pinlabel{$0$} at 170 140
\pinlabel{$0$} at 340 140
\endlabellist
\centering
\includegraphics[width=0.65\textwidth]{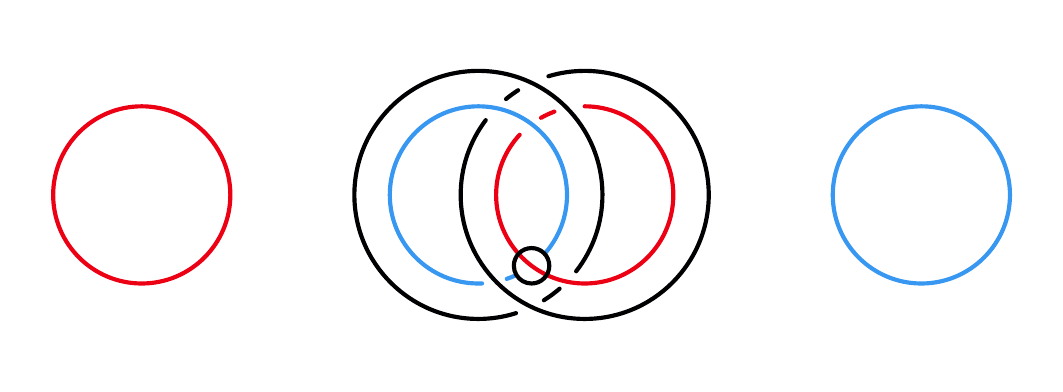}
\caption{Three types of stabilizations. \textbf{Left}: A first stabilization. \textbf{Middle}: A second stabilization. \textbf{Right}: A third stabilization.}
\label{fig: stabilizations on a Heegaard diagram}
\end{figure}

\begin{remark}
We note that for a first stabilization, we do not draw a $4$-handle that is cancelled by the $3$-handle attached along the trivial $2$-knot $U$, i.e., $M_\alpha\cong M_{\alpha'}\; \cup $  $4$-handle. More precisely, the $4$-handle is attached along the obvious $3$-sphere $\{x_0\}\times S^3\subset \Sigma(\alpha)\#(S^1\times S^3)\cong \Sigma(\alpha')$, where $x_0\in S^1$. Similarly, for a third stabilization, we omit the $4$-handle that cancels the corresponding $3$-handle attached along the trivial $2$-knot $U$.
\end{remark}

\begin{remark}
The definitions of isotopy, handle slide, and (first, second, and third) stabilization defined on Heegaard diagrams correspond to the original definitions of isotopy of a handle, handle slide, and cancelling $(1,2)$-, $(2,3)$-, and $(3,4)$- pairs in dimension $5$.
\end{remark}

Finally, we introduce diffeomorphisms of Heegaard diagrams.

\begin{definition}\label{def: diffeomorphisms on Heggaard diagrams}
Let $(\Sigma,\alpha,\beta)$ and $(\Sigma',\alpha',\beta')$ be Heegaard diagrams. We say that they are related by a \textit{diffeomorphism} if there exists a diffeomorphism $\phi:\Sigma \rightarrow \Sigma'$ sending $\alpha$ to $\alpha'$ and $\beta$ to $\beta'$.
\end{definition}

A Heegaard diagram is unique up to the moves defined above.

\Heegaardmoves

\begin{proof}
Let $(\Sigma,\alpha,\beta)$ and $(\Sigma',\alpha',\beta')$ be Heegaard diagrams. We will prove parts $(1)$, $(2)$, and $(3)$ using similar arguments. The ``only if" direction in each part follows from \autoref{thm: Cerf's moves} \cite{cerf1970stratification}. For the ``if" direction, we again apply Cerf's theorem and use the fact that the attachment of a $5$-dimensional $3$-handle is determined by its attaching $2$-sphere since $\pi_2(SO(2))\cong 1$. In parts $(2)$ and $(3)$, we also use the result of Cavicchioli and Hegenbarth that every self-diffeomorphism of $\#^k(S^1\times S^3)$ extends to a self-diffeomorphism of $\natural^k(S^1\times B^4)$ \cite{cavicchioli1993determination}.
    \begin{enumerate}
        \item $(\Rightarrow)$ Suppose $M_\alpha\cup_\Sigma M_\beta\cong M_{\alpha'}\cup_{\Sigma'} M_{\beta'}$. Let $\Phi:M_\alpha\cup_\Sigma M_\beta\rightarrow M_{\alpha'}\cup_{\Sigma'} M_{\beta'}$ be a diffeomorphism. Then $(\Phi(\Sigma),\Phi(\alpha),\Phi(\beta))$ is a Heegaard diagram of $M_{\alpha'}\cup_{\Sigma'} M_{\beta'}$. By \cite{cerf1970stratification}, it is related to $(\Sigma',\alpha',\beta')$ by isotopies, handle slides, and (first, second, and third) stabilizations. Therefore, the diagrams $(\Sigma,\alpha,\beta)$ and $(\Sigma',\alpha',\beta')$ are related by isotopies, handle slides, (first, second, and third) stabilizations, and diffeomorphisms.\\ 
        $(\Leftarrow)$ Suppose $(\Sigma,\alpha,\beta)$ and $(\Sigma',\alpha',\beta')$ are related by isotopies, handle slides, and (first, second, and third) stabilizations. Then $M_\alpha\cup_\Sigma M_\beta\cong M_{\alpha'}\cup_{\Sigma'} M_{\beta'}$ by \cite{cerf1970stratification}. Now suppose they are related by a diffeomorphism $\phi:(\Sigma,\alpha,\beta)\rightarrow (\Sigma',\alpha',\beta')$ with $\phi(\alpha)=\alpha'$ and $\phi(\beta)=\beta'$. Then $\phi$ extends to a diffeomorphism $\Phi:\Sigma\times[-1,1]\rightarrow \Sigma'\times[-1,1]$ defined by $\Phi(x,t)=(\phi(x),t)$. This map satisfies $\Phi(\alpha\times\{-1\})=\alpha'\times\{-1\}$ and $\Phi(\beta\times\{1\})=\beta'\times\{1\}$. Since the attaching map of a $5$-dimensional $3$-handle is determined by its attaching $2$-sphere (as $\pi_2(SO(2))=1$), $\Phi$ extends to a diffeomorphism $\tilde{\Phi}: M_\alpha\cup_\Sigma M_\beta\rightarrow M_{\alpha'}\cup_{\Sigma'} M_{\beta'}$. Therefore, $M_{\alpha}\cup_\Sigma M_{\beta}\cong M_{\alpha'}\cup_{\Sigma'} M_{\beta'}$.
        \item $(\Rightarrow)$ Suppose $\widehat{M_\alpha}\cup_\Sigma M_\beta\cong \widehat{M_{\alpha'}}\cup_{\Sigma'} M_{\beta'}$. Let $\Phi:\widehat{M_\alpha}\cup_\Sigma M_\beta\rightarrow \widehat{M_{\alpha'}}\cup_{\Sigma'} M_{\beta'}$ be a diffeomorphism. Then $(\Phi(\Sigma),\Phi(\alpha),\Phi(\beta))$ is a Heegaard diagram of $\widehat{M_{\alpha'}}\cup_{\Sigma'} M_{\beta'}$. As in part $(1)$, the diagrams $(\Sigma,\alpha,\beta)$ and $(\Sigma',\alpha',\beta')$ are related by isotopies, handle slides, (first, second, and third) stabilizations, and diffeomorphisms.\\ 
        $(\Leftarrow)$ Suppose $(\Sigma,\alpha,\beta)$ and $(\Sigma',\alpha',\beta')$ are related by isotopies, handle slides, (first, second, and third) stabilizations, and diffeomorphisms. As in part $(1)$, we have $\widehat{M_\alpha}\cup_\Sigma M_\beta\cong \widehat{M_{\alpha'}}\cup_{\Sigma'} M_{\beta'}$. Note that the original cobordisms $M_\alpha\cup_\Sigma M_\beta$ and $M_{\alpha'}\cup_{\Sigma'} M_{\beta'}$ may not be diffeomorphic because $\Sigma(\alpha)\cong\#^k(S^1\times S^3)$ and $\Sigma'(\alpha')\cong\#^{k'}(S^1\times S^3)$ are not diffeomorphic when $k\neq k'$. However, each boundary component $\Sigma(\alpha)$ and $\Sigma'(\alpha')$ can be capped off uniquely up to diffeomorphism by \cite{cavicchioli1993determination}, so the capped off cobordisms are diffeomorphic, i.e., $\widehat{M_\alpha}\cup_\Sigma M_\beta\cong \widehat{M_{\alpha'}}\cup_{\Sigma'} M_{\beta'}$.
        \item $(\Rightarrow)$ Suppose $\widehat{M_\alpha}\cup_\Sigma \widehat{M_\beta}\cong \widehat{M_{\alpha'}}\cup_{\Sigma'} \widehat{M_{\beta'}}$. Let $\Phi:\widehat{M_\alpha}\cup_\Sigma \widehat{M_\beta}\rightarrow \widehat{M_{\alpha'}}\cup_{\Sigma'} \widehat{M_{\beta'}}$ be a diffeomorphism. Then $(\Phi(\Sigma),\Phi(\alpha),\Phi(\beta))$ is a Heegaard diagram of $\widehat{M_{\alpha'}}\cup_{\Sigma'} \widehat{M_{\beta'}}$. As in parts $(1)$ and $(2)$, the diagrams are related by isotopies, handle slides, (first, second, and third) stabilizations, and diffeomorphisms.\\ 
        $(\Leftarrow)$ Suppose $(\Sigma,\alpha,\beta)$ and $(\Sigma',\alpha',\beta')$ are related by isotopies, handle slides, and (first, second, and third) stabilizations. As in part $(2)$, the cobordisms $M_\alpha\cup_\Sigma M_\beta$ and $M_{\alpha'}\cup_{\Sigma'} M_{\beta'}$ may not be diffeomorphic, but after capping off the boundary components, we obtain $\widehat{M_\alpha}\cup_\Sigma \widehat{M_\beta}\cong \widehat{M_{\alpha'}}\cup_{\Sigma'} \widehat{M_{\beta'}}$.
    \end{enumerate}
\end{proof}

\begin{remark} We recall \autoref{cor: Heegaard 4-manifolds can be used to distinguish 5-manifolds}, which follows immediately from the theorem above. This highlights that the Heegaard $4$-manifold $\Sigma$ can be used to distinguish $5$-manifolds, in contrast to the classical Heegaard surface, which cannot distinguish $3$-manifolds.
\end{remark}

The following corollary shows that the fundamental group of a $5$-manifold can be computed directly from its Heegaard diagrams.

\begin{corollary}
Let $(\Sigma,\alpha,\beta)$ be a Heegaard diagram of a $5$-dimensional $3$-handlebody or a closed $5$-manifold $X$. Then $\pi_1(X)\cong\pi_1(\Sigma)$.
\end{corollary}

\begin{proof}
    The fundamental group $\pi_1(X)$ is determined by its $5$-dimensional $2$-handlebody $\widehat{M_\alpha}\subset X$, so $\pi_1(X)\cong\pi_1(\widehat{M_\alpha})$. By \autoref{thm: Heegaard diagrams for 5d 2handlebodies}, we have $\widehat{M_\alpha}\cong Y\times B^1$ for some $4$-dimensional $2$-handlebody, and $\Sigma=\partial\widehat{M_\alpha}\cong DY$, where $DY$ is the double of $Y$. Since $\pi_1(Y)\cong \pi_1(DY)$ by the construction of the double, it follows that \[\pi_1(X)\cong \pi_1(\widehat{M_\alpha})\cong\pi_1(Y\times B^1)\cong\pi_1(Y)\cong\pi_1(DY)\cong\pi_1(\Sigma).\]
\end{proof}

\begin{remark}
The homology of a $5$-manifold can also be determined from its Heegaard diagram $(\Sigma,\alpha,\beta)$ using $(9)$ in \autoref{rem: basics for handle decompositions} since the intersections between $\alpha$ and $\beta$ in $\Sigma$ are encoded in the diagram; see \autoref{fig: cobordism from S^4 to non-simply connected homology sphere}, \autoref{fig: stabilizations on a Heegaard diagram}, and \autoref{fig: Wu manifold}.
\end{remark}

The following examples illustrate various constructions of $5$-manifolds using Heegaard diagrams.

\begin{example}\hfill
    \begin{enumerate}
        \item Let $(\Sigma, \alpha, \beta) = (\Sigma, \emptyset, \emptyset)$ be a Heegaard diagram. Then $M_\alpha \cup_\Sigma M_\beta \cong \Sigma \times [-1,1]$, which is the identity cobordism from $\Sigma$ to itself.
        \item Let $(\Sigma, \alpha, \beta) = (\Sigma, \alpha, \emptyset)$. Then $M_\alpha \cup_\Sigma M_\beta \cong M_\alpha$, which is a cobordism from $\Sigma(\alpha)$ to $\Sigma$.
        \item Let $(\Sigma, \alpha, \beta) = (\Sigma, \emptyset, \beta)$. Then $M_\alpha \cup_\Sigma M_\beta \cong M_\beta$, which is a cobordism from $\Sigma$ to $\Sigma(\beta)$.
        \item Let $(\Sigma, \alpha, \beta) = (S^4, U, \emptyset)$, where $U$ is the trivial $2$-knot in $S^4$. Then $M_\alpha \cup_\Sigma M_\beta$ is diffeomorphic to a once-punctured $S^3 \times B^2$, where $\Sigma(\alpha) \cong S^1 \times S^3$ and $\Sigma(\beta) \cong S^4$. Let $X = (M_\alpha \cup_\Sigma M_\beta)\cup(B^4 \times B^1)$ be the $5$-manifold obtained by attaching a $4$-handle along $\{x_0\} \times S^3 \subset S^1 \times S^3$. Then $X$ is diffeomorphic to $S^4 \times [-1,1]$, which corresponds to a first stabilization. We have $\widehat{M_\alpha} \cup_\Sigma M_\beta \cong B^5$ and $\widehat{M_\alpha} \cup_\Sigma \widehat{M_\beta} \cong S^5$. See the left of \autoref{fig: stabilizations on a Heegaard diagram}.
        \item Let $(\Sigma, \alpha, \beta) = (S^4, \emptyset, U)$, where $U$ is the trivial $2$-knot in $S^4$. Then $M_\alpha \cup_\Sigma M_\beta$ is diffeomorphic to a once-punctured $S^3 \times B^2$, where $\Sigma(\alpha) \cong S^4$ and $\Sigma(\beta) \cong S^1 \times S^3$. Let $Y = (M_\alpha\cup_\Sigma M_\beta) \cup (B^4 \times B^1)$ be the $5$-manifold obtained by attaching a $4$-handle along $\{x_0\} \times S^3 \subset S^1 \times S^3$. Then $Y$ is diffeomorphic to $S^4 \times [-1,1]$, which corresponds to a third stabilization. We have $\widehat{M_\alpha} \cup_\Sigma M_\beta \cong S^3\times B^2$ and $\widehat{M_\alpha} \cup_\Sigma \widehat{M_\beta} \cong S^5$. See the right of \autoref{fig: stabilizations on a Heegaard diagram}.
        \item Let $(\Sigma, \alpha, \beta) = (S^2 \times S^2, \{x_0\} \times S^2, \emptyset)$. Then $M_\alpha \cup_\Sigma M_\beta$ is diffeomorphic to a once-punctured $S^2 \times B^3$, where $\Sigma(\alpha) \cong S^4$ and $\Sigma(\beta) \cong S^2 \times S^2$. Also, $\widehat{M_\alpha} \cup_\Sigma M_\beta \cong S^2 \times B^3$. See the left of \autoref{fig: nontrivial s2bundle over s2 with a fiber}, where the $1$-framed unknot with a $0$-framed unknot.
        \item Let $(\Sigma, \alpha, \beta) = (S^2 \tilde{\times} S^2, F, \emptyset)$, where $F$ is a fiber of $S^2 \tilde{\times} S^2$. Then $M_\alpha \cup_\Sigma M_\beta$ is diffeomorphic to a once-punctured $S^2 \tilde{\times} B^3$, where $\Sigma(\alpha) \cong S^4$ and $\Sigma(\beta) \cong S^2 \tilde{\times} S^2$. We have $\widehat{M_\alpha} \cup_\Sigma M_\beta \cong S^2 \tilde{\times} B^3$. See the left of \autoref{fig: nontrivial s2bundle over s2 with a fiber}.
        \item Let $(\Sigma, \alpha, \beta) = (\mathbb{C}P^2 \# \overline{\mathbb{C}P^2}, \mathbb{C}P^1 \# \overline{\mathbb{C}P^1}, \emptyset)$. We can easily see that there is a natural diffeomorphism between $(S^2 \tilde{\times} S^2, F)$ and $(\mathbb{C}P^2 \# \overline{\mathbb{C}P^2}, \mathbb{C}P^1 \# \overline{\mathbb{C}P^1})$; see \autoref{fig: nontrivial s2bundle over s2 with a fiber}. Therefore, $M_\alpha \cup_\Sigma M_\beta$ is diffeomorphic to a once-punctured $S^2 \tilde{\times} B^3$, where $\Sigma(\alpha) \cong S^4$ and $\Sigma(\beta) \cong S^2 \tilde{\times} S^2$. We have $\widehat{M_\alpha} \cup_\Sigma M_\beta \cong S^2 \tilde{\times} B^3$. See the right of \autoref{fig: nontrivial s2bundle over s2 with a fiber}.
        \item Let $(\Sigma, \alpha, \beta) = (S^2 \times S^2, \{x_0\} \times S^2, S^2 \times \{y_0\})$. Then $M_\alpha \cup_\Sigma M_\beta$ is diffeomorphic to $S^4 \times [-1,1]$, which corresponds to a second stabilization. We have $\widehat{M_\alpha} \cup_\Sigma M_\beta \cong B^5$ and $\widehat{M_\alpha} \cup_\Sigma \widehat{M_\beta} \cong S^5$. See the middle of \autoref{fig: stabilizations on a Heegaard diagram}.
        \item Let $(\Sigma, \alpha, \beta) = (S^2 \times S^2, \{x_0\} \times S^2, \beta)$, where $\beta$ is a $2$-knot in $S^2 \times S^2$ homotopic but not isotopic to $S^2 \times \{y_0\}$; see the left of \autoref{fig: cobordism from S^4 to non-simply connected homology sphere}. The geometric intersection number between $\alpha$ and $\beta$ is $|\alpha \cap \beta| = 3$, and the algebraic intersection number is $\alpha \cdot \beta = 1$. The author showed in \cite{kim2025nknotssntimess2contractible} that $\widehat{M_\alpha} \cup_\Sigma M_\beta$ is contractible but not homeomorphic to $B^5$ since $(\widehat{M_\alpha} \cup_\Sigma M_\beta) \times B^1 \cong B^6$ and $\partial(\widehat{M_\alpha} \cup_\Sigma M_\beta) = \Sigma(\beta)$ is non-simply connected. 
        
        Then $M_\alpha \cup_\Sigma M_\beta$ is a $5$-dimensional cobordism from the standard $4$-sphere $\Sigma(\alpha) = S^4$ to a non-simply connected homology $4$-sphere $\Sigma(\beta)$, with a single $2$-handle and a single $3$-handle, which are algebraically but not geometrically cancelled. We can obtain Kirby diagrams of $\Sigma(\alpha)$ and $\Sigma(\beta)$ in the middle and right of \autoref{fig: cobordism from S^4 to non-simply connected homology sphere}, respectively, by replacing the unlink with a dotted unlink and the bands with $0$-framed $2$-handles; see \autoref{pro: kirby diagram for 2-surgeries}.

        The middle of the figure represents $S^4$ after removing a cancelling $(1,2)$-pair and a cancelling $(2,3)$-pair. In the right of the figure, we can compute the fundamental group $\pi_1(\Sigma(\beta))$ directly (dotted $1$-handles correspond to generators, and $2$-handles correspond to relations), and the author showed that there exists an epimorphism from $\pi_1(\Sigma(\beta))$ to the alternating group $A_5$. This technique can be generalized to construct contractible $n$-manifolds not homeomorphic to the standard ball $B^n$ for all $n \geq 5$; see \cite{kim2025nknotssntimess2contractible}.

        \item Let $(\Sigma, \alpha, \beta) = (X \# S^2 \tilde{\times} S^2, F, K \# F)$, where $X$ is a closed $4$-manifold, $K$ is a $2$-knot in $X$ with trivial normal bundle, and $F$ is a fiber of $S^2 \tilde{\times} S^2$. Then $M_\alpha \cup_\Sigma M_\beta$ is a $5$-dimensional cobordism from $X$ to the Gluck twist $X_K$ of $X$ along $K$, with a single $2$-handle and a single $3$-handle, where $\Sigma(\alpha) \cong X$ and $\Sigma(\beta) \cong X_K$; see \autoref{sec:Gluck twists} for more details.
    \end{enumerate}
\end{example}

\begin{remark}
    Let $(\Sigma,\alpha,\beta)$ be a Heegaard diagram. By \autoref{thm: existence of singular banded unlink diagram}, there exists a singular banded unlink diagram $(\mathcal{K},L,B)=(\mathcal{K},L_1\cup L_2,B_1\cup B_2)$ such that
    \begin{enumerate}
        \item $(\mathcal{K},L,B)$ is a singular banded unlink diagram of $(\Sigma,\alpha\cup\beta)$, 
        \item $(\mathcal{K},L_1,B_1)$ is a banded unlink diagram of $(\Sigma,\alpha)$, 
        \item $(\mathcal{K},L_2,B_2)$ is a banded unlink diagram of $(\Sigma,\beta)$.
    \end{enumerate}
\end{remark}

We may simply write $(\Sigma,\alpha,\beta)=(\mathcal{K},L_1\cup B_1, L_2\cup B_2)$.

\begin{proposition}[A Heegaard diagram of a $5$-dimensional $2$-handlebody]\label{pro:how to draw Heegaard diagrams for 5d 2handlebodies}
Let $X$ be a $5$-dimensional $2$-handlebody. Then there exists a $4$-dimensional $2$-handlebody $Y$ such that $X \cong Y \times B^1$ by \autoref{thm: Heegaard diagrams for 5d 2handlebodies}. Let $\mathcal{K}$ be a Kirby diagram of $Y$ and $\mathcal{K}'$ be a Kirby diagram of the double $DY$ of $Y$. By \autoref{pro: how to draw a Kirby diagram for the double of a 4-dimensional 2-handlebody}, $\mathcal{K}'$ is obtained from $\mathcal{K}$ by adding $0$-framed meridians to each $2$-handle in $\mathcal{K}$, $3$-handles with the same number of $1$-handles in $\mathcal{K}$, and a single $4$-handle. Add red circles $r$, each parallel to a $0$-framed meridian in $\mathcal{K}'$. Each red circle bounds a properly embedded trivial disk in the $0$-handle and bounds a disk that is a parallel copy of the core disk of a $2$-handle, corresponding to a banded unlink diagram of the belt sphere of each $5$-dimensional $2$-handle of $X$. Then $(\Sigma, \alpha, \beta) = (\mathcal{K}', r, \emptyset)$ is a Heegaard diagram of $X$ and $\widehat{M_\alpha} \cup M_\beta \cong X$. See the middle of \autoref{fig: Mazur manifold} for a Heegaard diagram of $M\times B^1$, where $M$ is the Mazur manifold.
\end{proposition}

We can perform some moves on Heegaard diagrams to show that two $5$-manifolds are diffeomorphic.

\begin{example}
Let $M$ be the Mazur manifold. Then $M\times B^1$ is diffeomorphic to $B^5$.
\end{example}

\begin{proof}
See the right of \autoref{fig: Mazur manifold}.
\end{proof}

\begin{theorem}\label{thm: cobordism from S2 bundles over S2 to the double}
Let $Y$ be a $4$-dimensional $2$-handlebody, and let $DY$ be the double of $Y$. Then there exists a $5$-dimensional cobordism from $(\#^m(S^2\times S^2))\#(\#^n(S^2\Tilde{\times}S^2))$ to $DY$ consisting only $3$-handles for some $m,n\geq 0$.
\end{theorem}

\begin{proof}
Let $\mathcal{K}=L_1\cup L_2$ be a Kirby diagram of $Y$, and let $\mathcal{K}'=L_1\cup L_2\cup J$ be the natural Kirby diagram of $DY$, as described in \autoref{pro: how to draw a Kirby diagram for the double of a 4-dimensional 2-handlebody}, where $J$ is a union of the $0$-framed meridians of $L_2$; see the bottom left of \autoref{fig: Kirby diagram for 1-surgery }. By the construction of $DY$, we have $\widehat{M_{\mathcal{K}'}}=M_{\mathcal{K}'}\cup \natural^{|L_1|}(S^1\times B^3)$, i.e., the number of $1$- and $3$-handles of $DY$ are the same. Replace the dotted link $L_1$ in $\mathcal{K}'$ with blue circles $b$; see the bottom left of \autoref{fig: Kirby diagram for 1-surgery }. Then $\mathcal{K}''=\mathcal{K}'\setminus L_1$ is a union of some Hopf links, each containing a $0$-framed unknotted component. Thus, $\mathcal{K}''$ is a Kirby diagram of $(\#^m(S^2\times S^2))\#(\#^n(S^2\Tilde{\times}S^2))$ for some $m,n\geq 0$. Here, $\widehat{M_{\mathcal{K}''}}=M_{\mathcal{K}''}\cup B^4$. Therefore, $(\mathcal{K}'',\emptyset,b)$ is a Heegaard diagram of a $5$-dimensional cobordism from $(\#^m(S^2\times S^2))\#(\#^n(S^2\Tilde{\times}S^2))$ to $DY$ since the Kirby diagram of the surgery on $(\#^m(S^2\times S^2))\#(\#^n(S^2\Tilde{\times}S^2))$ along $\beta$ is $\mathcal{K}'$ by \autoref{pro: kirby diagram for 2-surgeries}.
\end{proof}

We now provide several examples of Heegaard diagrams for closed $5$-manifolds. By Lawson  \cite{lawson1978decomposing}, every closed $5$-manifold $X$ can be constructed as a twisted double $X=W\cup_{f} W$, where $W$ is a $5$-dimensional $2$-handlebody and $f$ is a self-diffeomorphism of $\partial W$. Note this does not mean that for any self-indexing Morse function $g:X\rightarrow \mathbb{R}$, the two $5$-dimensional $2$-handlebodies $g^{-1}((-\infty,\frac{5}{2}])$ and $g^{-1}([\frac{5}{2},\infty))$ are necessarily diffeomorphic. If the following conjecture holds, then Lawson's theorem would follow as a consequence.

\begin{conjecture}[\cite{kim2025note}]
Let $X$ and $X'$ be $(2k+1)$-dimensional $k$-handlebodies. If their boundaries $\partial X$ and
$\partial X'$ are diffeomorphic, then $X$ and $X'$ are diffeomorphic.
\end{conjecture}

We first explain how to draw a Heegaard diagram of a twisted double of a $5$-dimensional $2$-handlebody in \autoref{pro: how to draw Heegaard diagrams for the twisted double}. We then discuss Barden's classification of simply connected 5-manifolds and present a Heegaard diagram of the Wu manifold, which is a twisted double of a $5$-dimensional $2$-handlebody, i.e. $f\neq id$; see \autoref{fig: Wu manifold}. Here, the Wu manifold is the generator of the $5$-dimensional oriented cobordism group $\Omega^{SO}_5\cong\mathbb{Z}_2$.

\begin{proposition}[A Heegaard diagram of the twisted double of a $5$-dimensional $2$-handlebody]\label{pro: how to draw Heegaard diagrams for the twisted double}
Let $X = W \cup_f W$ be a closed $5$-manifold, where $W$ is a $5$-dimensional $2$-handlebody and $f$ is a self-diffeomorphism of $\partial W$. Then $W = Y \times B^1$ for some $4$-dimensional $2$-handlebody $Y$. Let $C \subset Y$ be the collection of cocores of the $2$-handles of $Y$. The double of the pair $(Y, C)$ is $D(Y, C) = (DY, DC)$, where $DY$ is the double of $Y$ and $DC\subset DY$ is the union of the doubled cocores. This corresponds to the pair $(\partial W, S)$, where $S$ is the collection of belt spheres of the $2$-handles of $W$. Therefore, $(\Sigma, \alpha, \beta) = (DY, DC, f(DC))$ is a Heegaard diagram of $X$.

By \autoref{pro:how to draw Heegaard diagrams for 5d 2handlebodies}, we have $(\Sigma, \alpha) = (\mathcal{K}', r)$, where $\mathcal{K} = L_1 \cup L_2$ is a Kirby diagram of $Y$, $\mathcal{K}'$ is a Kirby diagram of $DY$ obtained from $\mathcal{K}$ by adding $0$-framed meridians of $L_2$, and $r$ is a banded unlink parallel to these $0$-framed meridians. To construct $\beta = f(DC)$, draw a banded unlink $b$ representing $f(r)$; see \autoref{fig: Wu manifold} for the Wu manifold. Then $(\mathcal{K}', r, b)$ gives a banded unlink diagram of the Heegaard diagram $(\Sigma, \alpha, \beta)$, and we simply write $(\Sigma, \alpha, \beta) = (\mathcal{K}', r, b)$.

For example, when $f$ is the identity map, the banded unlink $b$ is parallel to the red circles $r$, which corresponds to attaching $5$-dimensional $3$-handles along the belt spheres of the $2$-handles of $X$. See \autoref{fig: S3 bundle over a surface} for the case of an $S^3$-bundle over a surface.
\end{proposition}

Barden \cite{barden1965simply} classified simply connected, closed, orientable, smooth $5$-manifolds. The key idea is that every simply connected $5$-manifold can be constructed by gluing two copies of a $5$-dimensional $2$-handlebody without $1$-handles (a boundary connected sum of some copies of $B^3$-bundles over $S^2$), where the gluing map defined on the boundary is realized by an automorphism of the second homology group of the boundary.

There are two possible $B^3$-bundles over $S^2$; the trivial bundle $S^2\times B^3$ and the non-trivial bundle $S^2\tilde{\times}B^3$ since $\pi_1(SO(3))\cong\mathbb{Z}_2$. We note that $\partial(S^2\times B^3)=S^2\times S^2$ and $\partial(S^2\tilde{\times}B^3)=S^2\tilde{\times}S^2\cong \mathbb{C}P^2\#\overline{\mathbb{C}P^2}$; see \autoref{fig: nontrivial s2bundle over s2 with a fiber}. Let $\{a,b\}=\{S^2\times \{y_0\},\{x_0\}\times S^2\}$ be the canonical generators of $H_2(S^2\times S^2)$ and $\{c,d\}=\{\mathbb{C}P^1,\overline{\mathbb{C}P^1}\}$ be the generators of $\mathbb{C}P^2\#\overline{\mathbb{C}P^2}$. We also note that $\partial((S^2\times B^3)\natural(S^2\times B^3))=(S^2\times S^2)\#(S^2\times S^2)$ and $\partial((S^2\tilde{\times} B^3)\natural(S^2\tilde{\times} B^3))\cong (\mathbb{C}P^2\#\overline{\mathbb{C}P^2})\#(\mathbb{C}P^2\#\overline{\mathbb{C}P^2})$. Let $\{a_1,b_1,a_2,b_2\}$ be the canonical generators of $H_2((S^2\times S^2)\#(S^2\times S^2))$, and let $\{c_1,d_1,c_2,d_2\}$ be the canonical generators of $H_2((\mathbb{C}P^2\#\overline{\mathbb{C}P^2})\#(\mathbb{C}P^2\#\overline{\mathbb{C}P^2}))$.

Consider the matrices 
\[A(k)=\begin{pmatrix}
1 & 0 & 0 & -k \\
0 & 1 & 0 & 0 \\
0 & k & 1 & 0 \\
0 & 0 & 0 & 1 
\end{pmatrix}\; \text{and}\; 
B(n)=\begin{pmatrix}
1 & n & -n & 0 \\
n & 1 & 0 & n \\
n & 0 & 1 & n \\
0 & -n & n & 1 
\end{pmatrix}.\]

By work of Wall \cite{wall1964diffeomorphisms}, there exist three diffeomorphisms: 
\begin{enumerate}
    \item $f_k:(S^2\times S^2)\#(S^2\times S^2)\rightarrow (S^2\times S^2)\#(S^2\times S^2)$ such that the induced map $(f_k)_{*}$ on $H_2$ has matrix representation $A(k)$,
    \item $g_j:(\mathbb{C}P^2\#\overline{\mathbb{C}P^2})\#(\mathbb{C}P^2\#\overline{\mathbb{C}P^2})\rightarrow (\mathbb{C}P^2\#\overline{\mathbb{C}P^2})\#(\mathbb{C}P^2\#\overline{\mathbb{C}P^2})$ such that the induced map $(g_j)_{*}$ has matrix representation $B(2^{j-1})$,
    \item $g_{-1}:\mathbb{C}P^2\#\overline{\mathbb{C}P^2}\rightarrow \mathbb{C}P^2\#\overline{\mathbb{C}P^2}$ such that the induced map $(g_j)_{*}$ on $H_2$ has matrix representation $\begin{pmatrix}
       1 & 0 \\
       0 & -1
    \end{pmatrix}.$
\end{enumerate}

\begin{definition}[\cite{barden1965simply}]\label{def: examples of Barden's manifolds}We define some $5$-manifolds.
    \begin{enumerate}
        \item $M_{\infty}=(S^2\times B^3)\cup_{id}(S^2\times B^3)=S^2\times S^3$
        \item $M_k=((S^2\times B^3)\natural(S^2\times B^3))\cup_{f_k}((S^2\times B^3)\natural(S^2\times B^3))$
        \item $X_{-1}=(S^2\tilde{\times} B^3)\cup_{g_{-1}}(S^2\tilde{\times} B^3)=SU(3)/SO(3)$
        \item $X_0=S^5$
        \item $X_{\infty}=(S^2\tilde{\times} B^3)\cup_{id}(S^2\tilde{\times} B^3)=S^2\tilde{\times}S^3$
        \item $X_j=((S^2\tilde{\times} B^3)\natural(S^2\tilde{\times} B^3))\cup_{g_j}((S^2\tilde{\times} B^3)\natural(S^2\tilde{\times} B^3))$
\end{enumerate}
    \end{definition}

\begin{theorem}[\cite{barden1965simply}]\label{thm: Barden's classification}
Every simply connected, closed, orientable, smooth $5$-manifold is diffeomorphic to \[X_j\; \text{or} \;X_j\#M_{k_1}\#\cdots\#M_{k_s},\] where $-1\leq j \leq\infty,1<k_1\leq k_2\leq\cdots\leq k_s$, and either $k_i$ divides $k_{i+1}$ or $k_{i+1}=\infty$. 
\end{theorem}

\begin{example}Heegaard diagrams for some closed $5$-manifolds.
    \begin{enumerate}
        \item Three diagrams in \autoref{fig: stabilizations on a Heegaard diagram} are Heegaard diagrams of $S^5$.
        \item The top right of \autoref{fig: Kirby diagram for surface complement} is a Heegaard diagram of $S^1\times S^4$.
        \item \autoref{fig: S3 bundle over a surface} is a Heegaard diagram of the trivial $S^3$-bundle over a genus $g$ orientable surface $F_g$ when $n$ is even.
        \item \autoref{fig: S3 bundle over a surface} is a Heegaard diagram of the non-trivial $S^3$-bundle over a genus $g$ orientable surface $F_g$ when $n$ is odd.
        \item \autoref{fig: Wu manifold} is a Heegaard diagram of the Wu manifold (denoted $X_{-1}$ in \autoref{def: examples of Barden's manifolds}). The blue curve is the image of the red curve in the right of \autoref{fig: nontrivial s2bundle over s2 with a fiber} under the map $g_{-1}$. Here we can compute $H_2(X_{-1})\cong \mathbb{Z}_2$ from the algebraic intersection number $\alpha\cdot\beta=2$.
    \end{enumerate}
\end{example}

\begin{figure}[ht!]
\labellist
\small\hair 2pt
\pinlabel{$n$} at 5 100
\pinlabel{$g$} at 205 75
\pinlabel{$0$} at 162 15
\endlabellist
\centering
\includegraphics[width=0.42\textwidth]{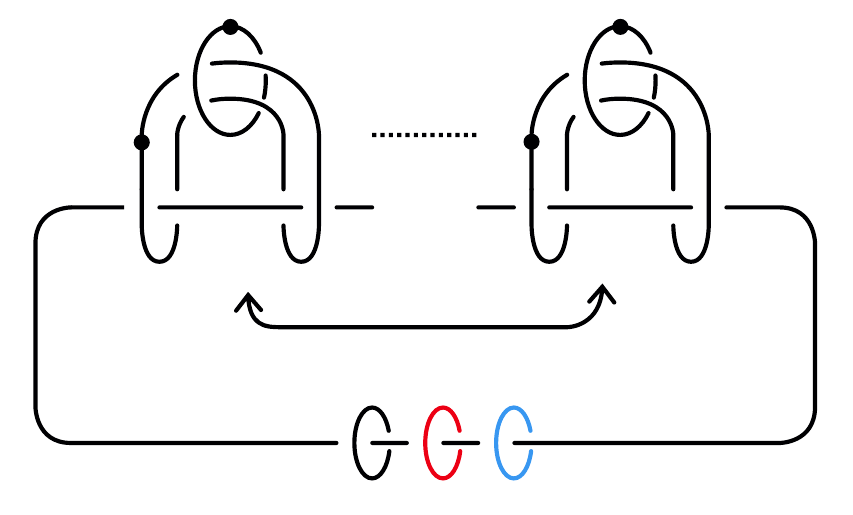}
\caption{An $S^3$-bundle over a orientable genus $g$ surface $F_g$. If we ignore the $0$-framed meridian, the red meridian, and the blue meridian, then this becomes a Kirby diagram of a $B^2$-bundle over $F_g$ with Euler number $n$. If we ignore only the blue meridian, by \autoref{pro:how to draw Heegaard diagrams for 5d 2handlebodies}, the diagram represents a $B^3$-bundle over $F_g$ which is the product of a $B^2$-bundle over $F_g$ and an interval. Therefore the original diagram is a Heegaard diagram of a $S^3$-bundle over $F_g$ by \autoref{pro: how to draw Heegaard diagrams for the twisted double}.}
\label{fig: S3 bundle over a surface}
\end{figure}

\begin{figure}[ht!]
\labellist
\small\hair 2pt
\pinlabel{$1$} at 50 180
\pinlabel{$-1$} at 600 180
\endlabellist
\centering
\includegraphics[width=0.65\textwidth]{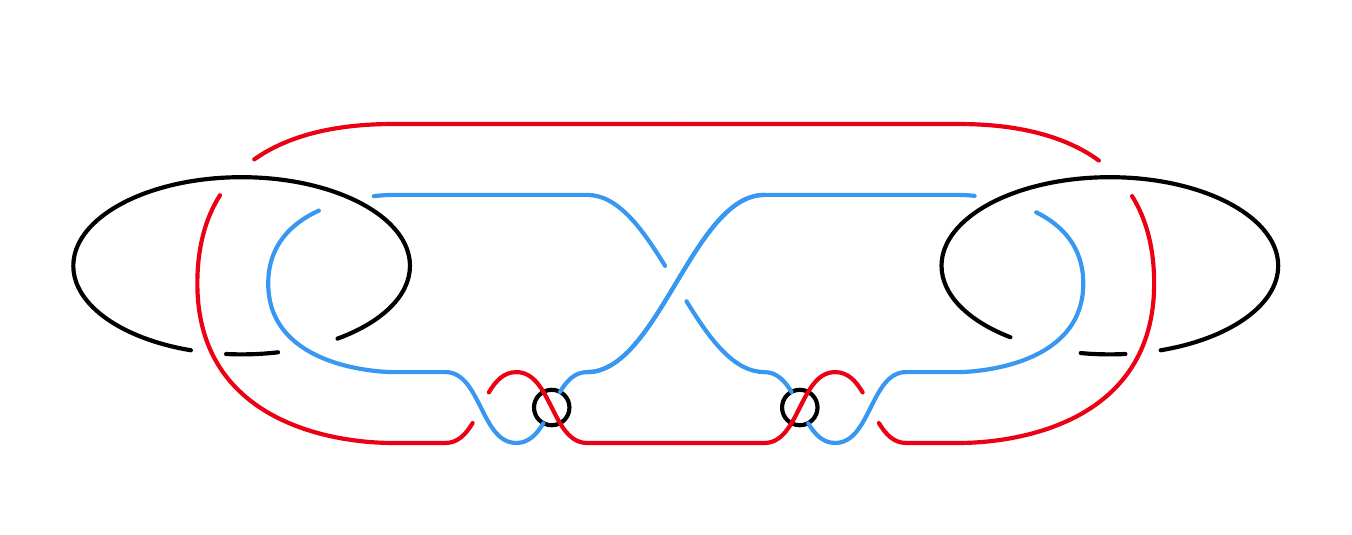}
\caption{Wu manifold.}
\label{fig: Wu manifold}
\end{figure}

\begin{figure}[ht!]
\labellist
\small\hair 2pt
\pinlabel{$0$} at 460 1913
\pinlabel{$1$} at 930 1908
\pinlabel{$0$} at 455 1643
\pinlabel{$0$} at 930 1643

\pinlabel{$1$} at 460 1553
\pinlabel{$1$} at 930 1548
\pinlabel{$0$} at 455 1283
\pinlabel{$0$} at 930 1283

\pinlabel{$1$} at 460 1163
\pinlabel{$1$} at 930 1160
\pinlabel{$0$} at 455 893
\pinlabel{$0$} at 930 893
 
\pinlabel{$1$} at 460 683
\pinlabel{$1$} at 930 680
\pinlabel{$0$} at 460 413
\pinlabel{$0$} at 930 560

\pinlabel{$1$} at 460 326
\pinlabel{$1$} at 930 323
\pinlabel{$0$} at 460 53
\pinlabel{$0$} at 930 53
\endlabellist
\centering
\includegraphics[width=0.72\textwidth]{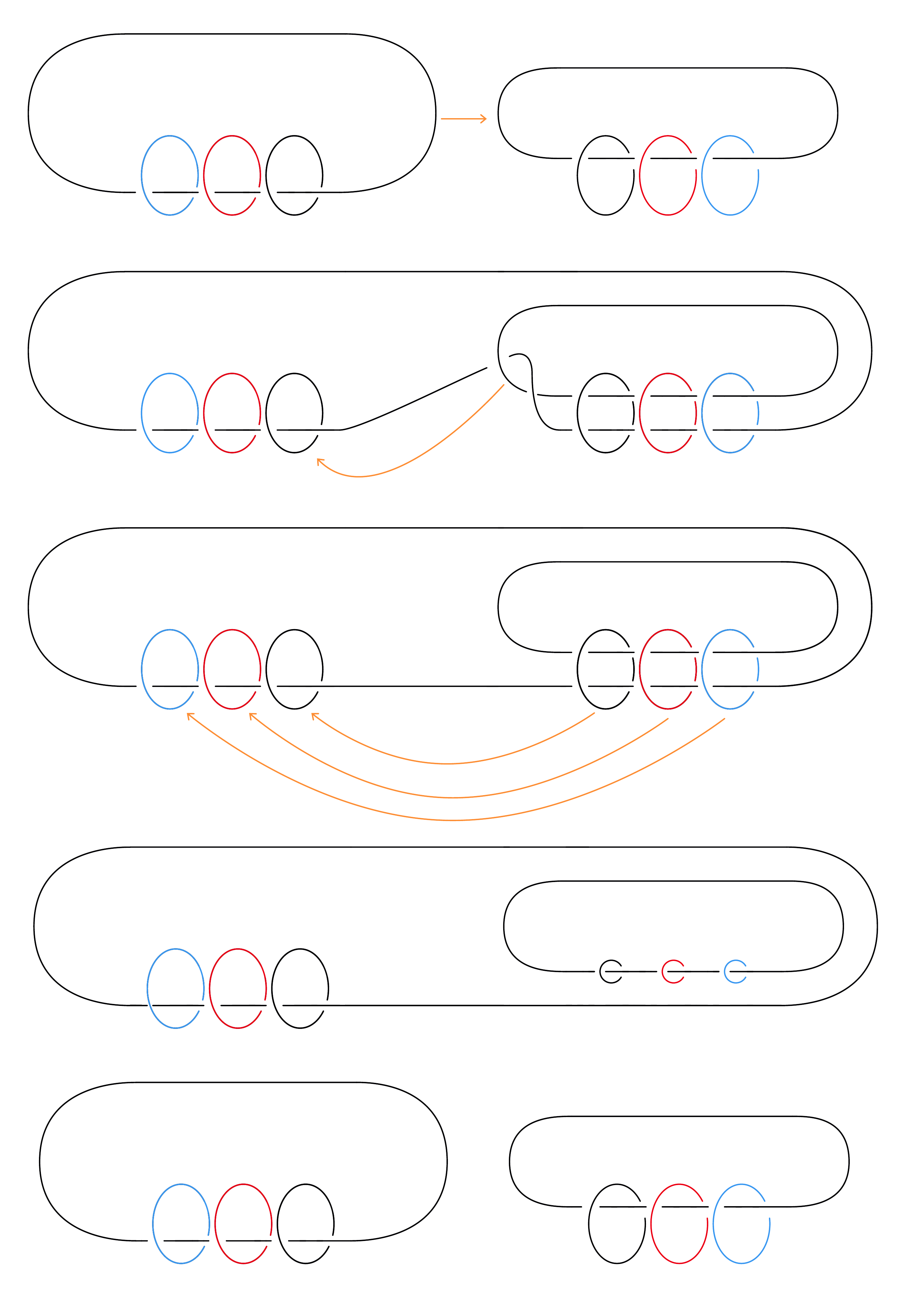}
\caption{$(S^2\times S^3)\#(S^2\tilde{\times} S^3)\cong (S^2\tilde{\times} S^3)\#(S^2\tilde{\times} S^3).$}
\label{fig: two copies of s3 bundle over s2}
\end{figure}

\begin{example}
$(S^2\times S^3)\#(S^2\tilde{\times} S^3)\cong (S^2\tilde{\times} S^3)\#(S^2\tilde{\times} S^3).$ 
\begin{proof}
The bottom of \autoref{fig: two copies of s3 bundle over s2}, which is a Heegaard diagram of $(S^2\tilde{\times} S^3)\#(S^2\tilde{\times} S^3)$, is obtained from the top of the same figure, which is a Heegaard diagram of $(S^2\times S^3)\#(S^2\tilde{\times} S^3)$), by handle slides along the orange guiding arcs. When a red circle (respectively, blue circle) is slid over another red circle (respectively, blue circle), the resulting banded unlink includes a long band and its dual band; see \autoref{fig: Banded unlink diagram for surgery along 3d1h}. These bands can be removed using a band/$2$-handle slide and a band/$2$-handle swim.
\end{proof}
\end{example}

\section{Gluck twists and Heegaard diagrams}\label{sec:Gluck twists}

\begin{definition}\label{def:Gluck twist} Fix $n\geq2$. Let $X$ be a closed, connected, orientable $(n+2)$-manifold. Let $K\subset X$ be an $n$-knot in $X$ with trivial normal bundle, that is, $\nu(K)\cong S^n\times B^2$. The \textit{Gluck twist} of $X$ along $K$ is the $(n+2)$-manifold \[X_K=(X\setminus \operatorname{int}(\nu(K)))\cup_\tau \nu(K),\] where $\partial(\nu(K))\cong S^n\times S^1$ and $\tau:S^n\times S^1\rightarrow S^n\times S^1$ is a diffeomorphism representing the non-trivial element of $\pi_1(SO(n))\cong \mathbb{Z}_2$.
\end{definition}

\begin{definition}\label{def: 5d cobordism gluck twist}
Fix $n\geq2$. Let $X$ be a closed, connected, orientable $(n+2)$-manifold. Let $K\subset X$ be an $n$-knot in $X$ with a trivial normal bundle, that is, $\nu(K)\cong S^n\times B^2$. Let $m_K$ be a meridian of $K$, identified with $\{x_0\}\times S^1\subset S^n\times S^1\cong\partial(\nu(K))$. We define \[W_{X,K}=(X\times [0,1])\cup_{m_K\times\{1\}}(B^2\times B^{n+1})\cup_{K\times\{1\}}(B^{n+1}\times B^2)\] to be the $(n+3)$-manifold obtained from $X\times [0,1]$ by attaching a $2$-handle along $m_K\times\{1\}$ with the non-trivial framing and a $(n+1)$-handle along $K\times\{1\}$.
\end{definition}

\begin{remark} There are two possible framings of the attaching sphere of an $(n+3)$-dimensional $2$-handle since $\pi_1(SO(n+1))\cong \mathbb{Z}_2$, and a unique framing of the attaching sphere of a $(n+3)$-dimensional $(n+1)$-handle since $\pi_n(SO(2))=1$.
\end{remark}

\begin{theorem}\label{thm: cobordism from X to XK}
$W_{X,K}$ is an $(n+3)$-dimensional cobordism from $X$ to $X_K$, i.e., $\partial_{-} (W_{X,K})\cong X$ and $\partial_{+} (W_{X,K})\cong X_K$.
\end{theorem}

\begin{proof}
Let $\phi:S^n\times B^2\hookrightarrow X$ be a framing of $K$ such that $\phi(S^2\times\{(0,0)\})=K$ and $\phi(S^n\times B^2)=\nu(K)$. Note that $\phi$ is unique up to isotopy because $\pi_n(SO(2))$ is trivial. We can then write $X_K$ as: \[X_K=(X\setminus \operatorname{int}(\phi(S^n\times B^2)))\cup_\tau \nu(K).\] Since $(B^{n+1}\times S^1) \setminus (B^{n+1}\times S^1)=\emptyset$, we can rewrite $X_K$ as:\[X_K=((X\setminus \operatorname{int}(\phi(S^n\times B^2)))\cup_{\phi|_{S^n\times S^1}} (B^{n+1}\times S^1))\setminus (B^{n+1}\times S^1)\cup_\tau \nu(K).\] Here, we can consider \[Y=(X\setminus \operatorname{int}(\phi(S^n\times B^2)))\cup_{\phi|_{S^n\times S^1}} (B^{n+1}\times S^1)\] as the result of $n$-surgery on $X$ along $K$, which corresponds to attaching an $(n+1)$-handle along $K$. Next, view $X_K$ as $1$-surgery on $Y$ along $\phi(\{pt\}\times S^1)$, which corresponds to attaching a $2$-handle along the meridian $m_K$ of $K$ with non-trivial framing. We can rearrange the handle attachment so that the $2$-handle is attached first, followed by the $(n+1)$-handle. Therefore, $W_{X,K}$ is a cobordism from $X$ to $X_K$ such that the bottom boundary $\partial_{-} (W_{X,K})=X\times\{0\}\cong X$ and the top boundary $\partial_{+} (W_{X,K})\cong X_K$.
\end{proof}

\begin{corollary}\label{thm: Heegaard diagrams for Gluck twists}
Let $(\Sigma,\alpha,\beta)=(X\#(S^n\tilde{\times}S^2),F,K\#F)$ be a triple, where $K\subset X$ is an $n$-knot in an $(n+2)$-manifold $X$ with trivial normal bundle, $F$ is a fiber of $S^n\tilde{\times}S^2$, and $(X\#S^n\tilde{\times}S^2, K\#F)=(X,K)\#(S^n\tilde{\times}S^2,F)$ is the connected sum of pairs. Then $W_{X,K}$ is diffeomorphic to the manifold obtained from $\Sigma\times[-1,1]$ by attaching an $(n+1)$-handle along $\alpha\times\{-1\}$ and an $(n+1)$-handle along $\beta\times\{1\}$.
\end{corollary}

\begin{proof}
Let $W_{X,K}=(X\times [0,1])\cup_{m_k\times\{1\}}(B^2\times B^{n+1})\cup_{K\times\{1\}}(B^{n+1}\times B^2)$ in \autoref{def: 5d cobordism gluck twist}. It suffices to show that $\partial_{+}(X\times[0,1]\cup_{m_K\times\{1\}}(B^2\times B^{n+1}))$ is diffeomorphic to $X\#(S^n\tilde{\times}S^2)$, the belt sphere of the $2$-handle is $F$ in $S^n\tilde{\times}S^2\subset X\#(S^n\tilde{\times}S^2)$, and the attaching sphere of the $(n+1)$-handle is $K\#F$ in $X\#(S^2\tilde{\times}S^2)$. By the construction of $W_{X,K}$, $\partial_{+}(X\times[0,1]\cup_{m_K\times\{1\}}(B^2\times B^{n+1}))$ is diffeomorphic to the surgery on $X\times\{1\}$ along the meridian $m_K\times\{1\}$ with non-trivial framing. Since $m_K\times\{1\}$ is null-homologous, the surgery is diffeomorphic to $X\#(S^n\tilde{\times}S^2)$, so the belt sphere of the $2$-handle is a fiber $F$ of $S^n\tilde{\times}S^2\subset X\#(S^2\tilde{\times}S^2)$. Clearly the attaching sphere $K\times\{1\}$ of the $(n+1)$-handle is embedded in $\partial_{+}(X\times[0,1]\cup_{m_K\times\{1\}}(B^2\times B^3))\cong X\#(S^2\tilde{\times}S^2)$ and is isotopic to $K\#F$ in $X\#(S^2\tilde{\times}S^2)$.
\end{proof}

\begin{remark}\label{rem: Heegaard diagram of Gluck twist}
The triple $(\Sigma,\alpha,\beta)=(X\#(S^n\tilde{\times}S^2),F,K\#F)$ is a Heegaard diagram of $W_{X,K}$ when $K\subset X$ is a $2$-knot with trivial normal bundle in a $4$-manifold $X$, that is, $W_{X,K}\cong M_\alpha\cup_\Sigma M_\beta$. Therefore, $\Sigma(\alpha)\cong X$ and $\Sigma(\beta)\cong X_K$.
\end{remark}

\begin{remark}
The $(n+2)$-manifolds $X$ and $X_K$ may not be 
diffeomorphic but they become diffeomorphic after connected summing with $S^n\tilde{\times}S^2$.
\end{remark}

\begin{corollary}
The $(n+2)$-manifolds $X\#(S^n\tilde{\times}S^2)$ and $X_K\#(S^n\tilde{\times}S^2)$ are diffeomorphic.
\end{corollary}

\begin{proof}
Let $\Sigma$ be the middle $(n+2)$-manifold of the cobordism $W_{X,K}$ from $X$ to $X_K$, which is the surgery on $X$ along a meridian $m_K$ of $K$. If we turn the handle decomposition of $W_{X,K}$ upside down, we can view $\Sigma$ as the result of performing surgery on $X_K$ along $m_{K'}$, where $K'\subset X_K$ is the canonical $n$-knot along which the Gluck twist of $X_K$ is $X$, and $m_{K'}$ is the meridian of $K'$. We note that $m_K$ and $m_{K'}$ bound meridian disks in $X$ and $X_K$, respectively. Therefore, $X\#(S^n\tilde{\times}S^2)\cong\Sigma\cong X_K\#(S^n\tilde{\times}S^2)$.
\end{proof}

\begin{remark}We focus on the case where $X$ is a $4$-manifold and $K\subset X$ is a $2$-knot with trivial normal bundle.
    \begin{enumerate}
        \item If $K$ is unknotted, then $X_K$ is diffeomorphic to $X$.
        \item If $K$ is null-homotopic in $X$, then $X$ and $X_K$ are homotopy equivalent by \cite{gluck1962embedding}.
        \item If $X$ is a simply connected $4$-manifold and $K$ is null-homotopic in $X$, then $X_K$ is homeomorphic to $X$ by Freedman \cite{freedman1982topology}.
        \item Gluck twists of $S^4$ along non-trivial $2$-knots may be potential counterexamples to the smooth $4$-dimensional Poincaré conjecture. Some families of $2$-knots in $S^4$ have Gluck twists that are known to be diffeomorphic to the standard $S^4$ \cite{gluck1962embedding,Gordon1976,melvin1977blowing,pao1978non,litherland1979deforming,nash2012gluck,naylor2022gluck,gabai2023doubles}.
        \item If $K$ is not null-homotopic, the diffeomorphism type may change. For example, the Gluck twist of $S^2\times S^2$ along $\{x_0\}\times S^2$ is diffeomorphic to $S^2\tilde{\times}S^2$.
    \end{enumerate}   
\end{remark}

\begin{figure}[ht!]
\labellist
\small\hair 2pt
\pinlabel{$1$} at 450 285
\pinlabel{$0$} at 600 285

\pinlabel{$1$} at 700 285
\pinlabel{$0$} at 855 285

\pinlabel{$1$} at 1360 285
\pinlabel{$0$} at 1515 285
\pinlabel{\textcolor[RGB]{56,151,241}{$0$}} at 1180 340
\pinlabel{\textcolor[RGB]{56,151,241}{$0$}} at 1180 200
\endlabellist
\centering
\includegraphics[width=1\textwidth]{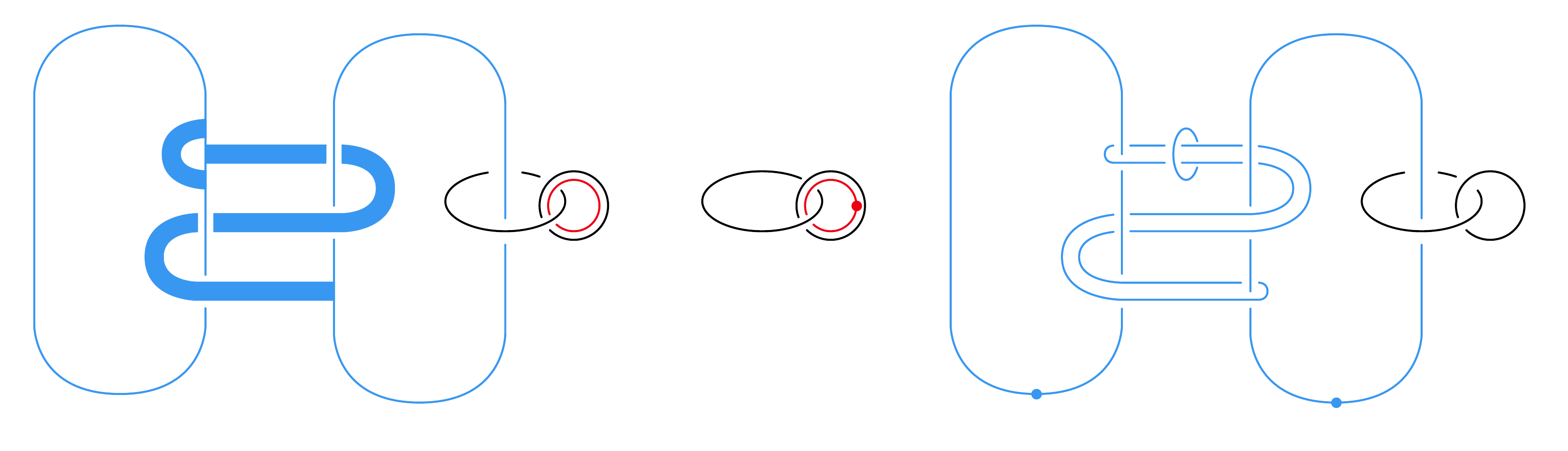}
\caption{\textbf{Left}: A Heegaard diagram $(\Sigma,\alpha,\beta)$ of a cobordism $W_{S^4,K}$ from $S^4$ to $S^4_K$. Here $\Sigma=S^2\tilde{\times}S^2$, red circle represents a fiber $F$ of $S^2\tilde{\times} S^2$, and blue banded unlink represents a connected sum of spun trefoil $K$ and $F$.    \textbf{Middle}: A Kirby diagram of $\Sigma(\alpha)=S^4$. \textbf{Right}: A Kirby diagram of $\Sigma(\beta)=S^4_K$.}
\label{fig: Gluck twists along a spun trefoil}
\end{figure}
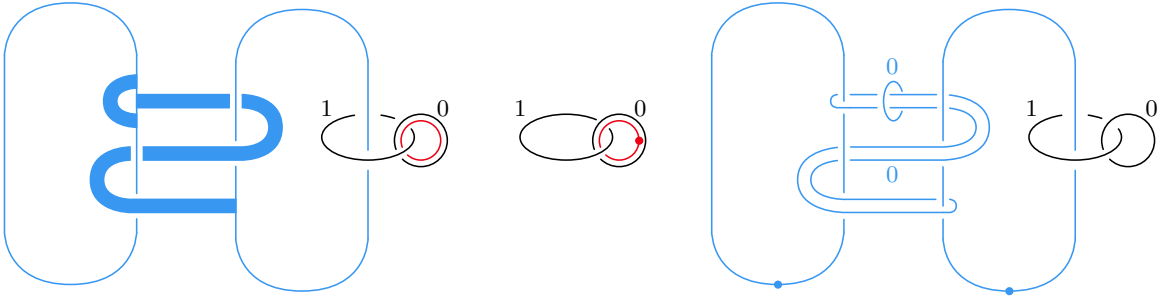

\begin{example}
The left of \autoref{fig: Gluck twists along a spun trefoil} shows a Heegaard diagram of a cobordism $W_{S^4,K}$ from $S^4$ to the Gluck twist $S^4_K$ of $S^4$ along a spun trefoil $K$. By \cite{gluck1962embedding}, the Gluck twist of $S^4$ along any spun knot is diffeomorphic to $S^4$, so $S^4_K\cong S^4$. We can also verify that the right side of \autoref{fig: Gluck twists along a spun trefoil} represents $S^4$ by performing a sequence of Kirby moves.

Another way to show that $S^4_K\cong S^4$ is to observe that $K\#F$ is isotopic to $F$; see \autoref{fig: connected sum of spun trefoil and F is isotopic to F}. The key idea is that $K\# F$ arises from surgery along a $3$-dimensional $1$-handle whose core is $c_1$, as shown in the top left of \autoref{fig: connected sum of spun trefoil and F is isotopic to F}. We can isotope the core $c_1$ to $c_3$, shown in the bottom left of the same figure. This isotopy extends to one between the corresponding $1$-handles $h_1$ and $h_3$, so the results of surgery along $h_1$ and $h_3$ are isotopic. In particular, surgery along $h_3$ yields the fiber $F$.

This strategy applies to any ribbon knot $R \subset S^4$, implying that $S^4_R \cong S^4$. Note that every spun knot is ribbon. Hughes, Kim, and Miller \cite{hughes2020isotopies} showed that for any ribbon knot $R \subset S^4$, the connected sum $R \# \mathbb{C}P^1 \subset \mathbb{C}P^2$ is isotopic to $\mathbb{C}P^1 \subset \mathbb{C}P^2$ via a sequence of moves involving long bands and their duals, as in \autoref{fig: Banded unlink diagram for surgery along 3d1h}. These moves correspond to isotopies of the cores of $3$-dimensional $1$-handles. This result implies $S^4_R \cong S^4$ by Melvin's theorem, which states that for every $2$-knot $K\subset S^4$, $S^4_K\cong S^4$ if and only if $(\mathbb{C}P^2,K\#\mathbb{C}P^1)\cong(\mathbb{C}P^2,\mathbb{C}P^1)$ \cite{melvin1977blowing}.

An advantage of using the core of a $3$-dimensional $1$-handle (rather than a long band and its dual band, as in \autoref{fig: Banded unlink diagram for surgery along 3d1h}) is that a homotopy of the core directly induces an isotopy of the $1$-handle since homotopy of $1$-manifolds implies isotopy in dimension $4$. Such isotopies are also easier to visualize in a Kirby diagram.
\end{example}

\begin{figure}[ht!]
\labellist
\small\hair 2pt
\pinlabel{$1$} at 435 730
\pinlabel{$0$} at 600 750
\pinlabel{\textcolor[RGB]{253,141,51}{$c_1$}} at 265 640
 
\pinlabel{$1$} at 1115 730
\pinlabel{$0$} at 1280 750
\pinlabel{\textcolor[RGB]{253,141,51}{$c_2$}} at 950 640

\pinlabel{$1$} at 435 270
\pinlabel{$0$} at 600 290
\pinlabel{\textcolor[RGB]{253,141,51}{$c_3$}} at 265 250

\pinlabel{$1$} at 1115 270
\pinlabel{$0$} at 1280 290
\endlabellist
\centering
\includegraphics[width=.85\textwidth]{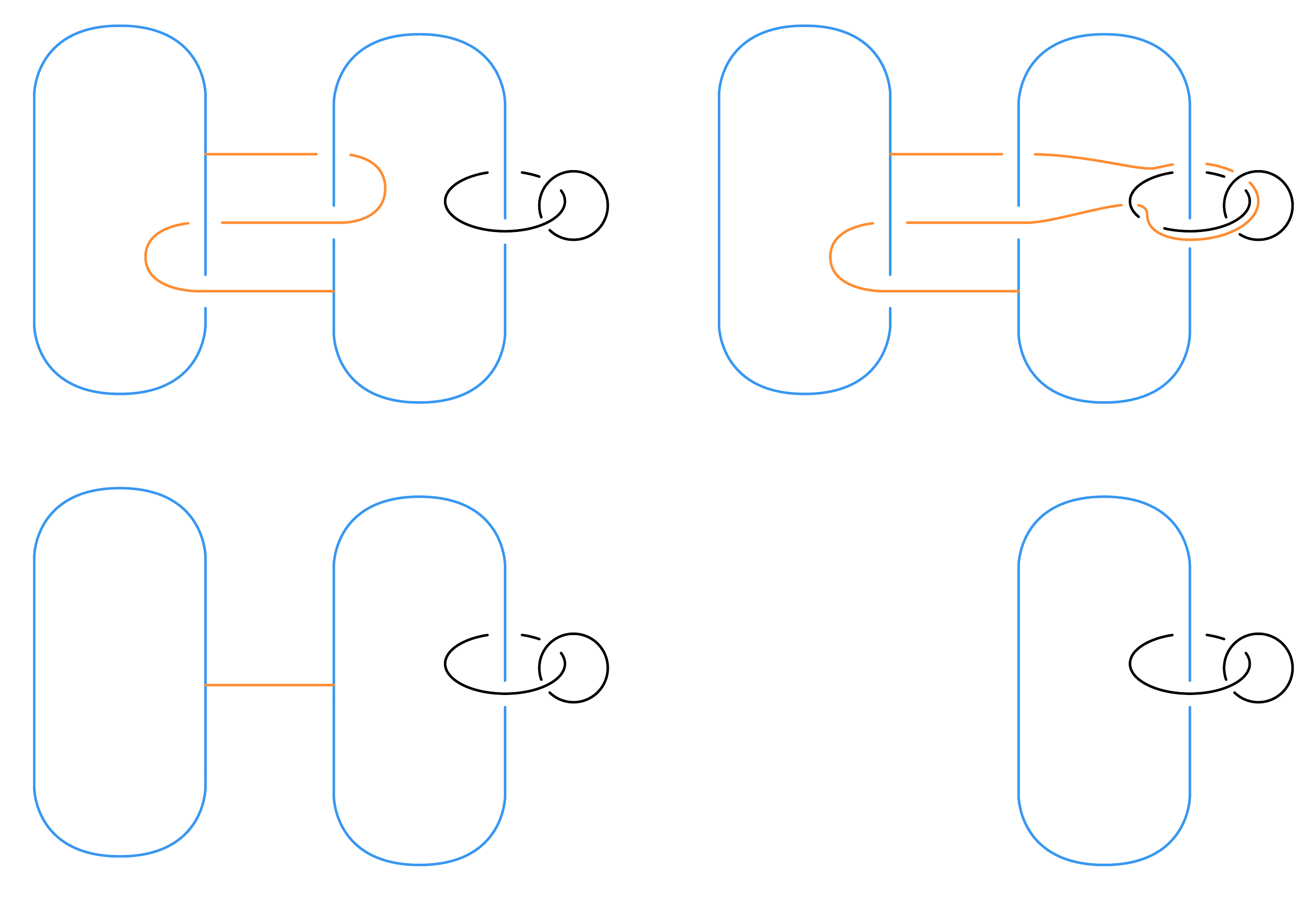}
\caption{The arc  $c_2$ is obtained by sliding $c_1$ over the $1$-framed unknot. Note that the Kirby diagram is drawn on the boundary $S^3 = \partial B^4$, and $c_2$ lies in $S^3$. Push $c_2$ into the interior of $B^4$, perform a small isotopy there, then pull it back into $S^3$ and apply another small isotopy to obtain $c_3$. The curve $c_3$ can be cancelled with the blue component on the left side, yielding the diagram shown at the bottom right. This implies that $S^4_K\cong S^4$.}
\label{fig: connected sum of spun trefoil and F is isotopic to F}
\end{figure}
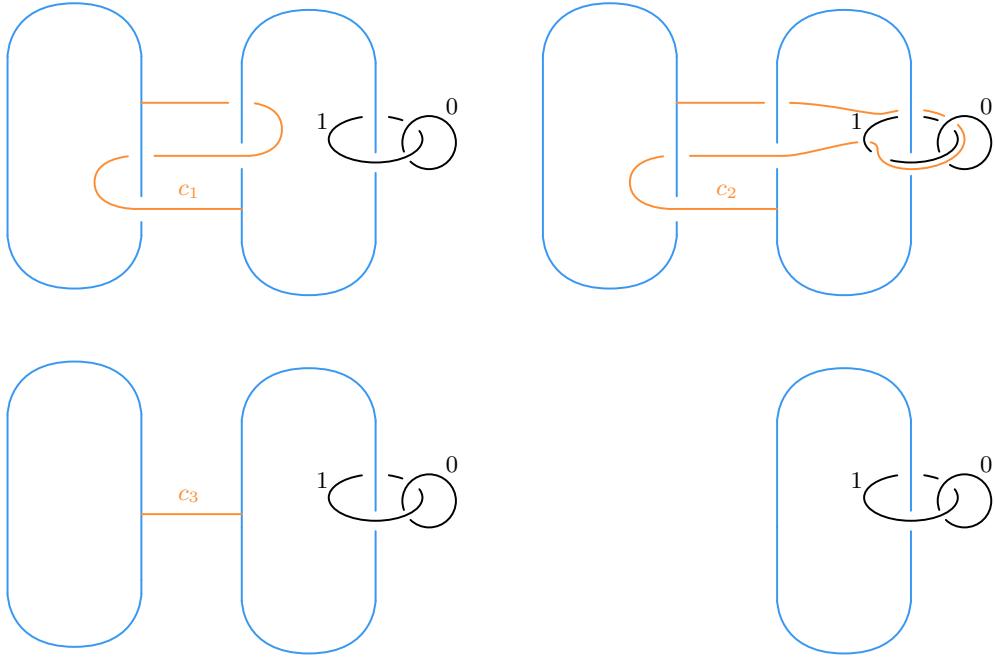

\begin{example}
Let $\gamma\subset X$ be a simple closed curve in a $4$-manifold $X$. Let $X(\gamma)$ be the surgery of $X$ along $\gamma$ with trivial framing, and let $X(\gamma)'$ be the surgery of $X$ along $\gamma$ with non-trivial framing. Then $X(\gamma)'$ is the Gluck twist of $X$ along a meridian sphere $S_\gamma$ of $\gamma$, where $S=\{x_0\}\times S^2\subset S^1\times S^2=\partial(\gamma\times B^3)=\nu(\gamma)$; see \autoref{fig: cobordism from spun lens space to twisted spun lens space}. For example, if $X=S^4$ and $\gamma$ is any simple closed curve in $X$, then $X(\gamma)'=S^2\tilde{\times}S^2$ is the Gluck twist of $X(\gamma)=S^2\times S^2$ along the fiber $\{x_0\}\times S^2\subset S^2\times S^2$. 
\end{example}

\HeegaardGluck

\begin{proof}
We prove $(1)\Rightarrow(2)\Rightarrow(4)\Rightarrow(1)$ and $(2)\Leftrightarrow(3)$.

$(1)\Rightarrow(2)$. Assume $S^4_K$ is diffeomorphic to $S^4$. By \autoref{thm: cobordism from X to XK} and \autoref{thm: Heegaard diagrams for Gluck twists}, we have $W_{S^4,K}\cong M_\alpha\cup_\Sigma M_\beta$ with $\Sigma(\alpha)\cong S^4$ and $\Sigma(\beta)\cong S^4_K$; see also \autoref{rem: Heegaard diagram of Gluck twist}. Since $\Sigma(\beta)\cong S^4_K\cong S^4$, we can construct a closed $5$-manifold $\widehat{M_\alpha}\cup_\Sigma \widehat{M_\beta}$ from the cobordism $M_\alpha\cup_\Sigma M_\beta$ by gluing two $5$-balls along its boundary components. Its second homology group is $H_2(\widehat{M_\alpha}\cup_\Sigma \widehat{M_\beta})\cong \mathbb{Z}$, so by the classification of simply-connected $5$-manifolds \cite{barden1965simply} (see \autoref{thm: Barden's classification}), it is diffeomorphic to either $S^2\times S^3$ or $S^2\tilde{\times}S^3$. Since the middle level $\Sigma$ of $\widehat{M_\alpha}\cup_\Sigma \widehat{M_\beta}$ is $S^2\tilde{\times}S^2$, we conclude that $\widehat{M_\alpha}\cup_\Sigma \widehat{M_\beta}$ is diffeomorphic to $S^2\Tilde{\times}S^3$. Therefore $M_\alpha\cup_\Sigma M_\beta$ is diffeomorphic to a twice-punctured $S^2\Tilde{\times}S^3$.

$(2)\Rightarrow (4)$. Consider again $\widehat{M_\alpha}\cup_\Sigma \widehat{M_\beta}\cong S^2\tilde{\times}S^3$. The manifold $\widehat{M_\beta}\cong S^2\tilde{\times}B^3$ is obtained from $B^5$ by attaching a $2$-handle along the unknot with the non-trivial framing. Hence, the pair $(S^2\tilde{\times}S^2,K\#F)=(\Sigma,\beta)$ can be considered as a pair of the boundary of $\widehat{M_\beta}$ and the belt sphere of the $2$-handle. Thus, $(S^2\Tilde{\times}S^2,K\#F)$ is diffeomorphic to $(S^2\Tilde{\times}S^2,F)$.

$(4)\Rightarrow (1)$. The Gluck twist $S^4_K$ is diffeomorphic to $\Sigma(\beta)$, which is the result of surgery on $S^2\tilde{\times}S^2$ along $K\#F$. Since $(S^2\Tilde{\times}S^2,K\#F)\cong(S^2\Tilde{\times}S^2,F)$ and the surgery on $S^2\tilde{\times}S^2$ along $F$ is diffeomorphic to $S^4$, the Gluck twist $S^4_K$ is diffeomorphic to $S^4$.

$(2)\Leftrightarrow(3)$. The manifold $S^2\tilde{\times}S^3$ is obtained by gluing two copies of $S^2\tilde{\times}B^3$ along the identity map on their common boundary, which is $S^2\tilde{\times}S^2$. Each copy $S^2\tilde{\times}B^3$ is obtained from $B^5$ by attaching a $2$-handle along the unknot with the non-trivial framing. The belt sphere of the $2$-handle is a fiber $F$ of $S^2\tilde{\times}S^2$, so the triple $(S^2\tilde{\times}S^2,F,F)$ is a Heegaard diagram not only for $S^2\tilde{\times}S^3$ but also for a twice-punctured $S^2\tilde{\times}S^3$ since the latter is obtained by removing two $5$-balls. By \autoref{thm: Heegaard diagrams representing diffeomorphic 5-manifolds}, statements $(2)$ and $(3)$ are equivalent.
\end{proof}

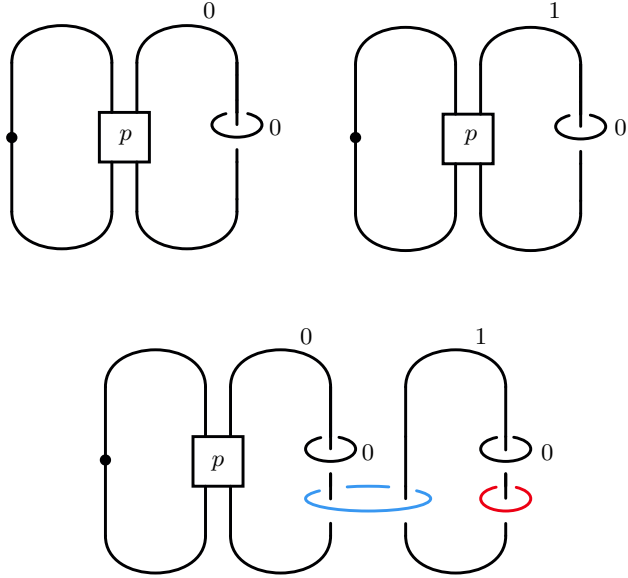
\begin{figure}[ht!]
\labellist
\small\hair 2pt

\pinlabel{$p$} at 103 333
\pinlabel{$0$} at 205 340
\pinlabel{$0$} at 160 420

\pinlabel{$p$} at 338 333
\pinlabel{$0$} at 440 340
\pinlabel{$1$} at 395 420

\pinlabel{$p$} at 166 110
\pinlabel{$0$} at 268 117
\pinlabel{$0$} at 226 197
\pinlabel{$0$} at 390 117
\pinlabel{$1$} at 345 197
\endlabellist
\centering
\includegraphics[width=0.55\textwidth]{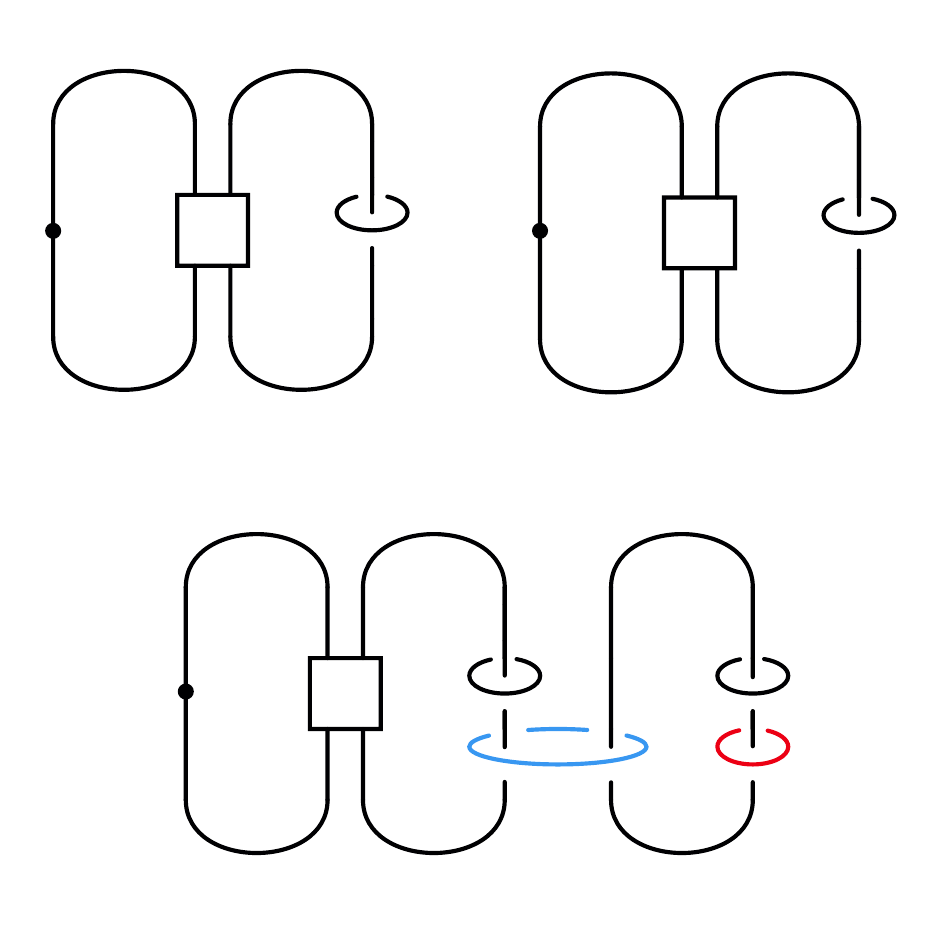}
\caption{\textbf{Top left}: Surgery $M(\gamma_p)$ on $S^1\times S^3$ along a circle $\gamma_p$ representing $p\in \mathbb{Z}\cong\pi_1(S^1\times S^3)$ with trivial framing; see \autoref{pro: Kirby diagram for 1-surgeries} and \autoref{fig: Kirby diagram for 1-surgery }. \textbf{Top right}: Surgery $M(\gamma_p)'$ on $S^1\times S^3$ along a circle $\gamma_p$ representing $p\in \mathbb{Z}\cong \pi_1(S^1\times S^3)$ with non-trivial framing. \textbf{Bottom}: A Heegaard diagram of a cobordism $W_{X,K}$ from $X=M(\gamma_p)$ to $X_K=M(\gamma_p)'$, where $M(\gamma_p)'$ is the Gluck twist of $M(\gamma_p)$ along a meridian sphere $S$ of $\gamma_p$. In this diagram, the blue circle and the red circle represent $S\#F$ and $F$, respectively, where $F$ is a fiber of $S^2\tilde{\times}S^2$. Replacing the red circle with a dotted circle produces the diagram in the top left after obvious Kirby moves, and similarly, replacing the blue circle with a dotted circle yields the diagram in the top right.}
\label{fig: cobordism from spun lens space to twisted spun lens space}
\end{figure}

\bibliographystyle{alpha}
\bibliography{refs} 
\end{document}